\documentclass[12pt]{amsart}
\usepackage{tikz-cd}
\usepackage{enumitem}
\setlist[description]{leftmargin=\parindent,labelindent=\parindent}
\usetikzlibrary{matrix,arrows,decorations.pathmorphing}
\usepackage{amssymb}
\usepackage{amsmath,amscd}
\usepackage{mathrsfs}
\usepackage{url}
\usepackage{cite}
\usepackage{fullpage}
\usepackage{hyperref}
\usepackage{verbatim}
\usepackage{comment}
\usepackage{fancyvrb}
\usepackage{fvextra}
\usepackage{caption}

\newtheorem{thm}{Theorem}[section]
\newtheorem{prop}[thm]{Proposition}
\newtheorem{lem}[thm]{Lemma}
\newtheorem{cor}[thm]{Corollary}

\theoremstyle{definition}
\newtheorem{definition}[thm]{Definition}

\newtheorem{rem}[thm]{Remark}
\newtheorem{question}[thm]{Question}

\numberwithin{equation}{section}

\renewcommand{\L}{\mathcal{L}}
\newcommand{\U}{\mathcal{U}}
\newcommand{\X}{\mathcal{X}}
\newcommand{\Y}{\mathcal{Y}}

\newcommand{\V}{\mathcal{V}}

\newcommand{\B}{\mathcal{B}}

\newcommand{\G}{\mathcal{G}}
\newcommand{\zz}{\mathbb{Z}}
\newcommand{\cc}{\mathbb{C}}
\newcommand{\qq}{\mathbb{Q}}

\newcommand{\I}{\mathcal{I}}
\newcommand{\C}{\mathcal{C}}
\newcommand{\W}{\mathcal{W}}
\newcommand{\M}{\mathcal{M}}
\newcommand{\p}{\mathbb{P}}
\newcommand{\pp}{\mathbb{P}}
\newcommand{\A}{\mathbb{A}}
\renewcommand{\H}{\mathcal{H}}
\newcommand{\F}{\mathcal{F}}
\renewcommand{\P}{\mathcal{P}}
\DeclareMathOperator{\cl}{cl}
\newcommand{\E}{\mathcal{E}}
\renewcommand{\O}{\mathcal{O}}

\newcommand{\Mg}{\mathcal{M}_g}
\DeclareMathOperator{\ct}{ct}
\DeclareMathOperator{\rt}{rt}
\renewcommand{\tilde}{\widetilde}

\DeclareMathOperator{\Hyp}{Hyp}

\DeclareMathOperator{\GL}{GL}
\DeclareMathOperator{\SL}{SL}

\DeclareMathOperator{\Supp}{Supp}

\DeclareMathOperator{\Sym}{Sym}

\DeclareMathOperator{\BSL}{BSL}
\DeclareMathOperator{\BPGL}{BPGL}
\DeclareMathOperator{\BGL}{BGL}

\renewcommand{\gg}{\mathbb{G}}

\newcommand{\dn}{\mathrm{dn}}

\usepackage{wrapfig}

\newcommand{\Mb}{\overline{\M}}
\newcommand{\Cb}{\overline{\mathcal{C}}}
\newcommand{\hannah}[1]{{\color{teal} ($\spadesuit$ Hannah: #1)}}

\title{On the Chow and cohomology rings of moduli spaces of stable curves}
\author{Samir Canning}
\author{Hannah Larson}
\thanks{
S.C. was partially supported by NSF RTG grant DMS-1502651 and by the Swedish Research Council under grant no. 2016-06596 while in residence at Institut Mittag-Leffler in Djursholm, Sweden in fall 2021. 
This research was partially conducted during the period H.L served as a Clay Research Fellow. H.L. was partially supported by the Hertz Foundation and NSF GRFP grant DGE-1656518. }

\begin{document}

\maketitle


\begin{abstract}
In this paper, we ask: for which $(g, n)$ is the rational Chow or cohomology ring of $\Mb_{g,n}$ generated by tautological classes? This question has been fully answered in genus $0$ by Keel (the Chow and cohomology rings are tautological for all $n$ \cite{Keel}) and genus $1$ by Belorousski (the rings are tautological if and only if $n \leq 10$ \cite{Belorousski}). For $g \geq 2$, work of van Zelm \cite{vanZelm} shows the Chow and cohomology rings are not tautological once $2g + n \geq 24$, leaving finitely many open cases. Here, we prove that the Chow and cohomology rings of $\Mb_{g,n}$ are isomorphic and generated by tautological classes for $g = 2$ and $n \leq 9$ and for $3 \leq g \leq 7$ and $2g + n \leq 14$. For such $(g, n)$, this implies that the tautological ring is Gorenstein and 
$\Mb_{g,n}$ has polynomial point count. 
\end{abstract}

\section{Introduction}
In his landmark 1983 paper, Mumford proposed the study of the rational Chow rings of the moduli space of curves $\M_g$ and its Deligne--Mumford compactification $\Mb_g$ in low genus, and began the project by determining $A^*(\Mb_{2})$ \cite{Mumford}. In 1990, Faber determined $A^*(\Mb_{3})$ \cite{FaberI}. Since then, substantial progress has been made for the open moduli spaces, eventually determining $A^*(\M_g)$ for $g \leq 9$ \cite{FaberII,Izadi,PenevVakil,789}. 
The aim of this paper is to bring progress to the more
challenging problem of determining $A^*(\Mb_{g})$.
In particular, for $g \leq 7$, we will prove $A^*(\Mb_{g})$ is generated by tautological classes and the cycle class map is an isomorphism. Thus, the ring is algorithmically determinable, as we explain in Section \ref{gorsec}.

Due to the nature of its boundary, the first step in pursuing results for the compactification $\Mb_g$ is to consider all the moduli spaces $\Mb_{g,n}$ of $n$-pointed stable curves of genus $g$.  Let $R^*(\Mb_{g,n}) \subseteq A^*(\Mb_{g,n})$ be the subring generated by tautological classes (see Definition \ref{tdef}).

\begin{question} \label{mainq}
For which $(g, n)$ do we have $A^*(\Mb_{g,n}) = R^*(\Mb_{g,n})$ ?
\end{question}

Our answers to Question \ref{mainq} will also be used to give answers to its analogue in cohomology, which in turn have several applications.
The tautological ring in cohomology $RH^*(\Mb_{g,n}) \subseteq H^*(\Mb_{g,n})$ is defined as the image of $R^*(\Mb_{g,n})$ under the cycle class map. 
\begin{question} \label{cohq}
For which $(g, n)$ do we have $H^*(\Mb_{g,n}) = RH^*(\Mb_{g,n})$?
\end{question}
\noindent

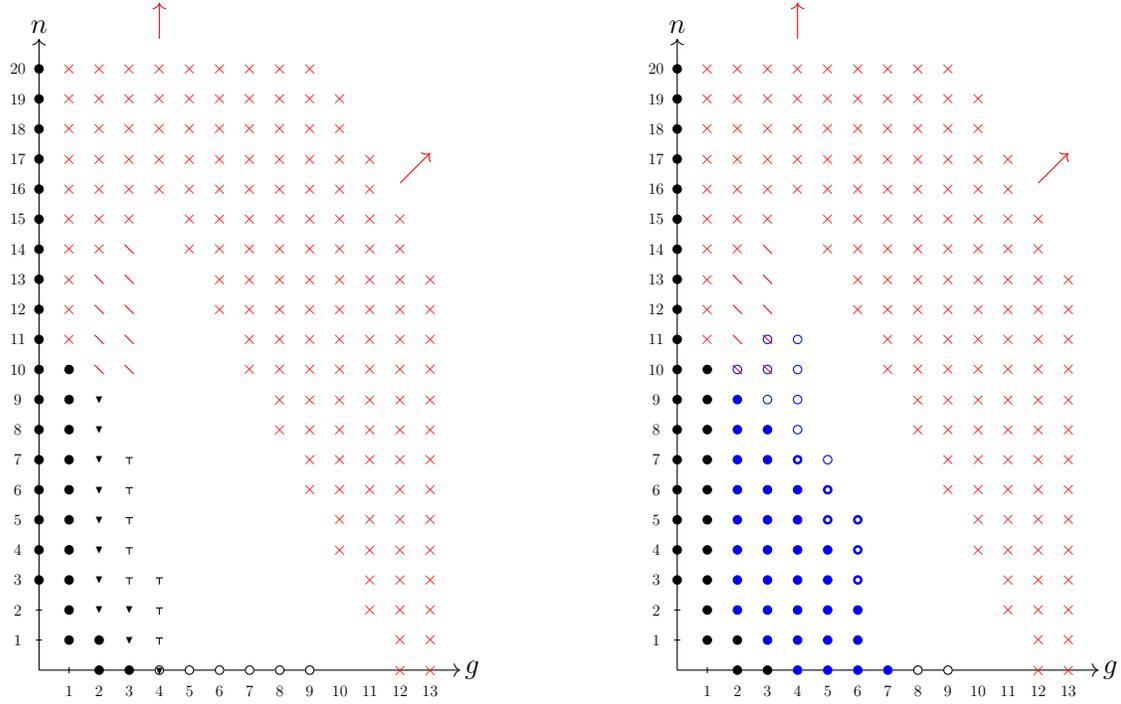
\begin{figure}[h!] 
\vspace{-.2in}
\begin{center}
\begin{tikzpicture}[scale = .4]

\node[scale=.35] at (2, 2) {$\blacktriangledown$};
\node[scale=.35] at (2, 3) {$\blacktriangledown$};
\node[scale=.35] at (2, 4) {$\blacktriangledown$};
\node[scale=.35] at (2, 5) {$\blacktriangledown$};
\node[scale=.35] at (2, 6) {$\blacktriangledown$};
\node[scale=.35] at (2, 7) {$\blacktriangledown$};
\node[scale=.35] at (2, 8) {$\blacktriangledown$};
\node[scale=.35] at (2, 9) {$\blacktriangledown$};

\node[scale=.5] at (1, -.7) {$1$};
\node[scale=.5] at (2, -.7) {$2$};
\node[scale=.5] at (3, -.7) {$3$};
\node[scale=.5] at (4, -.7) {$4$};
\node[scale=.5] at (5, -.7) {$5$};
\node[scale=.5] at (6, -.7) {$6$};
\node[scale=.5] at (7, -.7) {$7$};
\node[scale=.5] at (8, -.7) {$8$};
\node[scale=.5] at (9, -.7) {$9$};
\node[scale=.5] at (10, -.7) {$10$};
\node[scale=.5] at (11, -.7) {$11$};
\node[scale=.5] at (12, -.7) {$12$};
\node[scale=.5] at (13, -.7) {$13$};

\node[scale=.5] at (-.7, 1) {$1$};
\node[scale=.5] at (-.7, 2) {$2$};
\node[scale=.5] at (-.7, 3) {$3$};
\node[scale=.5] at (-.7, 4) {$4$};
\node[scale=.5] at (-.7, 5) {$5$};
\node[scale=.5] at (-.7, 6) {$6$};
\node[scale=.5] at (-.7, 7) {$7$};
\node[scale=.5] at (-.7, 8) {$8$};
\node[scale=.5] at (-.7, 9) {$9$};
\node[scale=.5] at (-.7, 10) {$10$};
\node[scale=.5] at (-.7, 11) {$11$};
\node[scale=.5] at (-.7, 12) {$12$};
\node[scale=.5] at (-.7, 13) {$13$};
\node[scale=.5] at (-.7, 14) {$14$};
\node[scale=.5] at (-.7, 15) {$15$};
\node[scale=.5] at (-.7, 16) {$16$};
\node[scale=.5] at (-.7, 17) {$17$};
\node[scale=.5] at (-.7, 18) {$18$};
\node[scale=.5] at (-.7, 19) {$19$};
\node[scale=.5] at (-.7, 20) {$20$};

\draw[->] (0, 0) -- (14, 0);
\draw[->] (0, 0) -- (0, 21);
\node[scale=.9] at (14.4,0) {$g$};
\node[scale=.9] at (0, 21.4) {$n$};
\draw (-.1, 1) -- (.1, 1);
\draw (-.1, 2) -- (.1, 2);
\draw (1, -.1) -- (1, .1);
\draw (2, -.1) -- (2, .1);
\draw (3, -.1) -- (3, .1);
\draw (4, -.1) -- (4, .1);
\filldraw (0, 3) circle (4pt);
\filldraw (0, 4) circle (4pt);
\filldraw (0, 5) circle (4pt);
\filldraw (0, 6) circle (4pt);
\filldraw (0, 7) circle (4pt);
\filldraw (0, 8) circle (4pt);
\filldraw (0, 9) circle (4pt);
\filldraw (0, 10) circle (4pt);
\filldraw (0, 11) circle (4pt);
\filldraw (0, 12) circle (4pt);
\filldraw (0, 13) circle (4pt);
\filldraw (0, 14) circle (4pt);
\filldraw (0, 15) circle (4pt);
\filldraw (0, 16) circle (4pt);
\filldraw (0, 17) circle (4pt);
\filldraw (0, 18) circle (4pt);
\filldraw (0, 19) circle (4pt);
\filldraw (0, 20) circle (4pt);
\filldraw (1, 1) circle (4pt);
\filldraw (1, 2) circle (4pt);
\filldraw (1, 3) circle (4pt);
\filldraw (1, 4) circle (4pt);
\filldraw (1, 5) circle (4pt);
\filldraw (1, 6) circle (4pt);
\filldraw (1, 7) circle (4pt);
\filldraw (1, 8) circle (4pt);
\filldraw (1, 9) circle (4pt);
\filldraw (1, 10) circle (4pt);
\node[scale = .6, color = red] at (1, 11) {$\times$};
\node[scale = .6, color = red] at (1, 12) {$\times$};
\node[scale = .6, color = red] at (1, 13) {$\times$};
\node[scale = .6, color = red] at (1, 14) {$\times$};
\node[scale = .6, color = red] at (1, 15) {$\times$};
\node[scale = .6, color = red] at (1, 16) {$\times$};
\node[scale = .6, color = red] at (1, 17) {$\times$};
\node[scale = .6, color = red] at (1, 18) {$\times$};
\node[scale = .6, color = red] at (1, 19) {$\times$};
\node[scale = .6, color = red] at (1, 20) {$\times$};

\draw[color=red] (2-.15,10+.15) -- (2+.15, 10-.15);
\draw[color=red] (2-.15,11+.15) -- (2+.15, 11-.15);
\draw[color=red] (2-.15,12+.15) -- (2+.15, 12-.15);
\draw[color=red] (2-.15,13+.15) -- (2+.15, 13-.15);
\draw[color=red] (3-.15,10+.15) -- (3+.15, 10-.15);
\draw[color=red] (3-.15,11+.15) -- (3+.15, 11-.15);
\draw[color=red] (3-.15,12+.15) -- (3+.15, 12-.15);
\draw[color=red] (3-.15,13+.15) -- (3+.15, 13-.15);
\draw[color=red] (3-.15,14+.15) -- (3+.15, 14-.15);

\node[scale = .6, color = red] at (2, 14) {$\times$};
\node[scale = .6, color = red] at (2, 15) {$\times$};
\node[scale = .6, color = red] at (2, 16) {$\times$};
\node[scale = .6, color = red] at (2, 17) {$\times$};
\node[scale = .6, color = red] at (2, 18) {$\times$};
\node[scale = .6, color = red] at (2, 19) {$\times$};
\node[scale = .6, color = red] at (2, 20) {$\times$};

\node[scale = .6, color = red] at (3, 15) {$\times$};
\node[scale = .6, color = red] at (3, 16) {$\times$};
\node[scale = .6, color = red] at (3, 17) {$\times$};
\node[scale = .6, color = red] at (3, 18) {$\times$};
\node[scale = .6, color = red] at (3, 19) {$\times$};
\node[scale = .6, color = red] at (3, 20) {$\times$};

\node[scale = .6, color = red] at (4, 16) {$\times$};
\node[scale = .6, color = red] at (4, 17) {$\times$};
\node[scale = .6, color = red] at (4, 18) {$\times$};
\node[scale = .6, color = red] at (4, 19) {$\times$};
\node[scale = .6, color = red] at (4, 20) {$\times$};

\node[scale = .6, color = red] at (5, 14) {$\times$};
\node[scale = .6, color = red] at (5, 15) {$\times$};
\node[scale = .6, color = red] at (5, 16) {$\times$};
\node[scale = .6, color = red] at (5, 17) {$\times$};
\node[scale = .6, color = red] at (5, 18) {$\times$};
\node[scale = .6, color = red] at (5, 19) {$\times$};
\node[scale = .6, color = red] at (5, 20) {$\times$};

\node[scale = .6, color = red] at (6, 12) {$\times$};
\node[scale = .6, color = red] at (6, 13) {$\times$};
\node[scale = .6, color = red] at (6, 14) {$\times$};
\node[scale = .6, color = red] at (6, 15) {$\times$};
\node[scale = .6, color = red] at (6, 16) {$\times$};
\node[scale = .6, color = red] at (6, 17) {$\times$};
\node[scale = .6, color = red] at (6, 18) {$\times$};
\node[scale = .6, color = red] at (6, 19) {$\times$};
\node[scale = .6, color = red] at (6, 20) {$\times$};

\node[scale = .6, color = red] at (7, 10) {$\times$};
\node[scale = .6, color = red] at (7, 11) {$\times$};
\node[scale = .6, color = red] at (7, 12) {$\times$};
\node[scale = .6, color = red] at (7, 13) {$\times$};
\node[scale = .6, color = red] at (7, 14) {$\times$};
\node[scale = .6, color = red] at (7, 15) {$\times$};
\node[scale = .6, color = red] at (7, 16) {$\times$};
\node[scale = .6, color = red] at (7, 17) {$\times$};
\node[scale = .6, color = red] at (7, 18) {$\times$};
\node[scale = .6, color = red] at (7, 19) {$\times$};
\node[scale = .6, color = red] at (7, 20) {$\times$};

\node[scale = .6, color = red] at (8, 8) {$\times$};
\node[scale = .6, color = red] at (8, 9) {$\times$};
\node[scale = .6, color = red] at (8, 10) {$\times$};
\node[scale = .6, color = red] at (8, 11) {$\times$};
\node[scale = .6, color = red] at (8, 12) {$\times$};
\node[scale = .6, color = red] at (8, 13) {$\times$};
\node[scale = .6, color = red] at (8, 14) {$\times$};
\node[scale = .6, color = red] at (8, 15) {$\times$};
\node[scale = .6, color = red] at (8, 16) {$\times$};
\node[scale = .6, color = red] at (8, 17) {$\times$};
\node[scale = .6, color = red] at (8, 18) {$\times$};
\node[scale = .6, color = red] at (8, 19) {$\times$};
\node[scale = .6, color = red] at (8, 20) {$\times$};

\node[scale = .6, color = red] at (9, 6) {$\times$};
\node[scale = .6, color = red] at (9, 7) {$\times$};
\node[scale = .6, color = red] at (9, 8) {$\times$};
\node[scale = .6, color = red] at (9, 9) {$\times$};
\node[scale = .6, color = red] at (9, 10) {$\times$};
\node[scale = .6, color = red] at (9, 11) {$\times$};
\node[scale = .6, color = red] at (9, 12) {$\times$};
\node[scale = .6, color = red] at (9, 13) {$\times$};
\node[scale = .6, color = red] at (9, 14) {$\times$};
\node[scale = .6, color = red] at (9, 15) {$\times$};
\node[scale = .6, color = red] at (9, 16) {$\times$};
\node[scale = .6, color = red] at (9, 17) {$\times$};
\node[scale = .6, color = red] at (9, 18) {$\times$};
\node[scale = .6, color = red] at (9, 19) {$\times$};
\node[scale = .6, color = red] at (9, 20) {$\times$};

\node[scale = .6, color = red] at (10, 4) {$\times$};
\node[scale = .6, color = red] at (10, 5) {$\times$};
\node[scale = .6, color = red] at (10, 6) {$\times$};
\node[scale = .6, color = red] at (10, 7) {$\times$};
\node[scale = .6, color = red] at (10, 8) {$\times$};
\node[scale = .6, color = red] at (10, 9) {$\times$};
\node[scale = .6, color = red] at (10, 10) {$\times$};
\node[scale = .6, color = red] at (10, 11) {$\times$};
\node[scale = .6, color = red] at (10, 12) {$\times$};
\node[scale = .6, color = red] at (10, 13) {$\times$};
\node[scale = .6, color = red] at (10, 14) {$\times$};
\node[scale = .6, color = red] at (10, 15) {$\times$};
\node[scale = .6, color = red] at (10, 16) {$\times$};
\node[scale = .6, color = red] at (10, 17) {$\times$};
\node[scale = .6, color = red] at (10, 18) {$\times$};
\node[scale = .6, color = red] at (10, 19) {$\times$};

\node[scale = .6, color = red] at (11, 2) {$\times$};
\node[scale = .6, color = red] at (11, 3) {$\times$};
\node[scale = .6, color = red] at (11, 4) {$\times$};
\node[scale = .6, color = red] at (11, 5) {$\times$};
\node[scale = .6, color = red] at (11, 6) {$\times$};
\node[scale = .6, color = red] at (11, 7) {$\times$};
\node[scale = .6, color = red] at (11, 8) {$\times$};
\node[scale = .6, color = red] at (11, 9) {$\times$};
\node[scale = .6, color = red] at (11, 10) {$\times$};
\node[scale = .6, color = red] at (11, 11) {$\times$};
\node[scale = .6, color = red] at (11, 12) {$\times$};
\node[scale = .6, color = red] at (11, 13) {$\times$};
\node[scale = .6, color = red] at (11, 14) {$\times$};
\node[scale = .6, color = red] at (11, 15) {$\times$};
\node[scale = .6, color = red] at (11, 16) {$\times$};
\node[scale = .6, color = red] at (11, 17) {$\times$};

\node[scale = .6, color = red] at (12, 0) {$\times$};
\node[scale = .6, color = red] at (12, 1) {$\times$};
\node[scale = .6, color = red] at (12, 2) {$\times$};
\node[scale = .6, color = red] at (12, 3) {$\times$};
\node[scale = .6, color = red] at (12, 4) {$\times$};
\node[scale = .6, color = red] at (12, 5) {$\times$};
\node[scale = .6, color = red] at (12, 6) {$\times$};
\node[scale = .6, color = red] at (12, 7) {$\times$};
\node[scale = .6, color = red] at (12, 8) {$\times$};
\node[scale = .6, color = red] at (12, 9) {$\times$};
\node[scale = .6, color = red] at (12, 10) {$\times$};
\node[scale = .6, color = red] at (12, 11) {$\times$};
\node[scale = .6, color = red] at (12, 12) {$\times$};
\node[scale = .6, color = red] at (12, 13) {$\times$};
\node[scale = .6, color = red] at (12, 14) {$\times$};
\node[scale = .6, color = red] at (12, 15) {$\times$};

\node[scale = .6, color = red] at (13, 0) {$\times$};
\node[scale = .6, color = red] at (13, 1) {$\times$};
\node[scale = .6, color = red] at (13, 2) {$\times$};
\node[scale = .6, color = red] at (13, 3) {$\times$};
\node[scale = .6, color = red] at (13, 4) {$\times$};
\node[scale = .6, color = red] at (13, 5) {$\times$};
\node[scale = .6, color = red] at (13, 6) {$\times$};
\node[scale = .6, color = red] at (13, 7) {$\times$};
\node[scale = .6, color = red] at (13, 8) {$\times$};
\node[scale = .6, color = red] at (13, 9) {$\times$};
\node[scale = .6, color = red] at (13, 10) {$\times$};
\node[scale = .6, color = red] at (13, 11) {$\times$};
\node[scale = .6, color = red] at (13, 12) {$\times$};
\node[scale = .6, color = red] at (13, 13) {$\times$};

\draw[color=red, ->] (12, 16.2) -- (13, 17.2);
\draw[color=red,->] (4, 21) -- (4, 22.2);

\filldraw (2, 0) circle (4pt);
\filldraw (2, 1) circle (4pt);
\filldraw (3, 0) circle (4pt);
\filldraw[color=white] (4, 0) circle (4pt);
\draw[color=black] (4, 0) circle (4pt);
\node[scale=.35] at (4.0075, 0-.02) {$\blacktriangledown$};

\draw (4-.12,.07+1) -- (4.12, .07+1);
\draw (4, .07+1) -- (4, -.15+1);

\draw (4-.12,.07+2) -- (4.12, .07+2);
\draw (4, .07+2) -- (4, -.15+2);

\draw (4-.12,.07+3) -- (4.12, .07+3);
\draw (4, .07+3) -- (4, -.15+3);

\node[scale=.35] at (3, 1) {$\blacktriangledown$};
\node[scale=.35] at (3, 2) {$\blacktriangledown$};

\draw (3-.12,.07+3) -- (3.12, .07+3);
\draw (3, .07+3) -- (3, -.15+3);

\draw (3-.12,.07+4) -- (3.12, .07+4);
\draw (3, .07+4) -- (3, -.15+4);

\draw (3-.12,.07+5) -- (3.12, .07+5);
\draw (3, .07+5) -- (3, -.15+5);

\draw (3-.12,.07+6) -- (3.12, .07+6);
\draw (3, .07+6) -- (3, -.15+6);

\draw (3-.12,.07+7) -- (3.12, .07+7);
\draw (3, .07+7) -- (3, -.15+7);

\filldraw[color=white] (5, 0) circle (4pt);
\draw[color=black] (5, 0) circle (4pt);

\filldraw[color=white] (6, 0) circle (4pt);
\draw[color=black] (6, 0) circle (4pt);

\filldraw[color=white] (7, 0) circle (4pt);
\draw[color=black] (7, 0) circle (4pt);

\filldraw[color=white] (8, 0) circle (4pt);
\draw[color=black] (8, 0) circle (4pt);

\filldraw[color=white] (9, 0) circle (4pt);
\draw[color=black] (9, 0) circle (4pt);
\end{tikzpicture}
\hspace{.7in}
\begin{tikzpicture}[scale = .4]
\node[scale=.5] at (1, -.7) {$1$};
\node[scale=.5] at (2, -.7) {$2$};
\node[scale=.5] at (3, -.7) {$3$};
\node[scale=.5] at (4, -.7) {$4$};
\node[scale=.5] at (5, -.7) {$5$};
\node[scale=.5] at (6, -.7) {$6$};
\node[scale=.5] at (7, -.7) {$7$};
\node[scale=.5] at (8, -.7) {$8$};
\node[scale=.5] at (9, -.7) {$9$};
\node[scale=.5] at (10, -.7) {$10$};
\node[scale=.5] at (11, -.7) {$11$};
\node[scale=.5] at (12, -.7) {$12$};
\node[scale=.5] at (13, -.7) {$13$};

\node[scale=.5] at (-.7, 1) {$1$};
\node[scale=.5] at (-.7, 2) {$2$};
\node[scale=.5] at (-.7, 3) {$3$};
\node[scale=.5] at (-.7, 4) {$4$};
\node[scale=.5] at (-.7, 5) {$5$};
\node[scale=.5] at (-.7, 6) {$6$};
\node[scale=.5] at (-.7, 7) {$7$};
\node[scale=.5] at (-.7, 8) {$8$};
\node[scale=.5] at (-.7, 9) {$9$};
\node[scale=.5] at (-.7, 10) {$10$};
\node[scale=.5] at (-.7, 11) {$11$};
\node[scale=.5] at (-.7, 12) {$12$};
\node[scale=.5] at (-.7, 13) {$13$};
\node[scale=.5] at (-.7, 14) {$14$};
\node[scale=.5] at (-.7, 15) {$15$};
\node[scale=.5] at (-.7, 16) {$16$};
\node[scale=.5] at (-.7, 17) {$17$};
\node[scale=.5] at (-.7, 18) {$18$};
\node[scale=.5] at (-.7, 19) {$19$};
\node[scale=.5] at (-.7, 20) {$20$};

\draw[->] (0, 0) -- (14, 0);
\draw[->] (0, 0) -- (0, 21);
\node[scale=.9] at (14.4,0) {$g$};
\node[scale=.9] at (0, 21.4) {$n$};
\draw (-.1, 1) -- (.1, 1);
\draw (-.1, 2) -- (.1, 2);
\draw (1, -.1) -- (1, .1);
\draw (2, -.1) -- (2, .1);
\draw (3, -.1) -- (3, .1);
\draw (4, -.1) -- (4, .1);
\filldraw (0, 3) circle (4pt);
\filldraw (0, 4) circle (4pt);
\filldraw (0, 5) circle (4pt);
\filldraw (0, 6) circle (4pt);
\filldraw (0, 7) circle (4pt);
\filldraw (0, 8) circle (4pt);
\filldraw (0, 9) circle (4pt);
\filldraw (0, 10) circle (4pt);
\filldraw (0, 11) circle (4pt);
\filldraw (0, 12) circle (4pt);
\filldraw (0, 13) circle (4pt);
\filldraw (0, 14) circle (4pt);
\filldraw (0, 15) circle (4pt);
\filldraw (0, 16) circle (4pt);
\filldraw (0, 17) circle (4pt);
\filldraw (0, 18) circle (4pt);
\filldraw (0, 19) circle (4pt);
\filldraw (0, 20) circle (4pt);
\filldraw (1, 1) circle (4pt);
\filldraw (1, 2) circle (4pt);
\filldraw (1, 3) circle (4pt);
\filldraw (1, 4) circle (4pt);
\filldraw (1, 5) circle (4pt);
\filldraw (1, 6) circle (4pt);
\filldraw (1, 7) circle (4pt);
\filldraw (1, 8) circle (4pt);
\filldraw (1, 9) circle (4pt);
\filldraw (1, 10) circle (4pt);
\node[scale = .6, color = red] at (1, 11) {$\times$};
\node[scale = .6, color = red] at (1, 12) {$\times$};
\node[scale = .6, color = red] at (1, 13) {$\times$};
\node[scale = .6, color = red] at (1, 14) {$\times$};
\node[scale = .6, color = red] at (1, 15) {$\times$};
\node[scale = .6, color = red] at (1, 16) {$\times$};
\node[scale = .6, color = red] at (1, 17) {$\times$};
\node[scale = .6, color = red] at (1, 18) {$\times$};
\node[scale = .6, color = red] at (1, 19) {$\times$};
\node[scale = .6, color = red] at (1, 20) {$\times$};

\draw[color=red] (2-.15,10+.15) -- (2+.15, 10-.15);
\draw[color=red] (2-.15,11+.15) -- (2+.15, 11-.15);
\draw[color=red] (2-.15,12+.15) -- (2+.15, 12-.15);
\draw[color=red] (2-.15,13+.15) -- (2+.15, 13-.15);
\node[scale = .6, color = red] at (2, 14) {$\times$};
\node[scale = .6, color = red] at (2, 15) {$\times$};
\node[scale = .6, color = red] at (2, 16) {$\times$};
\node[scale = .6, color = red] at (2, 17) {$\times$};
\node[scale = .6, color = red] at (2, 18) {$\times$};
\node[scale = .6, color = red] at (2, 19) {$\times$};
\node[scale = .6, color = red] at (2, 20) {$\times$};

\draw[color=red] (3-.15,10+.15) -- (3+.15, 10-.15);
\draw[color=red] (3-.15,11+.15) -- (3+.15, 11-.15);
\draw[color=red] (3-.15,12+.15) -- (3+.15, 12-.15);
\draw[color=red] (3-.15,13+.15) -- (3+.15, 13-.15);
\draw[color=red] (3-.15,14+.15) -- (3+.15, 14-.15);
\node[scale = .6, color = red] at (3, 15) {$\times$};
\node[scale = .6, color = red] at (3, 16) {$\times$};
\node[scale = .6, color = red] at (3, 17) {$\times$};
\node[scale = .6, color = red] at (3, 18) {$\times$};
\node[scale = .6, color = red] at (3, 19) {$\times$};
\node[scale = .6, color = red] at (3, 20) {$\times$};

\node[scale = .6, color = red] at (4, 16) {$\times$};
\node[scale = .6, color = red] at (4, 17) {$\times$};
\node[scale = .6, color = red] at (4, 18) {$\times$};
\node[scale = .6, color = red] at (4, 19) {$\times$};
\node[scale = .6, color = red] at (4, 20) {$\times$};

\node[scale = .6, color = red] at (5, 14) {$\times$};
\node[scale = .6, color = red] at (5, 15) {$\times$};
\node[scale = .6, color = red] at (5, 16) {$\times$};
\node[scale = .6, color = red] at (5, 17) {$\times$};
\node[scale = .6, color = red] at (5, 18) {$\times$};
\node[scale = .6, color = red] at (5, 19) {$\times$};
\node[scale = .6, color = red] at (5, 20) {$\times$};

\node[scale = .6, color = red] at (6, 12) {$\times$};
\node[scale = .6, color = red] at (6, 13) {$\times$};
\node[scale = .6, color = red] at (6, 14) {$\times$};
\node[scale = .6, color = red] at (6, 15) {$\times$};
\node[scale = .6, color = red] at (6, 16) {$\times$};
\node[scale = .6, color = red] at (6, 17) {$\times$};
\node[scale = .6, color = red] at (6, 18) {$\times$};
\node[scale = .6, color = red] at (6, 19) {$\times$};
\node[scale = .6, color = red] at (6, 20) {$\times$};

\node[scale = .6, color = red] at (7, 10) {$\times$};
\node[scale = .6, color = red] at (7, 11) {$\times$};
\node[scale = .6, color = red] at (7, 12) {$\times$};
\node[scale = .6, color = red] at (7, 13) {$\times$};
\node[scale = .6, color = red] at (7, 14) {$\times$};
\node[scale = .6, color = red] at (7, 15) {$\times$};
\node[scale = .6, color = red] at (7, 16) {$\times$};
\node[scale = .6, color = red] at (7, 17) {$\times$};
\node[scale = .6, color = red] at (7, 18) {$\times$};
\node[scale = .6, color = red] at (7, 19) {$\times$};
\node[scale = .6, color = red] at (7, 20) {$\times$};

\node[scale = .6, color = red] at (8, 8) {$\times$};
\node[scale = .6, color = red] at (8, 9) {$\times$};
\node[scale = .6, color = red] at (8, 10) {$\times$};
\node[scale = .6, color = red] at (8, 11) {$\times$};
\node[scale = .6, color = red] at (8, 12) {$\times$};
\node[scale = .6, color = red] at (8, 13) {$\times$};
\node[scale = .6, color = red] at (8, 14) {$\times$};
\node[scale = .6, color = red] at (8, 15) {$\times$};
\node[scale = .6, color = red] at (8, 16) {$\times$};
\node[scale = .6, color = red] at (8, 17) {$\times$};
\node[scale = .6, color = red] at (8, 18) {$\times$};
\node[scale = .6, color = red] at (8, 19) {$\times$};
\node[scale = .6, color = red] at (8, 20) {$\times$};

\node[scale = .6, color = red] at (9, 6) {$\times$};
\node[scale = .6, color = red] at (9, 7) {$\times$};
\node[scale = .6, color = red] at (9, 8) {$\times$};
\node[scale = .6, color = red] at (9, 9) {$\times$};
\node[scale = .6, color = red] at (9, 10) {$\times$};
\node[scale = .6, color = red] at (9, 11) {$\times$};
\node[scale = .6, color = red] at (9, 12) {$\times$};
\node[scale = .6, color = red] at (9, 13) {$\times$};
\node[scale = .6, color = red] at (9, 14) {$\times$};
\node[scale = .6, color = red] at (9, 15) {$\times$};
\node[scale = .6, color = red] at (9, 16) {$\times$};
\node[scale = .6, color = red] at (9, 17) {$\times$};
\node[scale = .6, color = red] at (9, 18) {$\times$};
\node[scale = .6, color = red] at (9, 19) {$\times$};
\node[scale = .6, color = red] at (9, 20) {$\times$};

\node[scale = .6, color = red] at (10, 4) {$\times$};
\node[scale = .6, color = red] at (10, 5) {$\times$};
\node[scale = .6, color = red] at (10, 6) {$\times$};
\node[scale = .6, color = red] at (10, 7) {$\times$};
\node[scale = .6, color = red] at (10, 8) {$\times$};
\node[scale = .6, color = red] at (10, 9) {$\times$};
\node[scale = .6, color = red] at (10, 10) {$\times$};
\node[scale = .6, color = red] at (10, 11) {$\times$};
\node[scale = .6, color = red] at (10, 12) {$\times$};
\node[scale = .6, color = red] at (10, 13) {$\times$};
\node[scale = .6, color = red] at (10, 14) {$\times$};
\node[scale = .6, color = red] at (10, 15) {$\times$};
\node[scale = .6, color = red] at (10, 16) {$\times$};
\node[scale = .6, color = red] at (10, 17) {$\times$};
\node[scale = .6, color = red] at (10, 18) {$\times$};
\node[scale = .6, color = red] at (10, 19) {$\times$};

\node[scale = .6, color = red] at (11, 2) {$\times$};
\node[scale = .6, color = red] at (11, 3) {$\times$};
\node[scale = .6, color = red] at (11, 4) {$\times$};
\node[scale = .6, color = red] at (11, 5) {$\times$};
\node[scale = .6, color = red] at (11, 6) {$\times$};
\node[scale = .6, color = red] at (11, 7) {$\times$};
\node[scale = .6, color = red] at (11, 8) {$\times$};
\node[scale = .6, color = red] at (11, 9) {$\times$};
\node[scale = .6, color = red] at (11, 10) {$\times$};
\node[scale = .6, color = red] at (11, 11) {$\times$};
\node[scale = .6, color = red] at (11, 12) {$\times$};
\node[scale = .6, color = red] at (11, 13) {$\times$};
\node[scale = .6, color = red] at (11, 14) {$\times$};
\node[scale = .6, color = red] at (11, 15) {$\times$};
\node[scale = .6, color = red] at (11, 16) {$\times$};
\node[scale = .6, color = red] at (11, 17) {$\times$};

\node[scale = .6, color = red] at (12, 0) {$\times$};
\node[scale = .6, color = red] at (12, 1) {$\times$};
\node[scale = .6, color = red] at (12, 2) {$\times$};
\node[scale = .6, color = red] at (12, 3) {$\times$};
\node[scale = .6, color = red] at (12, 4) {$\times$};
\node[scale = .6, color = red] at (12, 5) {$\times$};
\node[scale = .6, color = red] at (12, 6) {$\times$};
\node[scale = .6, color = red] at (12, 7) {$\times$};
\node[scale = .6, color = red] at (12, 8) {$\times$};
\node[scale = .6, color = red] at (12, 9) {$\times$};
\node[scale = .6, color = red] at (12, 10) {$\times$};
\node[scale = .6, color = red] at (12, 11) {$\times$};
\node[scale = .6, color = red] at (12, 12) {$\times$};
\node[scale = .6, color = red] at (12, 13) {$\times$};
\node[scale = .6, color = red] at (12, 14) {$\times$};
\node[scale = .6, color = red] at (12, 15) {$\times$};

\node[scale = .6, color = red] at (13, 0) {$\times$};
\node[scale = .6, color = red] at (13, 1) {$\times$};
\node[scale = .6, color = red] at (13, 2) {$\times$};
\node[scale = .6, color = red] at (13, 3) {$\times$};
\node[scale = .6, color = red] at (13, 4) {$\times$};
\node[scale = .6, color = red] at (13, 5) {$\times$};
\node[scale = .6, color = red] at (13, 6) {$\times$};
\node[scale = .6, color = red] at (13, 7) {$\times$};
\node[scale = .6, color = red] at (13, 8) {$\times$};
\node[scale = .6, color = red] at (13, 9) {$\times$};
\node[scale = .6, color = red] at (13, 10) {$\times$};
\node[scale = .6, color = red] at (13, 11) {$\times$};
\node[scale = .6, color = red] at (13, 12) {$\times$};
\node[scale = .6, color = red] at (13, 13) {$\times$};

\draw[color=red, ->] (12, 16.2) -- (13, 17.2);
\draw[color=red,->] (4, 21) -- (4, 22.2);

\filldraw (2, 0) circle (4pt);
\filldraw (2, 1) circle (4pt);
\filldraw[color=blue] (2, 2) circle (4pt);
\filldraw[color=blue] (2, 3) circle (4pt);
\filldraw[color=blue] (2, 4) circle (4pt);
\filldraw[color=blue] (2, 5) circle (4pt);
\filldraw[color=blue] (2, 6) circle (4pt);
\filldraw[color=blue] (2, 7) circle (4pt);
\filldraw[color=blue] (2, 8) circle (4pt);
\filldraw[color=blue] (2, 9) circle (4pt);
\draw[color=blue] (2, 10) circle (4pt);
\filldraw (3, 0) circle (4pt);
\filldraw[color=blue] (3, 1) circle (4pt);
\filldraw[color=blue] (3, 2) circle (4pt);
\filldraw[color=blue] (3, 3) circle (4pt);
\filldraw[color=blue] (3, 4) circle (4pt);
\filldraw[color=blue] (3, 5) circle (4pt);
\filldraw[color=blue] (3, 6) circle (4pt);
\filldraw[color=blue] (3, 7) circle (4pt);
\filldraw[color=blue] (3, 8) circle (4pt);
\draw[color=blue] (3, 9) circle (4pt);
\draw[color=blue] (3, 10) circle (4pt);
\draw[color=blue] (3, 11) circle (4pt);
\filldraw[color = white] (4,0) circle (4pt);
\filldraw[color=blue] (4, 0) circle (4pt);
\filldraw[color=blue] (4, 1) circle (4pt);
\filldraw[color=blue] (4, 2) circle (4pt);
\filldraw[color=blue] (4, 3) circle (4pt);
\filldraw[color=blue] (4, 4) circle (4pt);
\filldraw[color=blue] (4, 5) circle (4pt);
\filldraw[color=blue] (4, 6) circle (4pt);
\filldraw[color=blue] (4, 7) circle (4pt);
\filldraw[color=white] (4, 7) circle (1.6pt);
\draw[color=blue] (4, 8) circle (4pt);
\draw[color=blue] (4, 9) circle (4pt);
\draw[color=blue] (4, 10) circle (4pt);
\draw[color=blue] (4, 11) circle (4pt);
\filldraw[color=white] (6, 0) circle (4pt);
\filldraw[color=blue] (6, 0) circle (4pt);
\filldraw[color=blue] (6, 1) circle (4pt);
\filldraw[color=blue] (6, 2) circle (4pt);
\filldraw[color=blue] (6, 3) circle (4pt);
\filldraw[color=white] (6, 3) circle (1.6pt);
\filldraw[color=blue] (6, 4) circle (4pt);
\filldraw[color=white] (6, 4) circle (1.6pt);
\filldraw[color=blue] (6, 5) circle (4pt);
\filldraw[color=white] (6, 5) circle (1.6pt);

\filldraw[color=white] (5, 0) circle (4pt);
\filldraw[color=blue] (5, 0) circle (4pt);
\filldraw[color=blue] (5, 1) circle (4pt);
\filldraw[color=blue] (5, 2) circle (4pt);
\filldraw[color=blue] (5, 3) circle (4pt);
\filldraw[color=blue] (5, 4) circle (4pt);
\filldraw[color=blue] (5, 5) circle (4pt);
\filldraw[color=blue] (5, 6) circle (4pt);
\draw[color=blue] (5, 7) circle (4pt);
\filldraw[color=white] (5, 5) circle (1.6pt);
\filldraw[color=white] (5, 6) circle (1.6pt);

\filldraw[color=white] (7, 0) circle (4pt);
\filldraw[color=blue] (7, 0) circle (4pt);

\filldraw[color=white] (8, 0) circle (4pt);
\draw[color=black] (8, 0) circle (4pt);

\filldraw[color=white] (9, 0) circle (4pt);
\draw[color=black] (9, 0) circle (4pt);
\end{tikzpicture}
\end{center}
\caption{Left: Previously known results regarding Questions \ref{mainq} and \ref{cohq}. \\
Right: Our new results (Theorem \ref{main}) in context, pictured in blue.}
\label{bigfig}
\end{figure}

\begin{center}
\begin{tikzpicture}[scale=.4]
\draw (-2, 11) -- (-2, 0) -- (19.5, 0) -- (19.5, 11) -- (-2, 11);
\filldraw (-1, 10) circle (4pt);
\node[scale=.8] at (9.5, 10) {$A^*(\Mb_{g,n})=R^*(\Mb_{g,n})$ and $H^*(\Mb_{g,n}) = RH^*(\Mb_{g,n})$};

\node[scale=.35] at (-1, 8.5) {$\blacktriangledown$};
\node[scale=.8,align=left] at (4.5, 8.5) {$H^*(\Mb_{g,n})=RH^*(\Mb_{g,n})$};

\draw (-1-.12,.07+7) -- (-1+0.12, .07+7);
\draw (-1, .07+7) -- (-1, -.15+7);
\node[scale=.8,align=left] at (3.35, 7) {$\#\Mb_{g,n}(\mathbb{F}_q)=P(q)$};

\filldraw[color=black] (-1, 5.5) circle (4pt);
\filldraw[color=white] (-1, 5.5) circle (1.6pt);
\node[scale=.8,align=left] at (4, 5.5) {$A^*(\M_{g,n}^{\ct})=R^*(\M_{g,n}^{\ct})$};

\draw (-1, 4) circle (4pt);
\node[scale=.8,align=left] at (4, 4) {$A^*(\M_{g,n})=R^*(\M_{g,n})$};

\node[scale = .6, color = red] at (-1, 2.5) {$\times$};
\node[scale=.8] at (9.5, 2.5) {$A^*(\Mb_{g,n})\neq R^*(\Mb_{g,n})$ and $H^*(\Mb_{g,n}) \neq RH^*(\Mb_{g,n})$};

\draw[color = red] (-1+.15, 1-.15) -- (-1-.15,1+.15);
\node[scale=.8,align=left] at (4.5, 1) {$H^*(\Mb_{g,n})\neq RH^*(\Mb_{g,n})$};

\filldraw (30-2, 10) circle (4pt);
\draw[->] (30-2-.4, 10-.4) -- (27.5-2+.4, 7+.4);
\draw[->] (30-2+.4, 10-.4) -- (32.5-2-.4, 7+.4);

\node[scale=.35] at (27.5-2, 7) {$\blacktriangledown$};

\draw[->] (27.5-2,7-.6) -- (27.5-2, 3.5+.6);
\draw (27.5-2-.12,.07+3.5) -- (27.5-2+0.12, .07+3.5);
\draw (27.5-2, .07+3.5) -- (27.5-2, -.15+3.5);

\filldraw[color=black] (32.5-2, 7) circle (4pt);
\filldraw[color=white] (32.5-2, 7) circle (1.6pt);
\draw[->] (32.5-2,7-.6) -- (32.5-2, 3.5+.6);
\draw (32.5-2, 3.5) circle (4pt);
\draw[color=white] (33, 3) circle (4pt);

\node[color=red, scale=.6] at (27.5-2,  1-.3) {$\times$};

\draw[color=red] (30.5-.15, 1-.3+.15) -- (30.5+.15, 1-.3-.15);
\draw[->] (27.5-2+.5, 1-.3) -- (30.5-.5,1-.3);
\end{tikzpicture}
\end{center}

\subsection{Summary of previous work} \label{previouswork}
The filled and open circles in the bottom row of Figure \ref{bigfig} represent the results mentioned in the first paragraph of the paper.
A full account of the known cases where $A^*(\Mb_{g,n}) = R^*(\Mb_{g,n})$ and $H^*(\Mb_{g,n}) = RH^*(\Mb_{g,n})$ is
\begin{itemize}
    \item $g = 0$: $n \geq 3$ by Keel in 1992 \cite{Keel},
    \item $g = 1$: $n \leq 10$ by Belorousski in 1998 \cite{Belorousski} (Chow); Petersen in 2014 \cite{Petersengenus1} (cohomology),
    \item $g = 2$: $n=0$ by Mumford in 1983 \cite{Mumford}, $n=1$ by Faber in 1988 \cite{FaberPhD},
    \item $g = 3$: $n = 0$ by Faber in 1990 \cite{FaberI}.
\end{itemize}

If $A^*(\Mb_{g,n}) = R^*(\Mb_{g,n})$ then $A^*(\M_{g,n}) = R^*(\M_{g,n})$ by excision. The weaker statement that $A^*(\M_{g,n}) = R^*(\M_{g,n})$ is pictured by open circles $\circ$ and was previously known for:
\begin{itemize}
    \item[$\circ$] $g = 4$: $n=0$ by Faber in 1990 \cite{FaberII},
    \item[$\circ$]  $g = 5$: $n=0$ by Izadi in 1995 \cite{Izadi},
    \item[$\circ$]  $g = 6$: $n=0$ by Penev and Vakil in 2015 \cite{PenevVakil},
    \item[$\circ$]  $g = 7, 8, 9$: $n=0$ by the authors in 2021 \cite{789}.
\end{itemize}

\begin{rem}[Integral coefficients]
Determining the Chow ring of $\Mb_{g,n}$ with integral coefficients is a much more subtle problem. The cases where the integral Chow ring is known are $\Mb_{0,n}$ by Keel in 1992 \cite{Keel}; $\Mb_{1,1}$ by Edidin--Graham in 1998 \cite{EdidinGraham}; $\Mb_{2}$ by E. Larson \cite{EricLarson} in 2021; $\Mb_{1,2}$ and $\Mb_{2,1}$ by Di Lorenzo, Pernice, and Vistoli in 2021 \cite{DiLorenzoPerniceVistoli}.
\end{rem}

There are also cases where results are known in cohomology but not in Chow. 
Much of the previous work computing $H^*(\Mb_{g,n})$ has used point counting. 
The number of points $\#\Mb_{g,n}(\mathbb{F}_q)$ is defined as the sum over isomorphism classes of $\mathbb{F}_q$ points, weighted by size of the automorphism group.
We say $\Mb_{g,n}$ has \emph{polynomial point count} 
 if there exists a polynomial $P$ with integer coefficients such that $\#\Mb_{g,n}(\mathbb{F}_q) = P(q)$ for all prime powers $q$.
If $H^*(\Mb_{g,n}) = RH^*(\Mb_{g,n})$ then $\Mb_{g,n}$ has polynomial point count (
see Proposition \ref{ptcounts}).
On the other hand, if $\Mb_{g,n}$ has polynomial point count, then all odd cohomology vanishes, and knowing the polynomial
determines the dimensions of all cohomology groups. Previously, $\Mb_{g,n}$ was shown to have polynomial point count (and the polynomial determined) for
\begin{itemize}
    \item[{\tiny $\top$}] $g = 2$: $n \leq 7$ by Bergstr\"om in 2009 \cite{Bergstrom2}; 
    \item[{\tiny $\top$}] $g = 3$: $n \leq 1$ by Getzler--Looijenga in 1999 \cite{GetzlerLooijenga}; $n \leq 5$ by Bergstr\"om in 2008 \cite{Bergstrom3}; 
    \item[{\tiny $\top$}] $g = 4$: $n = 0$ by 
    Bergstr\"om and Tommasi in 2006 \cite{BergstromTommasi}; $1 \leq n \leq 3$ by  Bergstr\"om, Faber and Payne in 2022
    \cite{BergstromFaberPayne}.
\end{itemize}
Upon completing this manuscript, the authors learned that Begstr\"om and Faber have improved the first two cases above to $g = 2$, $n \leq 9$ and $g = 3$, $n \leq 7$ \cite{BergstromFaber}.

Having polynomial point count does not rule out the possibility of non-tautological classes in even degrees, but with some additional work, it is known that all cohomology is tautological in the following cases:
\begin{itemize}
    \item[{\tiny $\blacktriangledown$}] $g=2$: $n \leq 9$ by work of Petersen and Tommasi in 2016 \cite{Petersen, PetersenTommasi}, which proved all even cohomology is tautological for $g = 2$ and $n < 20$.
    \item[{\tiny $\blacktriangledown$}] $g = 3$: $n \leq 1$ and $g = 4$: $n = 0$ by Schmitt and van Zelm in 2020 \cite{SchmittvanZelm}.
\end{itemize}
In the remaining genus $3$ and $4$ cases where polynomial point count is known, it is expected that all cohomology is tautological.
While computer calculations in low codimension support this expectation, 
confirming it by computer seems beyond present techniques (see \cite[Remark 1.8]{BergstromFaberPayne}).
One consequence of our theorem will be confirmation that these polynomial point counts, as well as many new ones, arise because all cohomology is tautological.

Meanwhile, non-tautological classes have been discovered in the following cases: 
\begin{itemize}
    \item[{\color{red} $\times$}] $g = 1$: $n \geq 11$. The weight $12$ discriminant cusp form gives rise to a holomorphic differential form on $\Mb_{1,11}$ (see \cite[Section 2.3]{FPhandbook} for an exposition). 
    Hence, $H^{11}(\Mb_{1,11}) \neq 0$, giving non-tautological cohomology. 
    Moreover, by a theorem of Ro\u{\i}tman \cite{Roitman},
    existence of a holomorphic form implies that the group of zero-cycles in Chow is infinite-dimensional, so they cannot all be tautological. (See \cite{zerocycles} for an discussion of which points on $\Mb_{g,n}$ can give tautological zero cycles.) Because pullback along the forgetful maps $\Mb_{g,n+1} \to \Mb_{g,n}$ is injective on Chow and cohomology, the existence of non-tautological classes propagates upward to all $n \geq 11$.
    \item[{\color{red} $\times$}] $g = 2$: $n \geq 14$ and $g = 3$: $n \geq 15$. Similarly, there are known holomorphic differential forms on $\Mb_{2,14}$ (see  \cite[Section 3.5]{FPhandbook}) and $\Mb_{3,15}$ (see \cite{CvdGF}).
     \item[{\color{red} $\times$}] Any $2g + n \geq 24$. 
     Graber--Pandharipande \cite{GraberPandharipande} produced the first explicit non-tautological \emph{algebraic} cycle (in Chow and cohomology) on $\Mb_{2,20}$, the fundamental class of the bielliptic locus with $10$ pairs of conjugate markings.
     Generalizing these techniques, van Zelm \cite{vanZelm} then produced non-tautological algebraic cycles on $\Mb_{g,n}$ for all $g,n$ with $2g + n \geq 24$, arising again from the bielliptic loci. 
     \item[{\tiny {\color{red} $\diagdown$}}] $g = 2,3$: $n \geq 10$. Upon completing this manuscript, the authors learned that $H^*(\Mb_{g,n}) \neq RH^*(\Mb_{g,n})$ in these cases by work of Bergstr\"om and Faber \cite{BergstromFaber}.
    \item[{\tiny {\color{red} $\diagdown$}}] We also note an interesting result of Pikaart \cite{Pikaart} which says $H^{33}(\Mb_{g,n}) \neq 0$ for $g$ sufficiently large.
\end{itemize}

\subsection{Statement of main result}
Our main result contributes 32 new cases, pictured by filled blue dots {\color{blue} $\bullet$}, where $A^*(\Mb_{g,n}) = R^*(\Mb_{g,n})$ and $H^*(\Mb_{g,n}) = RH^*(\Mb_{g,n})$. Upon restricting to the locus of compact type curves $\M_{g,n}^{\ct} \subset \Mb_{g,n}$, we obtain a few additional cases in Chow where $A^*(\M_{g,n}^{\ct}) = R^*(\M_{g,n}^{\ct})$ (pictured by bold open circles {\put(6,3){\color{blue}\circle*{5}}\put(6,3){\color{white}\circle*{2}}} \ \ ). Restricting further to the locus of smooth curves or those with only rational tails, $\M_{g,n} \subset \M_{g,n}^{\rt} \subset \M_{g,n}^{\ct}$, we have a few more cases where $A^*(\M_{g,n}) = R^*(\M_{g,n})$ and $A^*(\M_{g,n}^{\rt}) = R^*(\M_{g,n}^{\rt})$ (pictured by open circles {\color{blue} $\circ$}).

\begin{thm}\label{main}
The Chow/cohomology rings of the following moduli spaces are generated by tautological classes. 

\begin{center}
\begin{tabular}{c|c|c|c|c|c|c|}
& \ $g = 2$ \ & \  $g=3$ \ & \ $g=4$ \ & \ $g=5$ \ & \ $g=6$ \ & \ $g=7$ \ \\[6pt]
\hline &&&&& &\\[-6pt]  
{\color{blue} $\bullet$} $A^*(\Mb_{g,n}) = R^*(\Mb_{g,n})$ for $n \leq$ &  $9$  & $8$ & $6$ & $4$ & $2$ & $0$  \\[6pt]
{\color{white} $\bullet$} $H^*(\Mb_{g,n}) = RH^*(\Mb_{g,n})$ {\color{white} foraiii} &&&&& & \\[6pt]
\hline &&&&& &\\
{\color{white} i} {\put(-2,3){\color{blue}\circle*{5}}\put(-2,3){\color{white}\circle*{2}}}  $A^*(\M_{g,n}^{\ct}) = R^*(\M_{g,n}^{\ct})$ for $n \leq$      & $9$ & $8$ & $7$ & $6$ & $5$ & $0$ \\[12pt]
\hline &&&&& &\\[-6pt] 
{\color{blue} $\circ$} $A^*(\M_{g,n}^{\mathrm{rt}}) = R^*(\M_{g,n}^{\mathrm{rt}})$ for $n \leq $ & $10$ & $11$ & $11$ & $7$ & $5$ & $0$ \\[6pt]
{\color{white} $\circ$} $A^*(\M_{g,n}) = R^*(\M_{g,n})$ {\color{white} for $n \leq $} &&&&&&
\end{tabular}
\end{center}
\medskip
Moreover, for the $(g, n)$ in the top row, the cycle class map $A^*(\Mb_{g,n}) \to H^*(\Mb_{g,n})$ is an isomorphism.
\end{thm}

Theorem \ref{main} has immediate implications for point counting.

\begin{cor} 
\label{ptcor} The moduli space $\Mb_{g,n}$ has polynomial point count for $g = 2$ and $n \leq 9$; and $g \geq 3$ and $2g + n \leq 14$.
\end{cor}

Considering contributions from the boundary as in \cite[Proposition 4.2]{BergstromFaberPayne}, we also obtain point counting results for the following open moduli spaces.
\begin{cor}
The moduli space $\M_{g,n}$ has polynomial point count for all $2g + n \leq 12$.
\end{cor}

\begin{rem}[Applications for all $g$ and $n$]
The cases in the top row of Theorem \ref{main} will serve as crucial base cases for an inductive argument describing
the cohomology of $\Mb_{g,n}$ in low degrees for all $g, n$. 
It was previously shown that $H^k(\Mb_{g,n}) = 0$ for $k = 1,3,5$ by Arbarello and Cornalba \cite{ArbarelloCornalba}. This was recently extended to $k = 7,9$ by Bergstr\"om, Faber and Payne \cite{BergstromFaberPayne}, where the key new input was knowing $\Mb_{4,n}$ has polynomial point count for $n \leq 3$.
Forthcoming work of the authors and Sam Payne will show that
$H^{11}(\Mb_{g,n}) = 0$ for all $g \geq 2$, and
$H^k(\Mb_{g,n})$ is of Hodge--Tate type for all $g, n$ and even $k \leq 12$.
\end{rem}

\subsection{Overview of the new techniques}
We prove Theorem \ref{main} with a novel method, which allows us to take advantage of the flexibility of excision in Chow, yet still obtain results in cohomology at the end. The key idea is prove our spaces have the \emph{Chow--K\"unneth generation Property (CKgP)}, which is a way of making up for the lack of a K\"unneth formula in Chow. We say $X$ has the CKgP if for all $Y$, there is a surjection
\[A^*(X) \otimes A^*(Y) \rightarrow A^*(X \times Y).\]
 
Below are some important observations about the CKgP, which pave the way for our approach. Each italicized statement is not difficult to establish (and some have already appeared in the literature \cite{BaeSchmitt2,TotaroCKgP}), but they combine into a very powerful tool when applied to moduli spaces of curves:

\begin{enumerate}
\item \textit{The CKgP plays well with stratifications, products, and finite group quotients.} Thus, using the inductive nature of its boundary, we can reduce showing $\Mb_{g,n}$ has the CKgP and $A^*(\Mb_{g,n}) = R^*(\Mb_{g,n})$ to showing the open moduli spaces $\M_{g',n'}$ have the CKgP and $A^*(\M_{g',n'}) = R^*(\M_{g',n'})$ for $g' \leq g$ and $2g' + n' \leq 2g + n$ (as explained in Section \ref{mainstrategy}).

\medskip
\item \label{2} \textit{If $X\rightarrow Y$ is proper and surjective, and $X$ has the CKgP, then $Y$ has the CKgP}. Thus, to attack each piece $\M_{g',n'}$, we can pass from studying its gonality strata to studying Hurwitz spaces of covers with marked points (see Sections \ref{fphsec}--\ref{tetsec}).

\medskip
\item \label{3} \textit{If $U \subset X$ is open and $X$ has the CKgP, then $U$ has the CKgP}. Often, we can construct a moduli space of marked curves (or curves with a map to $\pp^1$ or $\pp^2$) as a quotient of affine space minus some discriminant locus. In contrast with having polynomial point count or knowing generators in cohomology, we may throw out \emph{arbitrary} closed subsets and still preserve CKgP and generators in Chow.

\medskip
\item \textit{If $X$ is smooth and proper and has the CKgP, then the cycle class map is an isomorphism.} In other words --- no matter what techniques from Chow we have used in \eqref{2} and \eqref{3} --- so long as we put our pieces back together into a smooth, proper space, we now also have access to cohomology.
\end{enumerate}

Part (1) above naturally leads us prove results for the collection of $\overline{\M}_{g,n}$ that satisfy $2g + n \leq a$ for some $a$ (a triangular region bounded by the axes and a line of slope $-2$).
From this perspective, one should hope to prove
$A^*(\overline{\M}_{g,n}) = R^*(\Mb_{g,n})$ for all $2g + n \leq 12$ (the largest such triangular region that does not contain $\Mb_{1,11}$).
Theorem \ref{main} confirms this hope, but also provides some answers to Questions \ref{mainq} and \ref{cohq} \emph{above} this region. Obtaining filled circles with $2g + n > 12$, such as $\Mb_7$, requires some more delicate, ad hoc arguments, and we view the results in this region as more surprising.

\subsection{The Gorenstein property} \label{gorsec}
Having determined many cases where tautological classes generate the Chow ring, a natural next question is what are the relations among them?
Pixton conjectured a set of relations among the generators of the tautological ring, which were later proven to hold by Pandharipande--Pixton--Zvonkine \cite{PandharipandePixtonZvonkine} (in cohomology) and Janda \cite{Janda} (in the Chow ring). Pixton's relations restrict to the Faber--Zagier relations on $\Mg$, and one can also restrict Pixton's relations on $\Mb_{g,n}$ to $\M^{\ct}_{g,n}$ or $\M^{\rt}_{g,n}$ \cite{Pixton}. We will let $R^*_P(\M)$ denote the ring generated by tautological classes modulo Pixton's relations, where $\M$ is any of the above moduli spaces of curves. Pixton has conjectured that $R_P^*(\M)=R^*(\M)$.

It is known that $R^i(\M) = 0$ for $i > d$ \cite{Looijenga, FaberPandharipande} and $R^d(\M) = \qq$ \cite{Looijenga, fabernonvanishing, GraberVakil, FaberPandharipande} where
\begin{equation*}
d := \begin{cases} g - 2 + n - \delta_{0,g} & \text{if $\M = \M_{g,n}^{\rt}$} \\
2g - 3 + n & \text{if $\M = \M_{g,n}^{\ct}$} \\
3g - 3 + n & \text{if $\M = \Mb_{g,n}$.}
\end{cases}
\end{equation*}
One of the prominent questions asked about the tautological ring is if $R^*(\M)$ is Gorenstein with socle in codimension $d$, meaning there is a perfect pairing
\[R^i(\M) \times R^{d - i}(\M) \to R^d(\M) \cong \qq. \]

The Gorenstein property is known to hold for $R^*(\Mb_{g,n}) = R^*_P(\Mb_{g,n})$ corresponding to
black filled circles on the left of Figure \ref{bigfig}, as well as for $R^*(\M_g) = R^*_P(\M_g)$ for $g \leq 23$ by Faber. In these cases, it was computationally proven that $R^*_P(\M)$ is Gorenstein using the Sage package admcycles \cite{admcycles}; this in turn implies $R^*(\M) = R^*_P(\M)$, since the perfect pairing prohibits any further relations. 
Although it was previously speculated that $R^*(\M_g)$ may be Gorenstein \cite{Faberconjecture}, it is known that $R_P^*(\M_g)$ is not Gorenstein for $g \geq 24$.

Furthermore, the Gorenstein property was shown to fail by Petersen and Tommasi for $R^*(\Mb_{2,n})$ and $RH^*(\Mb_{2,n})$ with $n\geq 20$, as well as
$R^*(\M_{2,n}^{\ct})$ with $n\geq 8$ \cite{PetersenTommasi,Petersen}.
Pixton has also computationally shown $R_P^*(\M_{g,n}^{\rt})$ is not Gorenstein in some other cases \cite[Appendix A]{Pixton}. However, a case of Gorenstein failure has yet to be proven for $R^*(\M_{g,n}^{\rt})$ (which includes the case of $R^*(\M_g)$). See \cite{calculus} for further discussion.

Using Poincar\'e duality, the $(g, n)$ in the top row of Theorem \ref{main} provide many new cases where
the Gorenstein property does hold in Chow and cohomology.

\begin{cor} \label{gcor}
Suppose $(g, n)$ satisfies $g =2$ and $n \leq 9$ or $g \geq 3$ and $2g + n \leq 14$. Then $A^*(\Mb_{g,n})=R^*(\Mb_{g,n})=H^*(\Mb_{g,n}) = RH^*(\Mb_{g,n})$ is Gorenstein.
\end{cor}

In these cases, the ring is (in theory) completely determined:
a class is zero if and only if it pairs to zero with all generators in complementary dimension. The pairing on generators of $R^*(\Mb_{g,n})$ is can be effectively computed (as explained in \cite[Section 1.5]{FPhandbook}) and has been implemented in admcycles \cite{admcycles}.
Corollary \ref{gcor} also provides a large region
on which to test Pixton's conjectures: if $R^*_P(\Mb_{g,n})$ is not Gorenstein in one of these cases, then a relation is missing.
We therefore propose the cases in Theorem \ref{main} as being of particular interest for future computational study.


\subsection{Outline of the paper}
In Section \ref{preliminaries}, we discuss the basics of the tautological rings. In Section \ref{CKgP}, we discuss the Chow--K\"unneth generation property, which is a key property for studying the boundary of $\Mb_{g,n}$. In Section \ref{mainstrategy}, we explain the inductive method. In Section \ref{zeroone}, we summarize previous results on the moduli of pointed curves of genus $0$ and $1$. In Section \ref{hypsec}, we recall results from \cite{Hyperelliptic} to deal with the hyperelliptic locus (which completes everything in genus $2$). In Section \ref{planesec}, we give a quotient stack presentation for certain stacks of pointed plane curves (which can be used to complete the cases in genus $3$). In Section \ref{fphsec}, we introduce the Faber--Pandharipande--Hurwitz cycles, which are tautological cycles associated to certain finite covers of $\pp^1$. In Sections \ref{trigsec} and \ref{tetsec}, we show that the Chow rings of certain loci of pointed trigonal and tetragonal curves are generated by these cycles.

\subsection{Notations and conventions.}
All schemes are taken over a fixed algebraically closed field of characteristic $0$ or greater than $5$. Stacks are fibered in groupoids over the category of schemes over the base field (or ring). All Chow rings are taken with rational coefficients. We use the subspace convention for projective bundles and Grassmann bundles. 

\subsection*{Acknowledgments} We are grateful to Elham Izadi, Kiran Kedlaya, Jun Lau, Dan Petersen, David Stapleton, and Ravi Vakil for helpful conversations. We thank Burt Totaro for pointing us to the references \cite{Jannsen} and \cite{TotaroCKgP}, which inspired Lemma \ref{cycleclass}.
We thank Sam Payne for explaining the connection between cohomology and point counting. We are grateful to  Rahul Pandharipande, Dan Petersen, and Johannes Schmitt for comments on an earlier version of the manuscript. We also thank Johannes Schmitt for his explanations about and work on admcycles \cite{admcycles}.
We thank Jonas Bergstr\"om and Carel Faber for their comments and for sharing an early version of their work \cite{BergstromFaber} with us. We are especially grateful to Carel Faber for his extensive reading and comments.

\section{Preliminaries on the tautological ring}\label{preliminaries}
\subsection{The stable graph stratification} \label{sgsec}

The moduli space $\Mb_{g,n}$ of stable curves of genus $g$ with $n$ marked points is stratified by topological type. Associated to each topological type is a stable graph $\Gamma$. To a stable graph $\Gamma$, one associates the moduli space $\M_{\Gamma}$ consisting of points in $\Mb_{g,n}$ with dual graph $\Gamma$. The disjoint union
\[
\Mb_{g,n}=\coprod \M_{\Gamma}
\]
is called the \emph{stable graph stratification}.
We will also consider the space
\[
\Mb_{\Gamma}=\prod_{v} \Mb_{g(v),n(v)},
\]
where $v$ is a vertex of the stable graph $\Gamma$. There is an associated morphism
\begin{equation}\label{graphmap}
\xi_{\Gamma}:\Mb_{\Gamma}\rightarrow \Mb_{g,n}
\end{equation}
whose image is the closure of $\M_{\Gamma}$. The spaces $\xi_{\Gamma}(\Mb_{\Gamma})$ are in fact the union of $\M_{\Gamma'}$ for graphs $\Gamma'$
that can be taken to $\Gamma$ via a sequence of edge contractions.
A union of strata $\bigcup_{\Gamma \in S} \M_{\Gamma}$ is open if and only if the set $S$ is closed under contractions. (Contracting as far as possible, one gets to the graph with a single vertex of genus $g$, which is $\M_{g,n}$.)
 
\subsection{Tautological rings}
Roughly speaking, the tautological ring $R^*(\Mb_{g,n})\subset A^*(\Mb_{g,n})$ is the $\qq$-subalgebra of the Chow ring generated by the ``naturally occurring" cycles on $\Mb_{g,n}$. 
There are several natural morphisms between the moduli spaces of stable curves with marked points $\Mb_{g,n}$ as $g$ and $n$ vary. First, there are the forgetful morphisms
\[
\pi_j: \Mb_{g,n}\rightarrow \Mb_{g,n-1}
\]
obtained by forgetting the $j^{\mathrm{th}}$ marking. In addition, there are the gluing morphisms
\[
\xi_\Gamma:\Mb_{\Gamma}\rightarrow \Mb_{g,n}
\]
as in \eqref{graphmap} for stable graphs $\Gamma$.

\begin{definition}\leavevmode \label{tdef}
\begin{enumerate}
    \item The system of tautological rings $\{R^*(\Mb_{g,n})\}$ is the set of smallest $\qq$-subalgebras of the Chow rings $A^*(\Mb_{g,n})$ containing the unit elements $[\Mb_{g,n}]$ and that is closed under the pushforwards by the forgetful morphisms $\pi_j$ and the gluing morphisms $\xi_\Gamma$ for each stable graph $\Gamma$.
    \item Let $U$ be an open substack of $\Mb_{g,n}$. The tautological ring $R^*(U)$ is defined to be the image of $R^*(\Mb_{g,n})$ under the natural restriction map $A^*(\Mb_{g,n})\rightarrow A^*(U)$. 
    \item The tautological ring in cohomology $RH^*(\Mb_{g,n})$ is defined as the image of $R^*(\Mb_{g,n})$ under the cycle class map 
    \[
    A^*(\Mb_{g,n})\rightarrow H^*(\Mb_{g,n},\qq).
    \]
\end{enumerate}
\end{definition}
\noindent
The tautological rings are also closed under pullbacks along the forgetful and gluing maps (see \cite[Section 0.3]{FPGorenstein}).

It is also useful to have a more explicit description of the tautological ring for a given pair $(g,n)$. Let $f:\Cb_{g,n}\rightarrow \Mb_{g,n}$ be the universal curve, which comes equipped with $n$ sections $\sigma_i:\Mb_{g,n}\rightarrow \Cb_{g,n}$, corresponding to the $i^{\mathrm{th}}$ marked point. We define the $\psi$ classes as 
\[
\psi_i=\sigma_i^*c_1(\omega_f).
\]
Under the identification of $\Cb_{g,n}$ with $\Mb_{g,n+1}$, the universal curve map $f:\Cb_{g,n}\rightarrow \Mb_{g,n}$ is identified with the map forgetting the last marked point. Following \cite{ArbarelloCornalbaKappa}, we set
\[
\kappa_j=f_*(\psi_{n+1}^{j+1}).
\]
The $\lambda$ classes are the Chern classes of the Hodge bundle
\begin{equation} \label{ldef}
\lambda_i:=c_i(f_*\omega_f).
\end{equation}
They are expressed in terms of the $\kappa$ classes with rational coefficients using Grothendieck--Riemann--Roch.

Let $\alpha$ be a polynomial in $\psi$ and $\kappa$ classes on $\Mb_{\Gamma}$. The pushforward of $\alpha$ under $\xi_{\Gamma}$ is called a decorated stratum class on $\Mb_{g,n}$. The decorated stratum classes provide generators for $R^*(\Mb_{g,n})$ as a $\qq$-vector space \cite[Proposition 11]{GraberPandharipande}.

\section{The Chow--K\"unneth generation property}\label{CKgP}

Because of the nature of the stable graph stratification, we will want to compute generators for Chow rings of spaces $A^*(\prod_i \M_{g_i,n_i})$. But unlike for cohomology, there is no K\"unneth formula in general for Chow rings, so it is not obvious how the exterior product maps
\[
\bigotimes_i A_*(\M_{g_i,n_i}) \rightarrow A_*(\prod_i \M_{g_i,n_i})
\] 
behave.
In this section, we study when the exterior product map 
\[
A_*(Y)\otimes A_*(X)\rightarrow A_*(Y\times X)
\]
is surjective.
All of the stacks we need to consider later in the paper are finite type over the base field and admit a stratification by global quotient stacks. So throughout this section, we assume that all stacks are of this form. All Deligne--Mumford stacks admit a stratification by global quotient stacks (see \cite[Proposition 4.5.5]{Kresch}).


The following definition is taken from \cite{BaeSchmitt2}.
\begin{definition}[Definition 2.5 of \cite{BaeSchmitt2}] \label{ksurj}
We say $Y$ has the \emph{Chow--K\"unneth generation Property} (CKgP, for short) if for all algebraic stacks $X$ (of finite type and admitting a stratification by global quotient stacks), the exterior product map
    \[A_*(Y) \otimes A_*(X) \to A_*(Y \times X)\]
    is surjective.
\end{definition}
\noindent

\subsection{Basic lemmas}
We collect below several basic lemmas about the CKgP.
\begin{lem} \label{prod}
Suppose that $Y_1$ and $Y_2$ have the CKgP. Then $Y_1 \times Y_2$ has the CKgP. More generally, if $Y_1, \ldots, Y_n$ have CKgP then $Y_1 \times \cdots \times Y_n$ has CKgP.
\end{lem}
\begin{proof}
The first claim is \cite[Lemma 2.8]{BaeSchmitt2}, and the second claim follows by induction on $n$.
\end{proof}

\begin{lem}\label{open}
Suppose $Y$ has the CKgP and $U\subset Y$ is open. Then  $U$ has the CKgP.
\end{lem}
\begin{proof}
For any $X$, we have a commuting square
\begin{center}
\begin{tikzcd}
A_*(Y) \otimes A_*(X) \arrow{d} \arrow{r} & A_*(U) \otimes A_*(X) \arrow{d} \\
A_*(Y \times X) \arrow{r} & A_*(U \times X).
\end{tikzcd}
\end{center}
The composition 
\[
A_*(Y)\otimes A_*(X)\rightarrow A_*(Y\times X)\rightarrow A_*(U\times X)
\]
is surjective by the fact that $Y$ has the CKgP and excision. It follows that the right vertical arrow
is also surjective.
\end{proof}
\begin{lem}\label{Kstrat}
Suppose $Y$ admits a finite stratification $Y=\coprod_{S\in \mathcal{S}} S$ such that each $S\in \mathcal{S}$ has the CKgP. Then $Y$ has the CKgP.
\end{lem}
\begin{proof}
We induct on the size of the stratification, $\#\mathcal{S}$. The case $\#\mathcal{S}=1$ is clear. Let $T$ be a minimal element of $\mathcal{S}$, which exists because $\mathcal{S}$ is finite.  Set $U = Y \smallsetminus T$. Let $X$ be any finite type stack. The following commutative diagram consists of excision sequences, with vertical maps given by the exterior product.
\begin{center}
    \begin{tikzcd}
    A_*(T) \otimes A_*(X) \arrow{r} \arrow{d} & A_*(Y) \otimes A_*(X) \arrow{d} \arrow{r} & A_*(U) \otimes A_*(X) \arrow{d} \arrow{r} & 0\\
    A_*(T \times X) \arrow{r} & A_*(Y \times X) \arrow{r} & A_*(U \times X) \arrow{r} & 0
    \end{tikzcd}
\end{center}
By induction the right vertical map is surjective (using the stratification $\mathcal{S} - T$ of $U$). The left vertical map is surjective because $T$ has the CKgP. A diagram chase shows that the middle map is surjective as well, so $Y$ has the CKgP.
\end{proof}
\begin{lem}\label{affbunCKgP}
Let $\pi:V\rightarrow Y$ be an affine bundle. Then $V$ has the CKgP if and only if $Y$ does.
\end{lem}
\begin{proof}
For any algebraic stack $X$, the map $V\times X\rightarrow Y\times X$ is an affine bundle. We have the following commutative diagram:
\[
\begin{tikzcd}
A_*(V)\otimes A_*(X) \arrow[r]           & A_*(V\times X)           \\
A_*(Y)\otimes A_*(X) \arrow[u] \arrow[r] & A_*(Y\times X) \arrow[u]
\end{tikzcd}
\]
The vertical maps are isomorphisms induced by pullback. Therefore, one horizontal map is surjective if and only if the other is.
\end{proof}

\begin{lem}\label{gerbe}
Let $\pi:\Y\rightarrow Y$ be a gerbe banded by a finite group. Then $Y$ has the CKgP if and only if $\Y$ does.
\end{lem}
\begin{proof}
The proof is the same as Lemma \ref{affbunCKgP}, using that pullback induces an isomorphism on Chow rings for gerbes banded by finite groups.
\end{proof}

\begin{lem}\label{grassmann}
Suppose that $Y$ has the CKgP and that $G\rightarrow Y$ is a Grassmann bundle over $Y$. Then $G$ has the CKgP.
\end{lem}
\begin{proof}
The Chow ring of a Grassmann bundle is generated over the Chow ring of the base by the Chern classes $c_1, \ldots, c_n$ of the tautological subbundle. Consider the diagram
\begin{center}
\begin{tikzcd}
A^*(Y)[c_1, \ldots, c_n] \otimes A^*(X) \arrow{d} \arrow{r} & A^*(Y \times X)[c_1,\ldots, c_n] \arrow{d} \\
A^*(G) \otimes A^*(X) \arrow{r} & A^*(G \times X)
\end{tikzcd}
\end{center}
The top horizontal arrow is surjective because $Y$ has the CKgP. The right vertical arrow is surjective because $G \times X \to Y \times X$ is a Grassmann bundle whose tautological subbundle is the pullback of the tautological subbundle on $G$. It follows that the bottom horizontal arrow is also surjective.
\end{proof}

\begin{lem} \label{bgs}
The following stacks have the CKgP:
\begin{enumerate}
    \item The classifying stack $\BGL_n$,
    \item The classifying stack $\BSL_n$,
    \item The classifying stack $\BPGL_n$.
    
\end{enumerate}
\end{lem}
\begin{proof}
(1) Take $V_k$ be the representation of $\GL_n$ given by matrices of size $n\times k$.  Let $U_k\subset V_k$ be the open subset of full rank matrices. The complement of $U_k$ has codimension $k-n+1$. Then
\[
[U_k/\GL_n]\times X\cong G(n,k)\times X.
\]
Taking $k\to\infty$ and noting that $G(n,k)$ has the CKgP by Lemma \ref{grassmann}, we see that $\BGL_n$ has the CKgP.

(2) The case of $\BSL_n$ is similar, except that the quotient $[U_k/\SL_n]$ is isomorphic to the complement of the zero section of the line bundle $\det \mathcal{S}\rightarrow G(n,k)$, where $\mathcal{S}$ is the tautological subbundle. By Lemmas \ref{open} and \ref{affbunCKgP}, we see $[U_k/\SL_n]$ has CKgP too, and taking $k\to\infty$, we see that $\BSL_n$ has CKgP.

(3) The map $\BSL_n\rightarrow \BPGL_n$ is a gerbe banded by $\mu_n$. Thus, $\BPGL_n$ has the CKgP by (2) and Lemma \ref{gerbe}.
\end{proof}

\begin{lem}
Suppose $Y$ is a stack that admits a coarse moduli space $\pi:Y\rightarrow M$. Then $Y$ has the CKgP if and only if $M$ does.
\end{lem}
\begin{proof}
Let $X$ be a stack that admits a stratification by global quotient stacks. By noetherian induction, we may reduce to the case that $X$ is a global quotient stack. By \cite[Proposition 4.5.6]{Kresch}, for any integer $N$, there exists a vector bundle $E \to X$ and a representable open substack $U \subset E$ whose complement has codimension at least $N$. Formation of coarse moduli spaces commutes with flat base change, so $M \times U$ is a coarse moduli space for $Y \times U$. By \cite[Proposition 6.1]{Vistoli}, there is an isomorphism $A^*(M \times U) \cong  A^*(Y \times U)$. Then, 
for $* < N$, we have
\[A^*(X \times Y) \cong A^*(E \times Y) \cong A^*(U \times Y) \cong A^*(U \times M) \cong A^*(E \times M) \cong A^*(X \times M). \]
Taking $N$ to be arbitrarily large, we have an isomorphism $A^*(X \times Y) \cong A^*(X \times M)$.

Now we have the following commutative diagram, where both horizontal maps are isomorphisms.
\begin{center}
\begin{tikzcd}
A_*(Y) \otimes A_*(X) \arrow{d} \arrow{r} & A_*(M) \otimes A_*(X) \arrow{d} \\
A_*(Y \times X) \arrow{r} & A_*(M \times X).
\end{tikzcd}
\end{center}
If one of the vertical maps is surjective, then so is the other one.
\end{proof}


\begin{lem} \label{surjCKgP}
Let $f:Y\rightarrow X$ be a surjective proper morphism that is representable by Deligne--Mumford stacks. (For example, any quotient by a finite group $Y \to Y/G$.) If $Y$ has the CKgP, then $X$ has the CKgP.
\end{lem}
\begin{proof}
This is a special case of \cite[Lemma 2.9]{BaeSchmitt2}.
\end{proof}


\subsection{The cycle class map}
One important consequence of the CKgP is that, if the space is smooth and proper, the cycle class map is an isomorphism. 
The proof of the following lemma was inspired by the argument in \cite[Theorem 3.6]{Jannsen}, which involves constructing an algebraic decomposition of the diagonal to prove surjectivity of the cycle class map. 

\begin{lem}\label{cycleclass}
Suppose $X$ is a smooth, proper Deligne--Mumford stack of dimension $d$.
If $X$ has the CKgP, then the cycle class map $\mathrm{cl}: A^*(X) \to H^*(X, \qq)$ is an isomorphism.
\end{lem}
\begin{proof} Let $p_1: X \times X \to X$ and $p_2 : X \times X \to X$ be the projection maps onto the first and second factors.
Because $X$ has the CKgP, the map $\bigoplus_{i=0}^d A^i(X) \otimes A^{d-i}(X) \to A^d(X \times X)$ is surjective. In particular, the class of the diagonal $[\Delta] \in A^d(X \times X)$ has the form
\begin{equation} \label{inchow} [\Delta] = \sum_{i=0}^d \sum_{j=0}^{a_i} n_{i,j} \cdot (p_1^*y_{i,j}) \cdot (p_2^*z_{i,j}) \in A^d(X \times X). 
\end{equation}
where $y_{i,j} \in A^i(X)$ and $z_{i,j} \in A^{d-i}(X)$. Applying the cycle class map gives a decomposition of the diagonal in cohomology:
\[[\Delta] = \sum_{i=0}^d n_{i,j} \cdot p_1^*\mathrm{cl}(y_{i,j}) \cdot p_2^*\mathrm{cl}(z_{i,j}) \in H^{2d}(X \times X). \]
We use this to show that all cohomology is algebraic. Given $\alpha \in H^n(X)$, we have
\begin{align*}\alpha &= p_{2*}(p_1^*\alpha \cdot [\Delta])  \\
&= p_{2*}\left(\sum_{i=0}^d \sum_{j=0}^{a_i} p_1^*\alpha \cdot p_1^*\mathrm{cl}(y_{i,j}) \cdot p_2^*\mathrm{cl}(z_{i,j})\right) \\
&=  \sum_{i=0}^d \sum_{j=0}^{a_i} p_{2*}p_1^*(\alpha \cdot \mathrm{cl}(y_{i,j}))
\cdot \mathrm{cl}(z_{i,j}).
\intertext{For dimension reasons, $\alpha \cdot \cl(y_{i,j}) = 0$ whenever $n + 2i > 2d$.
On the other hand,
$p_{2*}p_1^*(\alpha \cdot \mathrm{cl}(y_i))$ has codimension $n + 2i - 2d$, so it vanishes if $n + 2i < 2d$. It follows that the only non-zero terms in the sum above are those with $n + 2i = 2d$. In this case, $\alpha \cdot \mathrm{cl}(y_{i,j}) \in H^{2d}(X) = \qq$ is top degree, so we get}
&= \sum_{j=0}^{a_i} \deg(\alpha \cdot \mathrm{cl}(y_{i,j})) \cdot \mathrm{cl}(z_{i,j}),
\end{align*}
showing $\alpha$ is in the image of $\mathrm{cl}$.

For injectivity, let $\beta\in A^i(X)$ be in the kernel of $\mathrm{cl}$. Then for any class $\gamma\in A^{d-i}(X)$, we have
\[
0=\mathrm{cl}(\beta)\cup \mathrm{cl}(\gamma)=\mathrm{cl}(\beta\cdot \gamma)=\deg(\beta\cdot \gamma).
\]
Applying the diagonal correspondence in Chow \eqref{inchow}, and using the push--pull formula as in the proof of surjectivity, we then have
\[
\beta=\sum_{j=0}^{a_i} \deg(\beta \cdot y_{i,j}) \cdot z_{i,j}=0. \qedhere
\]
\end{proof}
\begin{rem}
In \cite[Theorem 4.1]{TotaroCKgP}, Totaro proves a stronger version of Lemma \ref{cycleclass} for Chow motives, and in particular, for smooth, proper schemes over a field. 
\end{rem}
\begin{rem}
Properness of $X$ is essential and was used in the proof  to ensure the pushforward map $p_{2*}: A_*(X \times X) \to A_*(X)$ was well-defined.
For a non-example if the properness assumption is dropped, $\mathbb{A}^1 \smallsetminus \{0\}$ has the CKgP, but $H^1(\mathbb{A}^1 \smallsetminus \{0\}) \neq 0$, so $\mathrm{cl}$ is not surjective.
\end{rem}

As an immediate consequence of Lemma \ref{cycleclass}, we see that the CKgP allows us to convert results in Chow to results in cohomology.
\begin{lem} \label{arh}
Suppose $\Mb_{g,n}$ has the CKgP and $A^*(\Mb_{g,n}) = R^*(\Mb_{g,n})$. Then we also have $H^*(\Mb_{g,n}) = RH^*(\Mb_{g,n})$.
\end{lem}

\subsubsection{Tautological cohomology and polynomial point counts}\label{ppc}
By the Grothendieck--Lefschetz trace formula, the cohomology of $\Mb_{g,n}$ is intimately related to its point counts over finite fields. Given a Deligne--Mumford stack $X$ that is smooth and proper over $\zz$, we say $X$ has \emph{polynomial point count} if $\# X(\mathbb{F}_q) = P(q)$ for some polynomial $P$ with integer coefficients. Here, the count of points is weighted by size of the automorphism group of each point:
\[\# X(k) = \sum_{\xi \in X(k)} \frac{1}{|\mathrm{Aut}(\xi)|}.\]
If $X$ has polynomial point count, then van den Bogaart--Edixhoven \cite{BogaartEdixhoven} show that the polynomial $P$ determines the cohomology of $X$.
Conversely, information about the cohomology of $X$ 
can determine information about its point counts.
The following fact was explained to us by Sam Payne.

\begin{prop} \label{ptcounts}
If $H^*(\Mb_{g,n}) = RH^*(\Mb_{g,n})$, then $\Mb_{g,n}$ has polynomial point count.
\end{prop}

\begin{proof}
This result relies crucially on the fact that $X := \Mb_{g,n}$ is smooth and proper over $\zz$. 
The $\mathbb{F}_q$-points of $X$ are the fixed points of the action of Frobenius on $X_{\overline{\mathbb{F}}_q}$.
Using the Grothendieck--Lefschetz trace formula \cite{Behrend}, $X$ has polynomial point count if Frobenius acts by powers of $q$ on
$H^*_{\text{\'et}}(X_{\overline{\mathbb{F}}_q}, \qq_\ell)$ for $\ell$ a prime not equal to $p$.
Frobenius always acts by $q^k$ on the fundamental class of an algebraic subvariety of codimension $k$ defined over $\mathbb{F}_q$.
Thus, it suffices to show that
$H^*_{\text{\'et}}(X_{\overline{\mathbb{F}}_q}, \qq_\ell)$ is generated by algebraic classes defined over $\mathbb{F}_p$.

The tautological ring is generated by algebraic classes defined over $\zz$. Therefore, there is a surjection $R^*(X_{\zz}) \to RH^*(X_{\cc})$. Assuming $H^*(X_{\cc}) = RH^*(X_{\cc})$, we obtain the diagram below where the  arrow from the left to the far upper right is surjective.
\begin{equation} \label{pc}
\begin{tikzcd}
& A^*(X_{\cc}) \arrow{r} & H^*(X_{\cc}) \\
R^*(X_{\zz}) \ar[urr,bend right = 10, two heads] \arrow{ur} \arrow{dr} \\
& A^*(X_{\mathbb{F}_p}) \arrow{r} & H^*_{\text{\'et}}(X_{\overline{\mathbb{F}}_p}, \qq_\ell)
\end{tikzcd}
\end{equation}
Because $X$ is smooth over $\zz$, we have (see \cite[Proposition 3.1]{BogaartEdixhoven})
\[H^*_{\text{\'et}}(X_{\overline{\mathbb{F}}_p}, \qq_\ell) \cong
H^*_{\text{\'et}}(X_{\overline{\mathbb{Q}}_p}, \qq_\ell). 
\]
Meanwhile, choosing an isomorphism $\overline{\qq}_p \cong \mathbb{C}$, the comparison theorem gives an isomorphism
\[ H^*_{\text{\'et}}(X_{\overline{\mathbb{Q}}_p}, \qq_\ell) \cong H^*(X_{\mathbb{C}}, \qq_{\ell}). \]
Thus, upon tensoring the vector spaces in \eqref{pc} up to $\qq_\ell$, there is an isomorphism between the two vector spaces in the rightmost column. 
This implies that the image of $R^*(X_{\zz}) \to H^*_{\text{\'et}}(X_{\overline{\mathbb{F}}_p}, \qq_\ell)$ generates $H^*_{\text{\'et}}(X_{\overline{\mathbb{F}}_p}, \qq_\ell)$ over $\qq_\ell$.
Hence, the image of $A^*(X_{\mathbb{F}_p}) \to H^*_{\text{\'et}}(X_{\overline{\mathbb{F}}_p}, \qq_{\ell})$ generates 
$H^*_{\text{\'et}}(X_{\overline{\mathbb{F}}_p}, \qq_{\ell})$ over $\qq_\ell$. Frobenius acts on such elements by powers of $q$, so we obtain a polynomial point count.
\end{proof}

By Proposition \ref{ptcounts},
Corollary \ref{ptcor} follows immediately from Theorem \ref{main}. Thus, our work gives new proofs of all of the previously known cases establishing $\Mb_{g,n}$ has polynomial point count listed in Section \ref{previouswork}. Unlike previous work, we have not determined the polynomial, arguing instead just with our knowledge of generators for cohomology. However, the polynomial could in principle be determined by using Corollary \ref{gcor} and admcycles \cite{admcycles} to compute the ranks of $RH^*(\Mb_{g,n})$.

\section{Filling criteria} \label{mainstrategy}
Given any open subset $U \subseteq \Mb_{g,n}$ we say ``$U$ has $A^* = R^*$" when $A^*(U)$ is generated by restrictions of tautological classes on $\Mb_{g,n}$. We keep track of our progress proving $A^* = R^*$ results with circles on the $(g, n)$ grid as in Figure \ref{bigfig}:
\begin{itemize}
    \item[$\circ$] An open circle at coordinate $(g, n)$ means $\M_{g,n}$ has the CKgP and $A^* = R^*$.
    \item[ \put(.5,2.4){\circle*{5}}\put(.5,2.4){\color{white}\circle*{2}} \ ] A thick open circle at coordinate $(g, n)$ means $\M_{g,n}^{\ct}$ has the CKgP and $A^* = R^*$.
    \item A filled circle at coordinate $(g, n)$ means $\Mb_{g,n}$ has the CKgP and $A^* = R^*$.
\end{itemize} 
Filled circles are stronger than thick open circles are stronger than open circles: By Lemma \ref{open} and excision if one of these $\M$ has the CKgP and $A^* = R^*$, then any open subset of $\M$ has  the CKgP and $A^* = R^*$. By Lemma \ref{arh}, a filled circle $\bullet$ implies what it meant in Figure \ref{bigfig} of the introduction.

Considering the forgetful maps $\Mb_{g,n} \to \Mb_{g,n-1}$, which are surjective and proper, we see that 
filled circles always come in columns where all circles below a filled circle are also filled.
Indeed, Lemma \ref{surjCKgP} shows that if $\Mb_{g,n}$ has CKgP then $\Mb_{g,n-1}$ has CKgP. Meanwhile, the push forward map induces a surjection on Chow groups and sends tautological classes to tautological classes, so if $\Mb_{g,n}$ has $A^* = R^*$ then $\Mb_{g,n-1}$ has $A^* = R^*$.

The inductive nature of the boundary of $\Mb_{g,n}$ allows us to convert open circles to filled circles when certain circles to the left (i.e. for lower $g$) are already filled.

\begin{lem}[Filling criterion, version 1] \label{fc1}
Suppose that 
\begin{enumerate}
    \item $\M_{g,n'}$ has the CKgP and $A^* = R^*$ for all $n' \leq n$ (we have a column of $n$ open circles in genus $g$).
\item $\Mb_{g',n'}$ has the CKgP and $A^* = R^*$ for all $g' \leq g-1$ and $n'\leq n+1$ (all dots in the rectangular region to the left going one row higher are filled)
\item $\Mb_{g-1,n+2}$ has the CKgP and $A^* = R^*$ (the dot up two and one to the left is filled).
\end{enumerate}
Then $\Mb_{g,n'}$ has the CKgP and $A^* = R^*$ for all $n' \leq n$ (we get a column of $n$ filled circles in genus $g$).
\end{lem}
\begin{proof}
Inducting on $n$, we may assume that $\Mb_{g,n'}$ has the CKgP for $n' \leq n-1$. 
Each component of the boundary of $\Mb_{g,n}$ is the image of
a gluing map
\[\Mb_{g_1,n_1+1} \times \Mb_{g_2,n_2+1} \to \Mb_{g,n} \qquad \text{or} \qquad \Mb_{g-1,n+2} \to \Mb_{g,n}.\]
where $g_1 + g_2 = g$ and $n_1 + n_2 = n$.
Note that if $g_1 = g$, then 
$g_2 = 0$, so $n_2 \geq 2$ and hence $n_1 + 1 \leq n-1$.
Thus, the assumptions (1) -- (3) and our inductive hypothesis ensure that all of the relevant $\Mb_{g',n'}$ used above have the CKgP and $A^* = R^*$. 
By Lemmas \ref{prod} and \ref{surjCKgP}, we see that every component of the boundary of $\Mb_{g,n}$ has CKgP. To see classes supported there are tautological,
 consider the diagram
\begin{center}
    \begin{tikzcd}
    R^*(\Mb_{g_1, n_1+1}) \otimes R^*(\Mb_{g_2,n_2+1}) \ar[r, two heads] & A^*(\Mb_{g_1, n_1+1}) \otimes A^*(\Mb_{g_2,n_2+1}) \ar[d, two heads] \\
    & A^*(\Mb_{g_1, n_1+1} \times \Mb_{g_2,n_2+1}) \arrow{r} & A^*(\Mb_{g,n}).
    \end{tikzcd}
\end{center}
The first horizontal arrow is surjective because the moduli spaces have $A^* = R^*$. The vertical arrow is surjective because they have the CKgP. 
The image of the last horizontal arrow is all cycles supported on this boundary component. Because the first two arrows are surjections, the image of the last map is the same as the image of the composition. Finally, the image of the composition consists of tautological classes by the definition of the tautological rings.

A similar argument holds for the self-glue boundary stratum. We have
\[R^*(\Mb_{g-1,n+2}) \twoheadrightarrow A^*(\Mb_{g-1,n+2}) \rightarrow A^*(\Mb_{g,n}). \]
The image of the second map is all cycles supported on this boundary stratum. The image of the composition consists of tautological classes by definition of the tautological rings.
\end{proof}

Iterated use of Lemma \ref{fc1} tells us that if we have all open circles in a
region bounded by the axes, a vertical line, and a line of slope $-2$, then they can all be converted to filled circles.

\begin{center}
\begin{tikzpicture}[scale = .4]

\node[scale=.5] at (1, -.7) {$1$};
\node[scale=.5] at (2, -.7) {$2$};
\node[scale=.5] at (3, -.7) {$3$};
\node[scale=.5] at (4, -.7) {$4$};

\node[scale=.5] at (-.7, 1) {$1$};
\node[scale=.5] at (-.7, 2) {$2$};
\node[scale=.5] at (-.7, 3) {$3$};
\node[scale=.5] at (-.7, 4) {$4$};
\node[scale=.5] at (-.7, 5) {$5$};
\node[scale=.5] at (-.7, 6) {$6$};
\node[scale=.5] at (-.7, 7) {$7$};
\node[scale=.5] at (-.7, 8) {$8$};
\node[scale=.5] at (-.7, 9) {$9$};

\draw[->] (0, 0) -- (5, 0);
\draw[->] (0, 0) -- (0, 10);
\node[scale=.9] at (5.4,0) {$g$};
\node[scale=.9] at (0, 10.4) {$n$};
\draw (-.1, 1) -- (.1, 1);
\draw (-.1, 2) -- (.1, 2);
\draw (1, -.1) -- (1, .1);
\draw (2, -.1) -- (2, .1);
\draw (3, -.1) -- (3, .1);
\draw (4, -.1) -- (4, .1);
\draw[color=gray] (0, 9) -- (3,3);
\draw[color=gray] (3,3) -- (3, 0);
\filldraw[color=white] (0, 3) circle (4pt);
\filldraw[color=white] (0, 4) circle (4pt);
\filldraw[color=white] (0, 5) circle (4pt);
\filldraw[color=white] (0, 6) circle (4pt);
\filldraw[color=white] (0, 7) circle (4pt);
\filldraw[color=white] (0, 8) circle (4pt);
\filldraw[color=white] (0, 9) circle (4pt);
\draw (0, 3) circle (4pt);
\draw (0, 4) circle (4pt);
\draw (0, 5) circle (4pt);
\draw (0, 6) circle (4pt);
\draw (0, 7) circle (4pt);
\draw (0, 8) circle (4pt);
\filldraw[color=white] (0, 9) circle (4pt);
\draw (0, 9) circle (4pt);
\draw (1, 1) circle (4pt);
\draw (1, 2) circle (4pt);
\draw (1, 3) circle (4pt);
\draw (1, 4) circle (4pt);
\draw (1, 5) circle (4pt);
\draw (1, 6) circle (4pt);
\filldraw[color=white] (1, 7) circle (4pt);
\draw (1, 7) circle (4pt);
\filldraw[color=white] (2, 0) circle (4pt);
\draw (2, 0) circle (4pt);
\draw (2, 1) circle (4pt);
\draw (2, 2) circle (4pt);
\draw (2, 3) circle (4pt);
\draw (2, 4) circle (4pt);
\filldraw[color=white] (2, 5) circle (4pt);
\draw (2, 5) circle (4pt);
\filldraw[color=white] (3, 0) circle (4pt);
\draw (3, 0) circle (4pt);
\filldraw[color=white] (3, 1) circle (4pt);
\draw (3, 1) circle (4pt);
\filldraw[color=white] (3, 2) circle (4pt);
\draw (3, 2) circle (4pt);
\filldraw[color=white] (3, 3) circle (4pt);
\draw (3, 3) circle (4pt);
\draw[->] (5, 5) -- (9, 5);
\node[scale=.8] at (7, 6) {Filling criterion};
\node[scale=.8] at (7, 4) {version 1};
\end{tikzpicture}
\hspace{.07in}
\begin{tikzpicture}[scale = .4]
\draw[color=gray] (0, 9) -- (3,3);
\draw[color=gray] (3,3) -- (3, 0);
\node[scale=.5] at (1, -.7) {$1$};
\node[scale=.5] at (2, -.7) {$2$};
\node[scale=.5] at (3, -.7) {$3$};
\node[scale=.5] at (4, -.7) {$4$};

\node[scale=.5] at (-.7, 1) {$1$};
\node[scale=.5] at (-.7, 2) {$2$};
\node[scale=.5] at (-.7, 3) {$3$};
\node[scale=.5] at (-.7, 4) {$4$};
\node[scale=.5] at (-.7, 5) {$5$};
\node[scale=.5] at (-.7, 6) {$6$};
\node[scale=.5] at (-.7, 7) {$7$};
\node[scale=.5] at (-.7, 8) {$8$};
\node[scale=.5] at (-.7, 9) {$9$};

\draw[->] (0, 0) -- (5, 0);
\draw[->] (0, 0) -- (0, 10);
\node[scale=.9] at (5.4,0) {$g$};
\node[scale=.9] at (0, 10.4) {$n$};
\draw (-.1, 1) -- (.1, 1);
\draw (-.1, 2) -- (.1, 2);
\draw (1, -.1) -- (1, .1);
\draw (2, -.1) -- (2, .1);
\draw (3, -.1) -- (3, .1);
\draw (4, -.1) -- (4, .1);

\filldraw (0, 3) circle (4pt);
\filldraw (0, 4) circle (4pt);
\filldraw (0, 5) circle (4pt);
\filldraw (0, 6) circle (4pt);
\filldraw (0, 7) circle (4pt);
\filldraw (0, 8) circle (4pt);
\filldraw (0, 9) circle (4pt);
\filldraw (1, 1) circle (4pt);
\filldraw (1, 2) circle (4pt);
\filldraw (1, 3) circle (4pt);
\filldraw (1, 4) circle (4pt);
\filldraw (1, 5) circle (4pt);
\filldraw (1, 6) circle (4pt);
\filldraw (1, 7) circle (4pt);
\filldraw (2, 0) circle (4pt);
\filldraw (2, 1) circle (4pt);
\filldraw (2, 2) circle (4pt);
\filldraw (2, 3) circle (4pt);
\filldraw (2, 4) circle (4pt);
\filldraw (2, 5) circle (4pt);
\filldraw (3, 0) circle (4pt);
\filldraw (3, 1) circle (4pt);
\filldraw (3, 2) circle (4pt);
\filldraw (3, 3) circle (4pt);
\end{tikzpicture}
\end{center}

Since $\Mb_{1,11}$ fails $A^* = R^*$, the largest region of this form is the triangle below to the right, 
bounded by the axes and the line of slope $-2$ through $\Mb_{1,10}$, i.e. $2g + n \leq 12$.

\noindent
\begin{minipage}{.5\textwidth}
\ \ \ We shall obtain all filled circles in this region. However, we also obtain some results \emph{above} this maximal line of slope $-2$, i.e. with $2g + n > 12$.

\ \ \ The basic motivation behind such improvements is that additional symmetry is introduced when we self-glue, so even if Lemma \ref{fc1}(3) fails, there is still hope that classes supported on the image of $\Mb_{g-1,n+1} \to \Mb_{g,n}$ are tautological.
Let $\Delta_{\mathrm{dn}}$ be the union of all boundary strata with a disconnecting node. 
An alternative criterion is as follows.
\end{minipage}
\begin{minipage}{.49\textwidth}
\centering
\begin{tikzpicture}[scale = .4]
\node[scale=.5] at (1, -.7) {$1$};
\node[scale=.5] at (2, -.7) {$2$};
\node[scale=.5] at (3, -.7) {$3$};
\node[scale=.5] at (4, -.7) {$4$};
\node[scale=.5] at (5, -.7) {$5$};
\node[scale=.5] at (6, -.7) {$6$};

\node[scale=.5] at (-.7, 1) {$1$};
\node[scale=.5] at (-.7, 2) {$2$};
\node[scale=.5] at (-.7, 3) {$3$};
\node[scale=.5] at (-.7, 4) {$4$};
\node[scale=.5] at (-.7, 5) {$5$};
\node[scale=.5] at (-.7, 6) {$6$};
\node[scale=.5] at (-.7, 7) {$7$};
\node[scale=.5] at (-.7, 8) {$8$};
\node[scale=.5] at (-.7, 9) {$9$};
\node[scale=.5] at (-.7, 10) {$10$};
\node[scale=.5] at (-.7, 11) {$11$};
\node[scale=.5] at (-.7, 12) {$12$};

\draw[color=gray] (0,12) -- (6, 0);

\draw[->] (0, 0) -- (7, 0);
\draw[->] (0, 0) -- (0, 13);
\node[scale=.9] at (7.4,0) {$g$};
\node[scale=.9] at (0, 13.4) {$n$};
\draw (-.1, 1) -- (.1, 1);
\draw (-.1, 2) -- (.1, 2);
\draw (1, -.1) -- (1, .1);
\draw (2, -.1) -- (2, .1);
\draw (3, -.1) -- (3, .1);
\draw (4, -.1) -- (4, .1);
\filldraw (0, 3) circle (4pt);
\filldraw (0, 4) circle (4pt);
\filldraw (0, 5) circle (4pt);
\filldraw (0, 6) circle (4pt);
\filldraw (0, 7) circle (4pt);
\filldraw (0, 8) circle (4pt);
\filldraw (0, 9) circle (4pt);
\filldraw (0, 10) circle (4pt);
\filldraw (0, 11) circle (4pt);
\filldraw (0, 12) circle (4pt);
\filldraw (1, 1) circle (4pt);
\filldraw (1, 2) circle (4pt);
\filldraw (1, 3) circle (4pt);
\filldraw (1, 4) circle (4pt);
\filldraw (1, 5) circle (4pt);
\filldraw (1, 6) circle (4pt);
\filldraw (1, 7) circle (4pt);
\filldraw (1, 8) circle (4pt);
\filldraw (1, 9) circle (4pt);
\filldraw (1, 10) circle (4pt);
\node[scale = .6, color = red] at (1, 11) {$\mathbf{\times}$};
\filldraw (2, 0) circle (4pt);
\filldraw (2, 1) circle (4pt);
\filldraw (2, 2) circle (4pt);
\filldraw (2, 3) circle (4pt);
\filldraw (2, 4) circle (4pt);
\filldraw (2, 5) circle (4pt);
\filldraw (2, 6) circle (4pt);
\filldraw (2, 7) circle (4pt);
\filldraw (2, 8) circle (4pt);

\filldraw (3, 0) circle (4pt);
\filldraw (3, 1) circle (4pt);
\filldraw (3, 2) circle (4pt);
\filldraw (3, 3) circle (4pt);
\filldraw (3, 4) circle (4pt);
\filldraw (3, 5) circle (4pt);
\filldraw (3, 6) circle (4pt);

\filldraw (4, 0) circle (4pt);
\filldraw (4, 1) circle (4pt);
\filldraw (4, 2) circle (4pt);
\filldraw (4, 3) circle (4pt);
\filldraw (4, 4) circle (4pt);

\filldraw (5, 0) circle (4pt);
\filldraw (5, 1) circle (4pt);
\filldraw (5, 2) circle (4pt);

\filldraw (6, 0) circle (4pt);
\end{tikzpicture}
\label{trifig}
\end{minipage}

\begin{lem}[Filling criterion, version 2] \label{fc2}
Suppose that 
\begin{enumerate}
    \item $\M_{g,n'}$ has the CKgP and $A^* = R^*$ for all $n' \leq n$ (we have a column of $n$ open circles in genus $g$).
\item $\Mb_{g',n'}$ has the CKgP and $A^* = R^*$ for all $g' \leq g-1$ and $n'\leq n+1$ (all dots in the rectangular region to the left going one row higher are filled)
\item $\Mb_{g,n} \smallsetminus \Delta_{\mathrm{dn}}$ has the CKgP and $A^* = R^*$.
\end{enumerate}
Then $\Mb_{g,n'}$ has the CKgP and $A^* = R^*$ for all $n' \leq n$ (we get a column of $n$ filled circles in genus $g$).
\end{lem}
\begin{proof}
Arguing as in Lemma \ref{fc1}, assumptions (1) and (2) ensure that all classes supported on boundary components for curves with a disconnecting node are tautological and such boundary components have the CKgP. Using excision, part (3) completes the result.
\end{proof}

Our results with $2g + n > 12$ start with $(g, n) = (2, 9)$ and rely on version 2 of the filling criterion (see Lemma \ref{2end}). We also access $\Mb_{3,8}$ using version 2 of our filling criterion. In both of these cases, the key idea is to realize that our construction of $\M_{g,n}$ naturally extends to cover a slightly larger open subset $U$ of $\Mb_{g,n}$; then we make adhoc arguments to cover anything we have missed in $\Mb_{g,n} \smallsetminus (\Delta_{\mathrm{dn}} \cup U)$.
Once we obtain these extra filled circles, Lemma \ref{fc1} has a ``ripple effect" which allows us to fill in other open circles to the right along the line of slope $-2$ through $(3,8)$, eventually resulting in a filled circle for $\Mb_{7}$. 

As the proof of version 2 of the filling criterion suggests, if we focus on the curves where all nodes are disconnecting --- in other words the open locus $\M_{g,n}^{\ct}$ --- then we can convert open circles to thick circles along lines of higher slope.

\begin{lem}[Thickening criterion, for compact type] \label{tc}
Suppose that
\begin{enumerate}
    \item $\M_{g,n'}$ has the CKgP and $A^* = R^*$ for all $n' \leq n$ (we have a column of $n$ open circles in genus $g$).
\item $\M_{g',n'}^{\ct}$ has the CKgP and $A^* = R^*$ for all $g' \leq g-1$ and $n'\leq n+1$ (all dots in the rectangular region to the left going one row higher are thick circles)
\end{enumerate}
Then $\M_{g,n'}^{\ct}$ has the CKgP and $A^* = R^*$ for all $n' \leq n$ (we get a column of $n$ thick circles in genus $g$).
\end{lem}
\begin{proof}
As in Lemmas \ref{fc1} and \ref{fc2}, the assumptions (1) and (2) guarantee that all boundary components of $\M_{g,n}^{\ct}$ are images under gluing maps of products of moduli spaces that satisfy the CKgP and $A^* = R^*$.
\end{proof}

\begin{rem}[Column of open circles also gives rational tails]
We shall soon see that $\Mb_{0, n}$ has the CKgP and $A^* = R^*$ for all $n$ (Section \ref{g0}). Each component of the boundary of $\M_{g,n}^{\rt}$ is the image of $\M_{g,n_1+1}^{\rt} \times \Mb_{0,n_2+1}$
for some $n_1 + n_2 = n$ with $n_2 \geq 2$.
Inducting on $n$, it follows that any time we have a column of open circles for $\M_{g,n'}$ with $n' \leq n$, we also get $A^* = R^*$ and CKgP for $\M_{g,n'}^{\rt}$ with $n' \leq n$.
\end{rem}


\section{Genus $0$ and $1$}\label{zeroone}

In this section, we give a brief overview of past results that have completed the genus $0$ and $1$ columns of our chart. 

\subsection{Genus $0$} \label{g0}
In genus $0$,  we always have $n\geq3$. By an automorphism of $\p^1$, we can fix the first three marked points at $0, 1, \infty$. Then, $\M_{0,n}$ is isomorphic to a complement of hyperplanes in $\mathbb{A}^{n-3}$. Therefore, $A^*(\M_{0,n})$ is generated by the fundamental class $[\M_{0,n}]$, so $A^*(\M_{0,n})=R^*(\M_{0,n})$. We also see that $\M_{0,n}$ has the CKgP by Lemma \ref{open}. The filling criteria of Lemma \ref{fc1} are satisfied (we just have a column of open circles and nothing to the left). Inducting on $n$, we see that $\Mb_{0,n}$ has CKgP and $A^* = R^*$ for all $n\geq 3$. The structure of $A^*(\Mb_{0,n})$ has been determined (integrally, in fact) by Keel \cite{Keel}.

\subsection{Genus $1$} \label{g1}
In order to set up inductive arguments using the strategy of Section \ref{mainstrategy},
we want to show for each space $\M_{1,n}$ with $1\leq n\leq 10$ that the following two conditions are satisfied:
\begin{enumerate}
    \item The Chow ring of $\M_{1,n}$ is generated by tautological classes.
    \item $\M_{1,n}$ has the CKgP.
\end{enumerate}
Applying Lemma \ref{fc1} and the genus $0$ case above, these conditions imply that $\Mb_{1,n}$ has the CKgP and $A^* = R^*$ for $1 \leq n \leq 10$.

To start, Belorousski has computed the Chow ring of $\M_{1,n}$, $1\leq n \leq 10$.
\begin{thm}[Belorousski \cite{Belorousski}]\label{Belorousski}
For $1\leq n \leq 10$, we have $A^*(\M_{1,n})=R^*(\M_{1,n}) =\qq$. 
\end{thm}
What remains is to show that $\M_{1,n}$ has the CKgP. Although not stated explicitly in \cite{Belorousski}, the CKgP actually follows from Belorousski's proof of Theorem \ref{Belorousski}, as we now summarize.

\vspace{.1in}
\noindent
\begin{minipage}{.65\textwidth}
\subsection*{$n=1$}
The coarse moduli space of $\M_{1,1}$ is isomorphic to $\mathbb{A}^1$. Hence, it has the CKgP by Lemma \ref{affbunCKgP}.

\vspace{.1in}
\subsection*{$n=2$}
Mumford \cite{Mumford} proved that the coarse space of $\M_{1,2}$ is isomorphic to a quotient $U/S_3$ where $U\subset \A^2$ is open. It follows from Lemma \ref{surjCKgP} that $\M_{1,2}$ has the CKgP.

\vspace{.1in}
\subsection*{$3 \leq n\leq 10$}
Belorousski finds divisors $D_i\subset \M_{1,n}$ for each such $n$ so that 
\begin{enumerate}
    \item Each $D_i$ is isomorphic to an open subset of $\M_{1,n-1}$.
    \item The complement $\M_{1,n}\smallsetminus \bigcup D_i$ is the image under a proper map of an open subset of (a product of) projective spaces.
\end{enumerate}
It thus follows by induction and Lemmas \ref{prod}, \ref{open}, \ref{Kstrat}, and \ref{grassmann} that each $\M_{1,n}$ has the CKgP for $n\leq 10$.
\end{minipage}
\begin{minipage}{.33\textwidth}
\begin{center}
\begin{tikzpicture}[scale = .4]

\node[scale=.5] at (1, -.7) {$1$};
\node[scale=.5] at (2, -.7) {$2$};
\node[scale=.5] at (3, -.7) {$3$};
\node[scale=.5] at (4, -.7) {$4$};

\node[scale=.5] at (-.7, 1) {$1$};
\node[scale=.5] at (-.7, 2) {$2$};
\node[scale=.5] at (-.7, 3) {$3$};
\node[scale=.5] at (-.7, 4) {$4$};
\node[scale=.5] at (-.7, 5) {$5$};
\node[scale=.5] at (-.7, 6) {$6$};
\node[scale=.5] at (-.7, 7) {$7$};
\node[scale=.5] at (-.7, 8) {$8$};
\node[scale=.5] at (-.7, 9) {$9$};
\node[scale=.5] at (-.7, 10) {$10$};
\node[scale=.5] at (-.7, 11) {$11$};
\node[scale=.5] at (-.7, 12) {$12$};
\node[scale=.5] at (-.7, 13) {$13$};

\draw[->] (0, 0) -- (5, 0);
\draw[->] (0, 0) -- (0, 14);
\node[scale=.9] at (5.4,0) {$g$};
\node[scale=.9] at (0, 14.4) {$n$};
\draw (-.1, 1) -- (.1, 1);
\draw (-.1, 2) -- (.1, 2);
\draw (1, -.1) -- (1, .1);
\draw (2, -.1) -- (2, .1);
\draw (3, -.1) -- (3, .1);
\draw (4, -.1) -- (4, .1);
\filldraw (0, 3) circle (4pt);
\filldraw (0, 4) circle (4pt);
\filldraw (0, 5) circle (4pt);
\filldraw (0, 6) circle (4pt);
\filldraw (0, 7) circle (4pt);
\filldraw (0, 8) circle (4pt);
\filldraw (0, 9) circle (4pt);
\filldraw (0, 10) circle (4pt);
\filldraw (0, 11) circle (4pt);
\filldraw (0, 12) circle (4pt);
\filldraw (0, 13) circle (4pt);
\filldraw (1, 1) circle (4pt);
\filldraw (1, 2) circle (4pt);
\filldraw (1, 3) circle (4pt);
\filldraw (1, 4) circle (4pt);
\filldraw (1, 5) circle (4pt);
\filldraw (1, 6) circle (4pt);
\filldraw (1, 7) circle (4pt);
\filldraw (1, 8) circle (4pt);
\filldraw (1, 9) circle (4pt);
\filldraw (1, 10) circle (4pt);
\node[scale = .6, color = red] at (1, 11) {$\times$};
\node[scale = .6, color = red] at (1, 12) {$\times$};
\node[scale = .6, color = red] at (1, 13) {$\times$};
\end{tikzpicture}

Our starting point 
\end{center}
\end{minipage}

\subsection*{$n\geq 11$}
The strategy fails for $n\geq 11$ because $\Mb_{1,11}$ has a holomorphic differential form (see \cite[Section 2.3]{FPhandbook}). By a Theorem of Ro\u{\i}tman \cite{Roitman}, this implies that $A_0(\Mb_{1,n})$ is uncountable for $n\geq 11$, and hence $A_0(\Mb_{1,n})\neq R_0(\Mb_{1,n})$ for $n\geq 11$.

\section{Genus 2 and hyperelliptic curves in general} \label{hypsec}
In previous work \cite{Hyperelliptic}, we computed the Chow rings of moduli stacks $\mathrm{Hyp}_{g,n}$ of smooth pointed hyperelliptic curves of genus $g\geq 2$ with $n\leq 2g+6$ marked points. 
\begin{thm}[Corollary 1.1 of \cite{Hyperelliptic}]\label{hyperelliptic}
For $n\leq 2g+6$, we have $A^*(\mathrm{Hyp}_{g,n})$ is generated by the $\psi$ classes.
\end{thm}

\noindent
\begin{minipage}{.45\textwidth}
\ \ \ In particular, we see that $\M_{2,n}$ has the CKgP and $A^* = R^*$ for all $n \leq 10$. This gives the new blue column of open circles in the first chart to the right. Lemma \ref{fc1} and the results of Section \ref{zeroone} immediately allow us to fill in the circles with $g = 2$ and $n \leq 8$. 
However,
because $\Mb_{1,n}$ does \emph{not} have $A^* = R^*$ for $n \geq 11$,  Lemma \ref{fc1}(3) fails for filling in $\Mb_{2,n}$ with $n \geq 9$. To fill $\Mb_{2,9}$, we require an extension of Theorem \ref{hyperelliptic} to a partial compactification of $\mathrm{Hyp}_{g,n}$, which was also proved in \cite{Hyperelliptic}. 

\hspace{.1in} Let $\I_{g,n} \subseteq \Mb_{g,n}$ be the moduli stack parametrizing irreducible nodal hyperelliptic curves of genus $g$ with $n$ marked points. The argument in \cite{Hyperelliptic} provided a 
\end{minipage}
\begin{minipage}{.54\textwidth}
\begin{center}
\begin{tikzpicture}[scale = .4]

\node[scale=.5] at (1, -.7) {$1$};
\node[scale=.5] at (2, -.7) {$2$};
\node[scale=.5] at (3, -.7) {$3$};
\node[scale=.5] at (4, -.7) {$4$};

\node[scale=.5] at (-.7, 1) {$1$};
\node[scale=.5] at (-.7, 2) {$2$};
\node[scale=.5] at (-.7, 3) {$3$};
\node[scale=.5] at (-.7, 4) {$4$};
\node[scale=.5] at (-.7, 5) {$5$};
\node[scale=.5] at (-.7, 6) {$6$};
\node[scale=.5] at (-.7, 7) {$7$};
\node[scale=.5] at (-.7, 8) {$8$};
\node[scale=.5] at (-.7, 9) {$9$};
\node[scale=.5] at (-.7, 10) {$10$};
\node[scale=.5] at (-.7, 11) {$11$};
\node[scale=.5] at (-.7, 12) {$12$};
\node[scale=.5] at (-.7, 13) {$13$};

\draw[->] (0, 0) -- (5, 0);
\draw[->] (0, 0) -- (0, 14);
\node[scale=.9] at (5.4,0) {$g$};
\node[scale=.9] at (0, 14.4) {$n$};
\draw (-.1, 1) -- (.1, 1);
\draw (-.1, 2) -- (.1, 2);
\draw (1, -.1) -- (1, .1);
\draw (2, -.1) -- (2, .1);
\draw (3, -.1) -- (3, .1);
\draw (4, -.1) -- (4, .1);
\filldraw (0, 3) circle (4pt);
\filldraw (0, 4) circle (4pt);
\filldraw (0, 5) circle (4pt);
\filldraw (0, 6) circle (4pt);
\filldraw (0, 7) circle (4pt);
\filldraw (0, 8) circle (4pt);
\filldraw (0, 9) circle (4pt);
\filldraw (0, 10) circle (4pt);
\filldraw (0, 11) circle (4pt);
\filldraw (0, 12) circle (4pt);
\filldraw (0, 13) circle (4pt);
\filldraw (1, 1) circle (4pt);
\filldraw (1, 2) circle (4pt);
\filldraw (1, 3) circle (4pt);
\filldraw (1, 4) circle (4pt);
\filldraw (1, 5) circle (4pt);
\filldraw (1, 6) circle (4pt);
\filldraw (1, 7) circle (4pt);
\filldraw (1, 8) circle (4pt);
\filldraw (1, 9) circle (4pt);
\filldraw (1, 10) circle (4pt);
\node[scale = .6, color = red] at (1, 11) {$\times$};
\node[scale = .6, color = red] at (1, 12) {$\times$};
\node[scale = .6, color = red] at (1, 13) {$\times$};
\filldraw[color=white] (2,0) circle (4pt);
\draw[color=blue] (2, 0) circle (4pt);
\draw[color=blue] (2, 1) circle (4pt);
\draw[color=blue] (2, 2) circle (4pt);
\draw[color=blue] (2, 3) circle (4pt);
\draw[color=blue] (2, 4) circle (4pt);
\draw[color=blue] (2, 5) circle (4pt);
\draw[color=blue] (2, 6) circle (4pt);
\draw[color=blue] (2, 7) circle (4pt);
\draw[color=blue] (2, 8) circle (4pt);
\draw[color=blue] (2, 9) circle (4pt);
\draw[color=blue] (2, 10) circle (4pt);
\node[color=blue, scale = .8] at (3,-2) {Theorem 6.1};
\end{tikzpicture}
\hspace{.25in}
\begin{tikzpicture}[scale = .4]

\node[scale=.5] at (1, -.7) {$1$};
\node[scale=.5] at (2, -.7) {$2$};
\node[scale=.5] at (3, -.7) {$3$};
\node[scale=.5] at (4, -.7) {$4$};

\node[scale=.5] at (-.7, 1) {$1$};
\node[scale=.5] at (-.7, 2) {$2$};
\node[scale=.5] at (-.7, 3) {$3$};
\node[scale=.5] at (-.7, 4) {$4$};
\node[scale=.5] at (-.7, 5) {$5$};
\node[scale=.5] at (-.7, 6) {$6$};
\node[scale=.5] at (-.7, 7) {$7$};
\node[scale=.5] at (-.7, 8) {$8$};
\node[scale=.5] at (-.7, 9) {$9$};
\node[scale=.5] at (-.7, 10) {$10$};
\node[scale=.5] at (-.7, 11) {$11$};
\node[scale=.5] at (-.7, 12) {$12$};
\node[scale=.5] at (-.7, 13) {$13$};
\node[scale=.5] at (-.7, 14) {$14$};

\draw[->] (0, 0) -- (5, 0);
\draw[->] (0, 0) -- (0, 14);
\node[scale=.9] at (5.4,0) {$g$};
\node[scale=.9] at (0, 14.4) {$n$};
\draw (-.1, 1) -- (.1, 1);
\draw (-.1, 2) -- (.1, 2);
\draw (1, -.1) -- (1, .1);
\draw (2, -.1) -- (2, .1);
\draw (3, -.1) -- (3, .1);
\draw (4, -.1) -- (4, .1);
\filldraw (0, 3) circle (4pt);
\filldraw (0, 4) circle (4pt);
\filldraw (0, 5) circle (4pt);
\filldraw (0, 6) circle (4pt);
\filldraw (0, 7) circle (4pt);
\filldraw (0, 8) circle (4pt);
\filldraw (0, 9) circle (4pt);
\filldraw (0, 10) circle (4pt);
\filldraw (0, 11) circle (4pt);
\filldraw (0, 12) circle (4pt);
\filldraw (0, 13) circle (4pt);
\filldraw (1, 1) circle (4pt);
\filldraw (1, 2) circle (4pt);
\filldraw (1, 3) circle (4pt);
\filldraw (1, 4) circle (4pt);
\filldraw (1, 5) circle (4pt);
\filldraw (1, 6) circle (4pt);
\filldraw (1, 7) circle (4pt);
\filldraw (1, 8) circle (4pt);
\filldraw (1, 9) circle (4pt);
\filldraw (1, 10) circle (4pt);
\node[scale = .6, color = red] at (1, 11) {$\times$};
\node[scale = .6, color = red] at (1, 12) {$\times$};
\node[scale = .6, color = red] at (1, 13) {$\times$};
\filldraw[color=blue] (2, 0) circle (4pt);
\filldraw[color=blue] (2, 1) circle (4pt);
\filldraw[color=blue] (2, 2) circle (4pt);
\filldraw[color=blue] (2, 3) circle (4pt);
\filldraw[color=blue] (2, 4) circle (4pt);
\filldraw[color=blue] (2, 5) circle (4pt);
\filldraw[color=blue] (2, 6) circle (4pt);
\filldraw[color=blue] (2, 7) circle (4pt);
\filldraw[color=blue] (2, 8) circle (4pt);
\filldraw[color=blue] (2, 9) circle (4pt);
\draw (2, 10) circle (4pt);
\node[color=blue, scale =.8] at (3,-2) {Lemma 6.3};
\end{tikzpicture}
\end{center}
\end{minipage}

\vspace{.03in}
\noindent
stratification of $\I_{g,n}$ into spaces that we claim satisfy the CKgP.
This is clear from considering \cite[Equations 4.10 and 4.12]{Hyperelliptic}
once we know that $\mathrm{BPU}\cong \mathrm{B}(\gg_m\ltimes \gg_a)$ has the CKgP. To see this, note that $\mathrm{B}\gg_m \to \mathrm{BPU}$ is an affine bundle (see proof of \cite[Lemma 4.6]{Hyperelliptic}). Then apply Lemmas \ref{affbunCKgP} and \ref{bgs}.

\begin{thm}[Theorem 1.1 of \cite{Hyperelliptic}]\label{nodalhyperelliptic}
For $n\leq 2g+6$, we have $A^*(\I_{g,n})$ is generated by the $\psi$ classes and the boundary divisor parametrizing irreducible, nodal hyperelliptic curves. Moreover, $\I_{g,n}$ has the CKgP.
\end{thm}

We now proceed using version 2 of the Filling criteria.

\begin{lem}\label{2end}
For $n \leq 9$, we have $\Mb_{2,n} \smallsetminus \Delta_{\mathrm{dn}}$ has $A^* = R^*$ and the CKgP. Hence, $\Mb_{2,n}$ has $A^* = R^*$ and the CKgP for $n \leq 9$.
\end{lem}
\begin{proof}
Each of the graphs in the stable graph stratification of $\Mb_{2,n} \smallsetminus \Delta_{\dn}$ fits into one of the following four categories:
\begin{enumerate}
\item Graphs where all vertices have genus $0$. All classes supported here are tautological and these strata have CKgP by Section \ref{g0}.
\item One genus $0$ vertex and one genus $1$ vertex with two edges between them.
There are $n_0 \geq 1$ marked points coming out of the genus $0$ vertex and $n_1 = n - n_0 \leq 8$ marked points coming out of the genus $1$ vertex.
This is the graph for (gluing twice)
\[\Mb_{0,n_0+2} \times \Mb_{1,n_1+2} \to \Mb_{2,n}.\]
Note that $n_1 +2 \leq 10$, so these strata have CKgP and contribute only tautological classes by Sections \ref{g0} and \ref{g1}.
\item One vertex of genus $1$ with self edge and $n$ marked points.
\item One vertex of genus $2$ with $n$ marked points.
\end{enumerate}
\begin{center}
    \includegraphics[width=5in]{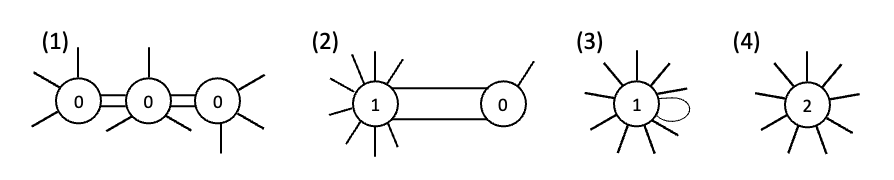} \\
    
   Examples of stable graphs of types (1)--(4).
    \end{center}
Finally, the union of strata of type (3) and (4) is $\I_{2,n} \subset \Mb_{2,n}$. Hence, we are done by Theorem \ref{nodalhyperelliptic}.
\end{proof}

\begin{rem}
When $g = 2$, Theorem \ref{hyperelliptic}
tells us $A^*(\M_{2,10}) = R^*(\M_{2,10})$. However, we have \emph{not} proved that $A^*(\Mb_{2,10}) = R^*(\Mb_{2,10})$ because we have not been able to prove that cycles supported on the boundary stratum for
$\Mb_{1,11} \times \Mb_{1,1}$ are tautological. We think it likely that many non-tautological $1$-cycles are supported on this stratum. 
\end{rem}


\section{Plane curves} \label{planesec}
In this section, we give a quotient stack presentation for stacks of smooth pointed plane curves of degree $d$. We use the presentation to show that the Chow rings of these stacks have the CKgP and are generated by restrictions of tautological classes from $\M_{g,n}$. In particular, we will finish the $g=3$ cases of Theorem \ref{main} because nonhyperelliptic curves of genus $3$ are plane quartics.

\subsection{Construction of the stack} \label{pcsetup}
First, we define the stack of pointed plane curves. Let $d\geq 3$ be a positive integer and set $g:={d-1 \choose 2}$. Let $\F_{d,n}$ denote the stack whose objects over a scheme $S$ are given by commutative diagrams 
\[
\begin{tikzcd}
C \arrow[r, "j", hook] \arrow[d, "f"]               & P \arrow[ld, "\pi"] \\
S \arrow[u, "{\sigma_1,\dots,\sigma_n}", bend left] &                    
\end{tikzcd}
\]
where $f:C\rightarrow S$ is a smooth proper relative curve with $n$ pairwise disjoint sections $\sigma_1,\dots,\sigma_n:S\rightarrow C$; $\pi:P\rightarrow S$ is a $\pp^2$ fibration; and $j:C\hookrightarrow P$ is a closed embedding such that for every geometric point $s\in S$, $C_s\hookrightarrow \p^2_{\kappa(s)}$ is of degree $d$. The morphisms in $\F_{d,n}(S)$ between objects $(C\rightarrow P\rightarrow S, \sigma_1,\dots,\sigma_n:S\rightarrow C)$ and
$(C'\rightarrow P'\rightarrow S, \sigma_1',\dots,\sigma_n':S\rightarrow C')$ are isomorphisms $P\rightarrow P'$ inducing isomorphisms $C\rightarrow C'$ sending the sections $\sigma_i$ to $\sigma_i'$. In \cite{Aaron}, Landesman shows that the natural map $\F_{d, 0} \to \M_{g}$ is a locally closed embedding of stacks when $d \geq 4$.
Taking the base change along $\M_{g,n} \to \M_{g}$,
it follows that $\F_{d, n} \to \M_{g,n}$ is a locally closed embedding for any $n$ when $d \geq 4$.

The stack $\F_{d,n}$ admits a natural morphism to $\BPGL_3$, sending the family of plane curves to its associated $\p^2$-fibration. We define the stack $\G_{d,n}$ by the following Cartesian diagram:
\[
\begin{tikzcd}
{\G_{d,n}} \arrow[r] \arrow[d] & {\F_{d,n}} \arrow[d] \\
\BSL_3 \arrow[r]               & \BPGL_3             
\end{tikzcd}
\]
The stack $\G_{d,n}$ is a $\mu_3$-banded gerbe over $\F_{d,n}$. Thus, $A^*(\F_{d,n})\cong A^*(\G_{d,n})$, as we always work with rational coefficients, and $\F_{d,n}$ satisfies CKgP if and only if $\G_{d,n}$ does (Lemma \ref{gerbe}). The points of $\G_{d,n}$ over a scheme $S$ are given by diagrams
\[
\begin{tikzcd}
C \arrow[r, "j", hook] \arrow[d, "f"]               & \pp V \arrow[ld, "\pi"] \\
S \arrow[u, "{\sigma_1,\dots,\sigma_n}", bend left] &                    
\end{tikzcd}
\]
where $V$ is a rank $3$ vector bundle on $S$ with trivial first Chern class. The universal bundle $\V$ is pulled back from the universal bundle over $\BSL_3$. 
We construct $\G_{d,n}$, for certain values of $n$ depending on $d$, explicitly, and use the explicit presentation to describe its Chow ring. 

Consider $\gamma:\p \V\rightarrow \BSL_3$, the universal $\pp^2$-bundle. Let $(\p \V)^n$ be the fiber product of $n$ copies of $\gamma:\p \V\rightarrow \BSL_3$. There are projection maps $\eta_i: (\p \V)^n\rightarrow \p \V$. Set $\W_d :=\gamma^*\gamma_*\O_{\p \V}(d) = \gamma^* \Sym^d \V^\vee$. We have an evaluation map
\[
\W_d\rightarrow \O_{\p \V}(d). 
\]
Note that $\eta_i^*\W_d\cong \eta_j^* \W_d$ for all $i,j$ because $\W_d$ is by definition a pullback from $\BSL_3$. We will denote this pullback by $\W_d\otimes \O_{(\p \V)^n}$. Therefore, we have a direct sum evaluation map
\begin{equation}\label{planecurveeval}
\W_d\otimes \O_{(\p V)^n}\rightarrow \bigoplus_{i=1}^n \eta_i^*\O_{\p \V}(d).
\end{equation}
Set $Y_{d,n}$ to be the preimage of the zero section under \eqref{planecurveeval}. The stack $Y_{d,n}$ parametrizes tuples $(f,p_1,\dots,p_n)$ where $f$ is a degree $d$ form in $\p V$, where $V$ is a rank $3$ vector bundle with trivial first Chern class, and $p_1,\dots,p_n$ are (not necessarily distinct) points such that $f(p_i)=0$. The stack $Y_{d,n}$ is not a vector bundle over $(\p \V)^n$ because the fiber dimension jumps when the points $p_i$ fail to impose independent conditions on degree $d$ curves. 
In order to study the Chow ring of $\G_{d,n}$, we will need to understand when \eqref{planecurveeval} is surjective. By cohomology and base change, we can check surjectivity on the fibers.

\begin{lem}\label{planecurvesurj}
Let $\Gamma\subset \p^2$ be a collection of $n\leq 3d-1$ distinct points. Suppose there is an irreducible nodal plane curve $C\subset \p^2$ of degree $d$ so that $\Gamma$ is contained in the smooth locus of $C$. Then the evaluation map
\[
H^0(\p^2,\O_{\p^2}(d))\rightarrow H^0(\Gamma,\O_{\p^2}(d)|_{\Gamma})
\]
is surjective. 
\end{lem}
\begin{proof}
The evaluation map factors as
\begin{equation} \label{factor}H^0(\pp^2, \O_{\pp^2}(d)) \to H^0(C, \O_{\pp^2}(d)|_C) \to H^0(\Gamma, \O_{\pp^2}(d)|_{\Gamma}). 
\end{equation}
The first map in \eqref{factor} is surjective because $H^1(\pp^2, \O_{\pp^2}) = 0$. So it remains to show that the second map is surjective. To do so, consider the exact sequence
\[
0\rightarrow \O_{\p^2}(d)|_{C}(-\Gamma)\rightarrow \O_{\p^2}(d)|_{C}\rightarrow \O_{\p^2}(d)|_{\Gamma}\rightarrow 0.
\]
After taking global sections, it suffices to show that $H^1(\O_{\p^2}(d)|_{C}(-\Gamma))=0$. By Serre duality, this is equivalent to showing that $H^0(\O_{\p^2}(-d)|_{C}\otimes \omega_C(\Gamma))=0$ where $\omega_C$ is the dualizing sheaf, which is a line bundle of degree $2g - 2$. The bundle $\O_{\p^2}(-d)|_{C}\otimes \omega_C(\Gamma)$ is of degree 
\[-d^2+2g-2+n=-d^2+(d-1)(d-2)-2+n=-3d+n.\]
On an irreducible curve, a line bundle of negative degree has no global sections. Hence, this line bundle has no global sections when $n\leq 3d-1$.
\end{proof}
Define $U\subset (\p V)^n$ to be the locus over which \eqref{planecurveeval} is surjective.
Note that $U$ is necessarily contained in the complement of the diagonals in $(\pp V)^n$.
We define $K_{d,n}:=Y_{d,n}|_{U}$ to be the kernel of \eqref{planecurveeval} restricted to $U$. Thus, $K_{d,n}$ is a vector bundle over $U$. The projectivization $\p K_{d,n}$ parametrizes tuples $(C,V, p_1,\dots,p_n)$ where $C$ is a degree $d$ curve in $\p V$ containing the points $p_1,\dots,p_n$ and this collection of points imposes independent conditions on polynomials of degree $d$. Lemma \ref{planecurvesurj} says that the image of $\G_{d,n}$ in $(\p \V)^n$ is contained in $U$ when $n\leq 3d-1$. 
\begin{lem} \label{gck}
The stack $\G_{d,n}$ has the CKgP for $n\leq 3d-1$.
\end{lem}
\begin{proof}
We have a sequence of morphisms
\[
\G_{d,n}\subset \p K_{d,n}\rightarrow U \subset (\p V)^n \rightarrow \BSL_3,
\]
where each $\subset$ is an open embedding and each $\rightarrow$ is a  (product of) projective bundles. Thus, $\G_{d,n}$ has the CKgP by Lemmas \ref{open} and \ref{grassmann}.
\end{proof}

\subsection{Generators for the Chow ring}
The Chow ring of $\BSL_3$ is generated by 
the Chern classes of the tautological rank $3$ bundle $c_2(\V), c_3(\V) \in A^*(\BSL_3)$.
By the projective bundle theorem, the Chow ring of $(\pp \V)^n$ is generated over $A^*(\BSL_3)$ by the classes of the relative $\O(1)$'s.
Let $\eta_i: (\pp \V)^n \to \pp \V$ be the $i^{\mathrm{th}}$ projection and write $z_i = \eta_i^*c_1(\O_{\pp \V}(1)) \in A^1((\pp \V)^n)$. 
By the projective bundle theorem again, the Chow ring of $\pp K_{d,n}$ is generated over $A^*(U)$ by $\zeta := c_1(\O_{\pp K_{d,n}}(1))$. In summary, we have:
\begin{lem} \label{planegens}
The ring $A^*(\pp K_{d,n})$, and therefore $A^*(\G_{d,n})$, is generated by (the pullbacks of)
\[c_2(V), \qquad c_3(V), \qquad z_1, \ldots, z_n, \qquad \text{and} \qquad \zeta.  \]
\end{lem}

In order to give a geometric interpretation of these generators, we first construct the universal pointed plane curve over $\pp K_{d,n}$. By an abuse of notation, we will continue to denote by $\V$ the pullback of $\V$ from $\BSL_3$ to $\p K_{d,n}$, and we consider the diagram
\begin{center}
\begin{tikzcd}
\mathcal{P} := \pp K_{d,n} \times_{\BSL_3} \p \V  \arrow{r}{\pi_2} \arrow{d}{\pi_1} & \p \V \\
\p K_{d,n}.
\end{tikzcd}
\end{center}
We have a morphism of bundles on $\P$:
\[
\pi_1^*\O_{\p K_{d,n}}(-1)\otimes \pi_2^*\O_{\p \V}(-d)\rightarrow K_{d,n}\otimes \pi_2^*\O_{\p \V}(-d)\rightarrow \O_{\P}
\]
where the first map arises from the tautological sequence on $\p K_{d,n}$ and the second is obtained from multiplying forms. This morphism defines a sheaf of ideals in $\O_{\P}$, and we will denote by $\C\subset \P$ the corresponding divisor. By construction, the restriction of $\pi_1$ to $\C$,  is a relative plane curve $f:\C\rightarrow \p K_{d,n}$ of degree $d$. Next, we show that it has $n$ sections.
The product of the identity map together with the composition
\[
\p K_{d,n}\rightarrow U \hookrightarrow (\pp V)^n \xrightarrow{\eta_i} \p V,
\]
induces a section of $\pi_1$,
\[\sigma_i : \pp K_{d,n} \to \pp K_{d,n} \times_{\BSL_3} \pp \V = \P. \]
These sections factor through $\C$ by the definition of $\C$ and $\p K_{d,n}$. Indeed, the fiber of $\pp K_{d,n}$ over $(p_1, \ldots, p_n) \in U$ is the space of degree $d$ forms $G$ vanishing at the $p_i$ and the fiber of $\C \to \pp K_{d,n}$ over a degree $d$ form $G$ is the vanishing of $G$. Note that $\G_{d,n} \subset \pp K_{d,n}$ is the locus over which $\C \to \pp K_{d,n}$ is smooth. Let $X \subset \pp K_{d,n}$ be the open locus over which $\C \to \pp K_{d,n}$ is stable (so, of course, $\G_{d,n} \subseteq X$).

\begin{lem} \label{gdn}
Suppose $n\leq 3d-1$ and $d\geq 4$. 
The generators $c_2(\V), c_3(\V), z_1, \ldots, z_n$ and $\zeta$ of $A^*(\pp K_{d,n})$ restrict to polynomials in the $\lambda$ and $\psi$ classes on $X \subset \pp K_{d,n}$. In particular, $A^*(X)$ and $A^*(\G_{d,n})$ are generated by pullbacks of tautological classes.
\end{lem}

\begin{proof}
We have the universal diagram restricted to $X \subset \pp K_{d,n}$:
\[
\begin{tikzcd}
\C \arrow[r, "j", hook] \arrow[d, "f"]               & \P \arrow[ld, "\pi"] \\
X \arrow[u, "{\sigma_1,\dots,\sigma_n}", bend left] & 
\end{tikzcd}
\]
Set $\O_X(1)=\O_{\p K_{d,n}}(1)|_X$.
By construction, $\C$ is the zero locus of a section of $\pi^*\O_X(1) \otimes \O_{\pp \V}(d)$. 
By adjunction and the fact that $\V$ is a rank $3$ vector bundle with trivial first Chern class,
\[
\omega_f=j^*(\omega_{\pi}\otimes \O_{\p \V}(d)\otimes\pi^*\O_X(1))=f^*\O_X(1)\otimes j^*\O_{\p \V}(d-3).
\]
Pushing forward by $f$, we have
\[
f_*\omega_f=\O_{X}(1)\otimes \pi_*j_*j^*\O_{\p V}(d-3)=\O_{X}(1)\otimes \Sym^{d-3} \V^{\vee}.
\]
By taking first Chern classes, we see that $\lambda_1=\zeta$ (recall the $\lambda$ classes defined in \eqref{ldef}).  Taking higher Chern classes and using the splitting principle shows that $c_2(\V)$ and $c_3(\V)$ are also polynomials in the $\lambda$ classes.
Meanwhile, pulling back by $\sigma_i$, we have
\[
\sigma_i^*\omega_f=\eta_i^*\O_{\p \V}(d-3)\otimes \O_{X}(1).
\]
 Taking first Chern classes, we see that
 $\psi_i = (d-3)z_i + \zeta$ so $z_i = \frac{1}{d-3}(\psi_i - \lambda_1)$. (Note that we are using $d - 3 \neq 0$.) The result now follows by Lemma \ref{planegens}.
\end{proof}

\begin{cor}\label{3open}
For $n\leq 11$, $A^*(\M_{3,n})=R^*(\M_{3,n})$ and $\M_{3,n}$ has the CKgP.
\end{cor}
\begin{proof}
By Theorem \ref{hyperelliptic}, we know all classes supported on $\Hyp_{3,n}$ are tautological (the fundamental class is tautological by \cite[Proposition 1]{FaberPandharipande}; then use the push-pull formula.)
Moreover, $\Hyp_{3,n}$ has the CKgP, as discussed in Section \ref{hypsec}. By excision, it remains to show that $\M_{3,n} \smallsetminus \Hyp_{3,n}$ has the CKgP and $A^* = R^*$.

Every curve in $\M_{3,n} \smallsetminus \Hyp_{3,n}$ is an $n$-pointed plane quartic, so $\M_{3,n} \smallsetminus \Hyp_{3,n} \cong \F_{4,n}$. Since $\G_{4,n}\rightarrow \F_{4,n}$ is a $\mu_3$-banded gerbe, we see that by combining Lemma \ref{gerbe} with Lemma \ref{gck}, $\M_{3,n} \smallsetminus \Hyp_{3,n}$ has the CKgP. Moreover, by Lemma \ref{gdn}, it has $A^* = R^*$.
\end{proof}

\subsection{Conclusion for $g = 3$}
Corollary \ref{3open} gives us a new column of open circles in genus $3$, pictured in the left chart below. By the first version of the filling criterion, Lemma \ref{fc1}, we can immediately fill in the circles with $n\leq 7$, pictured in the middle chart below. With some extra work, we show in Lemma \ref{38} that we can also fill in the additional circle where $n=8$, pictured on the right.
\begin{center}
\begin{tikzpicture}[scale = .4]
\node[scale=.5] at (1, -.7) {$1$};
\node[scale=.5] at (2, -.7) {$2$};
\node[scale=.5] at (3, -.7) {$3$};
\node[scale=.5] at (4, -.7) {$4$};

\node[scale=.5] at (-.7, 1) {$1$};
\node[scale=.5] at (-.7, 2) {$2$};
\node[scale=.5] at (-.7, 3) {$3$};
\node[scale=.5] at (-.7, 4) {$4$};
\node[scale=.5] at (-.7, 5) {$5$};
\node[scale=.5] at (-.7, 6) {$6$};
\node[scale=.5] at (-.7, 7) {$7$};
\node[scale=.5] at (-.7, 8) {$8$};
\node[scale=.5] at (-.7, 9) {$9$};
\node[scale=.5] at (-.7, 10) {$10$};
\node[scale=.5] at (-.7, 11) {$11$};
\node[scale=.5] at (-.7, 12) {$12$};
\node[scale=.5] at (-.7, 13) {$13$};

\draw[->] (0, 0) -- (5, 0);
\draw[->] (0, 0) -- (0, 14);
\node[scale=.9] at (5.4,0) {$g$};
\node[scale=.9] at (0, 14.4) {$n$};
\draw (-.1, 1) -- (.1, 1);
\draw (-.1, 2) -- (.1, 2);
\draw (1, -.1) -- (1, .1);
\draw (2, -.1) -- (2, .1);
\draw (3, -.1) -- (3, .1);
\draw (4, -.1) -- (4, .1);
\filldraw (0, 3) circle (4pt);
\filldraw (0, 4) circle (4pt);
\filldraw (0, 5) circle (4pt);
\filldraw (0, 6) circle (4pt);
\filldraw (0, 7) circle (4pt);
\filldraw (0, 8) circle (4pt);
\filldraw (0, 9) circle (4pt);
\filldraw (0, 10) circle (4pt);
\filldraw (0, 11) circle (4pt);
\filldraw (0, 12) circle (4pt);
\filldraw (0, 13) circle (4pt);
\filldraw (1, 1) circle (4pt);
\filldraw (1, 2) circle (4pt);
\filldraw (1, 3) circle (4pt);
\filldraw (1, 4) circle (4pt);
\filldraw (1, 5) circle (4pt);
\filldraw (1, 6) circle (4pt);
\filldraw (1, 7) circle (4pt);
\filldraw (1, 8) circle (4pt);
\filldraw (1, 9) circle (4pt);
\filldraw (1, 10) circle (4pt);
\node[scale = .6, color = red] at (1, 11) {$\times$};
\node[scale = .6, color = red] at (1, 12) {$\times$};
\node[scale = .6, color = red] at (1, 13) {$\times$};
\filldraw (2, 0) circle (4pt);
\filldraw (2, 1) circle (4pt);
\filldraw (2, 2) circle (4pt);
\filldraw (2, 3) circle (4pt);
\filldraw (2, 4) circle (4pt);
\filldraw (2, 5) circle (4pt);
\filldraw (2, 6) circle (4pt);
\filldraw (2, 7) circle (4pt);
\filldraw (2, 8) circle (4pt);
\filldraw (2, 9) circle (4pt);
\draw (2, 10) circle (4pt);
\filldraw[color=white] (3, 0) circle (4pt);
\draw[color=blue] (3, 0) circle (4pt);
\draw[color=blue] (3, 1) circle (4pt);
\draw[color=blue] (3, 2) circle (4pt);
\draw[color=blue] (3, 3) circle (4pt);
\draw[color=blue] (3, 4) circle (4pt);
\draw[color=blue] (3, 5) circle (4pt);
\draw[color=blue] (3, 6) circle (4pt);
\draw[color=blue] (3, 7) circle (4pt);
\draw[color=blue] (3, 8) circle (4pt);
\draw[color=blue] (3, 9) circle (4pt);
\draw[color=blue] (3, 10) circle (4pt);
\draw[color=blue] (3, 11) circle (4pt);

\node[color=blue, scale=.8] at (2.5, -2)  {Corollary 7.5};

\node[scale=.8] at (8, 8) {Filling criterion};
\draw[->] (5.5, 7) -- (10.5, 7);
\node[scale=.8] at (8, 6) {version 1};

\end{tikzpicture}
\hspace{.15in}
\begin{tikzpicture}[scale = .4]

\node[scale=.5] at (1, -.7) {$1$};
\node[scale=.5] at (2, -.7) {$2$};
\node[scale=.5] at (3, -.7) {$3$};
\node[scale=.5] at (4, -.7) {$4$};

\node[scale=.5] at (-.7, 1) {$1$};
\node[scale=.5] at (-.7, 2) {$2$};
\node[scale=.5] at (-.7, 3) {$3$};
\node[scale=.5] at (-.7, 4) {$4$};
\node[scale=.5] at (-.7, 5) {$5$};
\node[scale=.5] at (-.7, 6) {$6$};
\node[scale=.5] at (-.7, 7) {$7$};
\node[scale=.5] at (-.7, 8) {$8$};
\node[scale=.5] at (-.7, 9) {$9$};
\node[scale=.5] at (-.7, 10) {$10$};
\node[scale=.5] at (-.7, 11) {$11$};
\node[scale=.5] at (-.7, 12) {$12$};
\node[scale=.5] at (-.7, 13) {$13$};

\draw[->] (0, 0) -- (5, 0);
\draw[->] (0, 0) -- (0, 14);
\node[scale=.9] at (5.4,0) {$g$};
\node[scale=.9] at (0, 14.4) {$n$};
\draw (-.1, 1) -- (.1, 1);
\draw (-.1, 2) -- (.1, 2);
\draw (1, -.1) -- (1, .1);
\draw (2, -.1) -- (2, .1);
\draw (3, -.1) -- (3, .1);
\draw (4, -.1) -- (4, .1);
\filldraw (0, 3) circle (4pt);
\filldraw (0, 4) circle (4pt);
\filldraw (0, 5) circle (4pt);
\filldraw (0, 6) circle (4pt);
\filldraw (0, 7) circle (4pt);
\filldraw (0, 8) circle (4pt);
\filldraw (0, 9) circle (4pt);
\filldraw (0, 10) circle (4pt);
\filldraw (0, 11) circle (4pt);
\filldraw (0, 12) circle (4pt);
\filldraw (0, 13) circle (4pt);
\filldraw (1, 1) circle (4pt);
\filldraw (1, 2) circle (4pt);
\filldraw (1, 3) circle (4pt);
\filldraw (1, 4) circle (4pt);
\filldraw (1, 5) circle (4pt);
\filldraw (1, 6) circle (4pt);
\filldraw (1, 7) circle (4pt);
\filldraw (1, 8) circle (4pt);
\filldraw (1, 9) circle (4pt);
\filldraw (1, 10) circle (4pt);
\node[scale = .6, color = red] at (1, 11) {$\times$};
\node[scale = .6, color = red] at (1, 12) {$\times$};
\node[scale = .6, color = red] at (1, 13) {$\times$};
\filldraw (2, 0) circle (4pt);
\filldraw (2, 1) circle (4pt);
\filldraw (2, 2) circle (4pt);
\filldraw (2, 3) circle (4pt);
\filldraw (2, 4) circle (4pt);
\filldraw (2, 5) circle (4pt);
\filldraw (2, 6) circle (4pt);
\filldraw (2, 7) circle (4pt);
\filldraw (2, 8) circle (4pt);
\filldraw (2, 9) circle (4pt);
\draw (2, 10) circle (4pt);
\filldraw[color=blue] (3, 0) circle (4pt);
\filldraw[color=blue] (3, 1) circle (4pt);
\filldraw[color=blue] (3, 2) circle (4pt);
\filldraw[color=blue] (3, 3) circle (4pt);
\filldraw[color=blue] (3, 4) circle (4pt);
\filldraw[color=blue] (3, 5) circle (4pt);
\filldraw[color=blue] (3, 6) circle (4pt);
\filldraw[color=blue] (3, 7) circle (4pt);
\draw (3, 8) circle (4pt);
\draw (3, 9) circle (4pt);
\draw (3, 10) circle (4pt);
\draw (3, 11) circle (4pt);
\node[color = white] at (2.5, -2) {hi};
\end{tikzpicture}
\hspace{1.25in}
\begin{tikzpicture}[scale = .4]

\node[scale=.5] at (1, -.7) {$1$};
\node[scale=.5] at (2, -.7) {$2$};
\node[scale=.5] at (3, -.7) {$3$};
\node[scale=.5] at (4, -.7) {$4$};

\node[scale=.5] at (-.7, 1) {$1$};
\node[scale=.5] at (-.7, 2) {$2$};
\node[scale=.5] at (-.7, 3) {$3$};
\node[scale=.5] at (-.7, 4) {$4$};
\node[scale=.5] at (-.7, 5) {$5$};
\node[scale=.5] at (-.7, 6) {$6$};
\node[scale=.5] at (-.7, 7) {$7$};
\node[scale=.5] at (-.7, 8) {$8$};
\node[scale=.5] at (-.7, 9) {$9$};
\node[scale=.5] at (-.7, 10) {$10$};
\node[scale=.5] at (-.7, 11) {$11$};
\node[scale=.5] at (-.7, 12) {$12$};
\node[scale=.5] at (-.7, 13) {$13$};

\draw[->] (0, 0) -- (5, 0);
\draw[->] (0, 0) -- (0, 14);
\node[scale=.9] at (5.4,0) {$g$};
\node[scale=.9] at (0, 14.4) {$n$};
\draw (-.1, 1) -- (.1, 1);
\draw (-.1, 2) -- (.1, 2);
\draw (1, -.1) -- (1, .1);
\draw (2, -.1) -- (2, .1);
\draw (3, -.1) -- (3, .1);
\draw (4, -.1) -- (4, .1);
\filldraw (0, 3) circle (4pt);
\filldraw (0, 4) circle (4pt);
\filldraw (0, 5) circle (4pt);
\filldraw (0, 6) circle (4pt);
\filldraw (0, 7) circle (4pt);
\filldraw (0, 8) circle (4pt);
\filldraw (0, 9) circle (4pt);
\filldraw (0, 10) circle (4pt);
\filldraw (0, 11) circle (4pt);
\filldraw (0, 12) circle (4pt);
\filldraw (0, 13) circle (4pt);
\filldraw (1, 1) circle (4pt);
\filldraw (1, 2) circle (4pt);
\filldraw (1, 3) circle (4pt);
\filldraw (1, 4) circle (4pt);
\filldraw (1, 5) circle (4pt);
\filldraw (1, 6) circle (4pt);
\filldraw (1, 7) circle (4pt);
\filldraw (1, 8) circle (4pt);
\filldraw (1, 9) circle (4pt);
\filldraw (1, 10) circle (4pt);
\node[scale = .6, color = red] at (1, 11) {$\times$};
\node[scale = .6, color = red] at (1, 12) {$\times$};
\node[scale = .6, color = red] at (1, 13) {$\times$};
\filldraw (2, 0) circle (4pt);
\filldraw (2, 1) circle (4pt);
\filldraw (2, 2) circle (4pt);
\filldraw (2, 3) circle (4pt);
\filldraw (2, 4) circle (4pt);
\filldraw (2, 5) circle (4pt);
\filldraw (2, 6) circle (4pt);
\filldraw (2, 7) circle (4pt);
\filldraw (2, 8) circle (4pt);
\filldraw (2, 9) circle (4pt);
\draw (2, 10) circle (4pt);
\filldraw[color=black] (3, 0) circle (4pt);
\filldraw[color=black] (3, 1) circle (4pt);
\filldraw[color=black] (3, 2) circle (4pt);
\filldraw[color=black] (3, 3) circle (4pt);
\filldraw[color=black] (3, 4) circle (4pt);
\filldraw[color=black] (3, 5) circle (4pt);
\filldraw[color=black] (3, 6) circle (4pt);
\filldraw[color=black] (3, 7) circle (4pt);
\filldraw[color=blue] (3, 8) circle (4pt);
\draw (3, 9) circle (4pt);
\draw (3, 10) circle (4pt);
\draw (3, 11) circle (4pt);
\node[color = blue, scale=.8] at (2.5, -2) {Lemma 7.8};
\end{tikzpicture}
\end{center}

The reason $(g, n) = (3, 8)$ is not already filled is because
\begin{enumerate}
    \item  Lemma \ref{fc1}(3) is not satisfied because the circle for $(g-1,n+2) = (2,10)$ is not filled. In turn, the reason that the circle at $(2,10)$ is not filled is that:
    \begin{enumerate}
        \item Lemma \ref{fc1}(2) fails for  $(g-2,n+1)=(1,11)$,
        \item Lemma \ref{fc1}(3) fails for $(g-1, n+2) = (1,12)$.
    \end{enumerate}
\end{enumerate}

Consider the following four stable graphs
\begin{equation} \label{sg38}\M_{1,11} \times \M_{0,3} \qquad \qquad \quad \M_{1,12} \qquad \qquad \M_{2,10} \qquad \qquad \M_{3,8}. 
\end{equation}

    \begin{center}
    \includegraphics[width=5in]{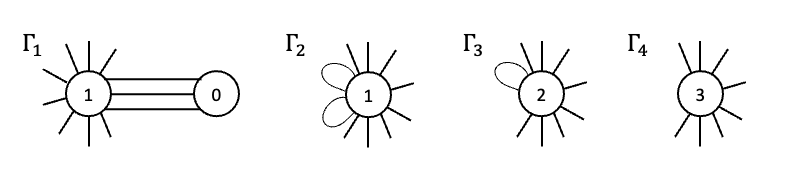} \\
    The graphs corresponding to the ``unfilled" strata in $\Mb_{3,8}$.
    \end{center}

We claim that all other strata of $\Mb_{3,8}$
lie in the image of some $\prod \Mb_{g_i,n_i} \to \Mb_{3,8}$ where the vertices for $(g_i, n_i)$ are already filled (i.e. $g_i = 0$; $g_i = 1$ and $n_i \leq 10$; $g_i = 2$ and $n_i \leq 9$; or $g_i = 3$ and $n_i \leq 7$.) 
Indeed, if $\Gamma$ is in the stable graph stratification for $\Mb_{3,8}$, then all vertices of $\Gamma$ satisfy
$2g(v) + n(v) \leq 14$.
Moreover, the only $(g, n)$ with $2g + n \leq 14$ which are not already filled
are
\[(1, 11) \qquad (1, 12) \qquad (2, 10) \qquad (3, 8). \]
For the latter three, there is only one stable graph that uses such a vertex, which is the one pictured in \eqref{sg38}. For $(1, 11)$ there are a few stable graphs:

\begin{center}
    \includegraphics[width=5in]{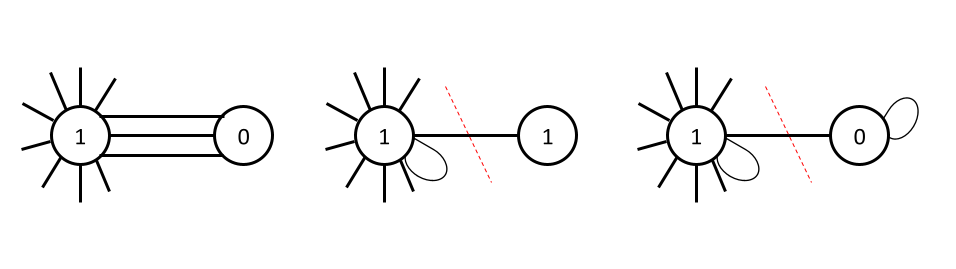}
    \end{center}
\noindent
However, all but the first one appear in the image of $\Mb_{2,9} \times \Mb_{1,1} \to \Mb_{3,8}$ because they have a disconnecting node (corresponding to the edge with a red dashed line through it above).

Let $\M^\circ = \M_{\Gamma_1} \cup \M_{\Gamma_2} \cup \M_{\Gamma_3} \cup \M_{\Gamma_4}$ be the union of the four strata in \eqref{sg38}. Note that $\M^\circ \subset \Mb_{3,8}$ is open as the set of graphs is closed under taking edge contractions. (In fact, each stratum above lies in the closure of the one to its right: $\Mb_{\Gamma_1} \subset \Mb_{\Gamma_2} \subset \Mb_{\Gamma_3} \subset \Mb_{\Gamma_4} = \M^\circ$.) It remains to show that $\M^\circ$ has the CKgP and $A^* = R^*$.

We do this by cutting $\M^\circ$ into two pieces.
Let $Y \subset \M^\circ$ be the open locus where the dualizing sheaf is very ample and let $Z \subset \M^\circ$ be the closed complement of $Y$.

\begin{lem}
We have $Z = \I_{3,8} \cap \M^\circ$. In particular, $Z$ has the CKgP and all classes supported on $Z \subset \M^\circ$ are tautological in $A^*(\M^\circ)$.
\end{lem}

\begin{proof}
We shall show that $Z \cap \M_{\Gamma_i} = \I_{3,8} \cap \M_{\Gamma_i}$ for each $i = 1, \ldots, 4$.

\medskip
($i=1$) Curves in $\M_{\Gamma_1}$ are reducible, so $\I_{3,8} \cap \M_{\Gamma_1} = \varnothing$.
We must show $Z \cap \M_{\Gamma_1}= \varnothing$ too.
Suppose $C = E \cup L$ with $E$ genus $1$ and $L$ genus $0$ is in $\M_{\Gamma_1}$. Let $p_1, p_2, p_3 \in E$ and $q_1, q_2, q_3 \in L$ be the points above the nodes (see Figure \ref{normalizations}, left).
Then 
\[\omega_C|_E = \omega_E(p_1+p_2+p_3) = \O_E(p_1 + p_2 + p_3),\] which is very ample on $E$. Meanwhile, \[\omega_C|_L = \omega_L(q_1+q_2+q_3) = \O_{\pp^1}(-2) \otimes \O_{\pp^1}(3) = \O_{\pp^1}(1)\]
is also very ample.
Hence, $|\omega_C|$ is very ample and sends $C$ to the union of a smooth cubic and a line.

\begin{figure}[h!]
    \includegraphics[width=4.5in]{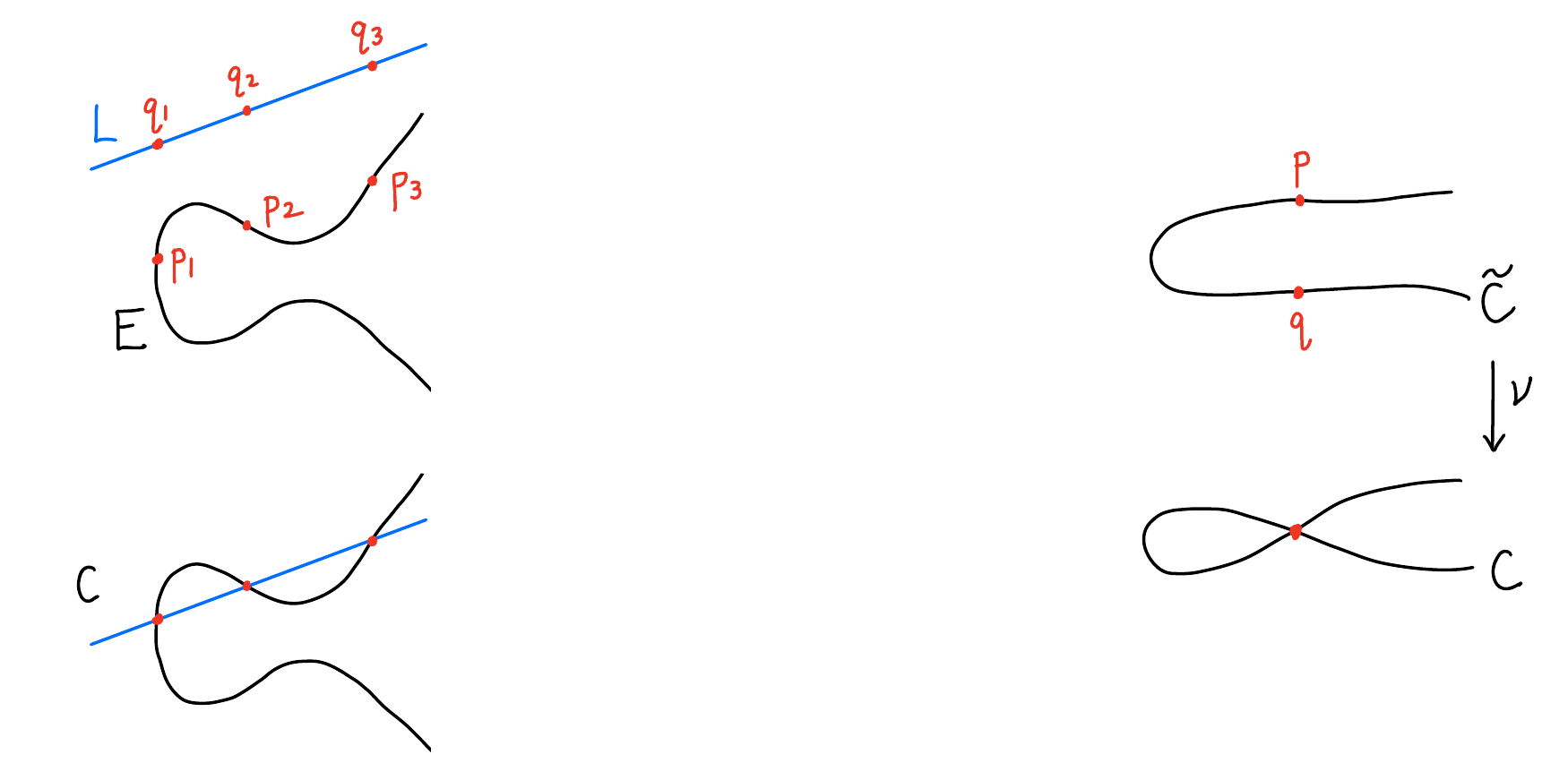}
    \caption{Left: a curve in $\M_{\Gamma_1}$ and its normalization. Right: a curve in $\M_{\Gamma_2} \cup \M_{\Gamma_3}$ and its partial normalization}
    \label{normalizations}
\end{figure}

($i=2,3$) Next, let us describe the intersection of $Z$ with $\M_{\Gamma_2} \cup \M_{\Gamma_3}$.
Let $C$ be a curve in $\M_{\Gamma_2} \cup \M_{\Gamma_3}$.
Consider the partial normalization $\nu: \tilde{C} \to C$ at one node.
Let $p,q$ be the points lying over the node (see Figure \ref{normalizations}, right). Now, $\tilde{C}$ is an irreducible genus $2$ curve.
We have $\nu^* \omega_C = \omega_{\tilde{C}}(p+q)$.
Using this, it is not hard to see that $\omega_C$ fails to be very ample if and only if $p, q$ are conjugate under the involution defined by $|\omega_{\tilde{C}}|$. This is equivalent to $C$ being in $\I_{3,8}$ (see \cite[Section 2.1]{Hyperelliptic}).

($i = 4$) The final stratum $\M_{\Gamma_4}$ consists of smooth genus $3$ curves. It is well-known that the canonical line bundle fails to be very ample if and only if the curve is hyperelliptic, so $\M_{\Gamma_4} \cap Z = \Hyp_{4,8} = \M_{\Gamma_4} \cap \I_{3,8}$.

\medskip
To finish the proof, note that $\I_{3,8} \cap \M^\circ$ is open inside $\I_{3,8}$. Applying Theorem \ref{nodalhyperelliptic}, we see that $\I_{3,8} \cap \M^\circ$ has the CKgP and is generated by restrictions of tautological classes. The fundamental class 
of $\I_{3,8} \cap \M^\circ \subset \M^\circ$
is tautological since it is the pullback under $\M^\circ \to \Mb_{3,8} \to \Mb_3$ of the closure of the hyperelliptic locus inside $\Mb_3$ (which is tautological by \cite[Proposition 1]{FaberPandharipande}). Thus, we are done using the push-pull formula. 
\end{proof}

Next, we study the open complement $Y \subset \M^\circ$. Let $X \subset \pp K_{4,8}$ be as in the previous subsection.
We consider the base change
\begin{center}
\begin{tikzcd}
\tilde{Y}  \arrow{d} \arrow{r} & Y \arrow{d} \\
X \arrow{r} & \Mb_{3,8}.
\end{tikzcd}
\end{center}

\begin{lem}
The substack $\tilde{Y}  \subset \pp K_{4,8}$ is open and $\tilde{Y} \to Y$ is a $\mu_3$-banded gerbe.
Hence, $Y$ has the CKgP and $A^* = R^*$.
\end{lem}
\begin{proof}
The map $Y \to \Mb_{3,8}$ is open, so $\tilde{Y} \subset X$ is open, and hence $\tilde{Y} \subset \pp K_{4,8}$ is open.
The map $X \to \Mb_{3,8}$ is a $\mu_3$-banded gerbe over its image (induced from our base change $\BSL_3 \to \BPGL_3$), so it will suffice to show that $Y$ is contained in the image of $X$.

By definition, curves in $Y$ have very ample canonical.
To see that the image of $Y$ is contained in $X$ we need to know that the following additional property holds: The images of the marked points under the canonical model impose independent conditions on quartics. In the irreducible nodal case,
the configuration of markings imposes independent conditions by Lemma \ref{planecurvesurj} with $d = 4$, since $8 \leq 3d - 1$. In the case of a cubic union a line with all markings on the cubic, the configuration of markings imposes independent conditions on cubics by Lemma \ref{planecurvesurj} with $d = 3$, since $8 \leq 3d - 1$. Then, these points also imposes independent conditions on quartics.

Now, Lemma \ref{gdn} implies $A^*(Y) = R^*(Y)$. An argument as in Lemma \ref{gck} shows that $Y$ has the CKgP.
\end{proof}

Using excision and Lemma \ref{Kstrat}, 
the previous two lemmas combine to give the following.
\begin{lem} \label{38}
The union $\M^\circ = Y \cup Z$ has the CKgP and $A^* = R^*$. Consequently, $\Mb_{3,8}$ has the CKgP and $A^* = R^*$. 
\end{lem}

\section{Faber--Pandharipande Hurwitz (FPH) cycles} \label{fphsec}
To access pointed curves of genus $g \geq 4$, we use Hurwitz spaces with marked points. We will study degree three covers in Section \ref{trigsec} and degree four covers in Section \ref{tetsec}. This section contains some preliminary results about push forwards of certain geometrically defined cycles, which hold for covers of any degree $k$. In this section, we assume the characteristic of the ground field is $> k$. In our main theorems, we shall only need results for covers of degree $\leq 5$, so this only excludes characteristic $2, 3$ and $5$.

In \cite{FaberPandharipande}, Faber--Pandharipande study pushforwards to $\overline{\M}_{g,n}$ of cycles from
Hurwitz spaces of covers with specified ramification behavior. This gives rise to many examples of effective, geometrically defined cycles on $\overline{\M}_{g,n}$ that are tautological.

\subsection{Hurwitz spaces with marked specified ramification}
Fix a genus $g$ and a degree $k$. Let
$\mu^1,\dots,\mu^m$
be a collection of partitions of $k$ satisfying 
\begin{equation} \label{rh}
2g-2 + 2k =\sum_{i=1}^{m} (k-\ell(\mu^i))
\end{equation}
where $\ell$ is the length of the partition.
Faber--Pandharipande study the space  $H_g(\mu^1, \ldots, \mu^m)$, which parametrizes degree $k$, genus $g$ covers $C\rightarrow \p^1$, where $C$ has marked ramification profile $\mu^1,\dots,\mu^m$ in $m$ marked fibers (and no other ramification).
For example, in the case all $\mu^i=(2,1,\ldots, 1)$, the space $H_g((2,1,\ldots,1)^{2g-2+2k})$ parametrizes simply branched covers together with a marking of all of the points in the ramified fibers. 

The admissible covers compactification of this space, which they call $\overline{H}_g(\mu^1, \ldots, \mu^m)$, has a natural map $\rho:\overline{H}_g(\mu^1, \ldots, \mu^m)\rightarrow \Mb_{g,\sum \ell(\mu^i)}$ by taking the stabilization of the marked source curve. Faber--Pandharipande prove that the pushforward of the fundamental class of $\overline{H}_g(\mu^1, \ldots, \mu^m)$ is tautological. 
We might then consider the image $Z$ of this cycle in $\Mb_{g,n}$, after forgetting some subset of the marked points. When the composition 
\[\overline{H}_g(\mu^1, \ldots, \mu^m)\rightarrow \Mb_{g,\sum \ell(\mu^i)} \to \Mb_{g,n}\]
is generically finite onto $Z$, it follows that $[Z]\in R^*(\Mb_{g,n})$. Restricting to the locus of smooth curves we find:

\begin{lem} \label{fpc}
If the map $H_g(\mu^1, \ldots, \mu^m) \to \M_{g,\sum \ell(\mu^i)} \to \M_{g,n}$ has generically finite fibers, then
the closure of the image of $H_g(\mu^1, \ldots, \mu^m)$ in $\M_{g,n}$ has tautological fundamental class.
\end{lem}

\begin{definition}
When $H_g(\mu^1, \ldots, \mu^m) \to \M_{g,\sum \ell(\mu^i)} \to \M_{g,n}$  has generically finite fibers, we will call the closure of the image a \emph{Faber--Pandharipande Hurwitz cycle}, or \emph{FPH cycle}, for short.  Lemma \ref{fpc} says FPH cycles are tautological. 
\end{definition}

\begin{rem}
Although Faber--Pandharipande state their results over $\mathbb{C}$,
the main tool they use is equivariant localization for torus actions, which holds over arbitrary algebraically closed fields (see \cite{EdidinGrahamLocalization}).
Their results thus remain valid over any algebraically closed field of characteristic $> k$ (taking the characteristic $> k$ is needed to ensure that there will be no wild ramification so the Riemann--Hurwitz formula \eqref{rh} continues to hold).
\end{rem}

\subsection{Our Hurwitz spaces}
We find it more convenient to work with the pointed Hurwitz spaces $\H_{k,g,n}$ parametrizing degree $k$, genus $g$ covers $C \to \pp^1$ with $n$ distinct marked points $p_1, \ldots, p_n$ and no restrictions on the ramification behavior of the cover.
(The marked points need not be ramification points, though they are not prohibited from being so.)
There is a basic fiber diagram
\begin{center}
\begin{tikzcd}
\H_{k,g,n} \arrow{d}[swap]{\beta_n} \arrow{r}{\epsilon} &\H_{k,g} \arrow{d}{\beta} \\
\M_{g,n} \arrow{r} & \M_{g}
\end{tikzcd}
\end{center}
We define $\M_{g,n}^k$ to be the closure of the image $\beta_n$. That is, $\M_{g,n}^k$ is the locus of curves of gonality $\leq k$. There are proper maps
\begin{equation} \label{bnp} \beta_n': \H_{k,g,n}\smallsetminus \beta_n^{-1}(\M_{g,n}^{k-1}) \to \M_{g,n} \smallsetminus \M_{g,n}^{k-1}
\end{equation}
along which we can push forward cycles. The image of $\beta_n'$ is the stratum $\M_{g,n}^k \smallsetminus \M_{g,n}^{k-1}$ of curves of gonality exactly $k$.

We now translate the work of Faber--Pandharipande to show that push forwards of certain cycles along $\beta_n'$ are tautological. First, on $\H_{k,g}$, we define the cycle
\begin{align} \label{tadef}
T^a &= \{C \to \pp^1 : \text{$C \to \pp^1$ has a point of ramification order $\geq a+2$}\}.
\end{align}
The subvariety $T^a \subset \H_{k,g}$ has codimension $a$.
Note that $T^0 = \H_{k,g}$ since every genus $g \geq 1$ cover of $\pp^1$ has a point of ramification.
Then, on $\H_{k,g,n}$, we define
\begin{align}
R_i &=\{(C \to \pp^1, p_1, \ldots, p_n) : p_i \text{ is a ramification point of $C \to \pp^1$}\} \label{ridef}
\end{align}
Each $R_i$ has codimension $1$. Let
\begin{center}
\begin{tikzcd}
\C \arrow{rr}{\alpha} \arrow{dr}{f} && \P \arrow{dl}{\pi} \\
& \H_{k,g,n} \arrow[bend left = 30, "\sigma_i"]{ul}
\end{tikzcd}
\end{center}
be the universal diagram over $\H_{k,g,n}$. If $R \subset \C$ is the ramification divisor of $\alpha$, then
\begin{equation} \label{ri}
[R_i] = \sigma_i^*[R] = \sigma_i^*c_1(\omega_{\alpha}) = \sigma_i^*c_1(\omega_f) - \sigma_i^*c_1(\omega_\pi)
\end{equation}
By slight abuse of notation, we shall also write $[R_i]$ and $\epsilon^*[T^a]$ for the restrictions of these classes to $\H_{k,g,n} \smallsetminus \beta_n^{-1}(\M_{g,n}^{k-1})$.

The main goal of this subsection is to prove the following.
\begin{prop} \label{ourfp}
Let $i_1, \ldots, i_j$ be a subset of distinct indices in $1, \ldots, n$ and let $a+2 \leq k$. Then the push forward $\beta_{n'*}([R_{i_1}] \cdots [R_{i_j}] \cdot \epsilon^*[T^a])$ is tautological.
\end{prop}

First let us reduce to the case $\{i_1, \ldots, i_j\} = \{1, \ldots, n\}.$
Consider the fibered square
\begin{center}
    \begin{tikzcd}
    \H_{k,g,n} \smallsetminus \beta_n^{-1}(\M_{g,n}^{k-1}) \arrow{d}[swap]{\beta_n'} \ar[bend left = 20, rr, "\epsilon_n"] \arrow{r}{ q} & \H_{k,g,j} \smallsetminus \beta_j^{-1}(\M_{g,j}^{k-1}) \arrow{d}{\beta_j'} \arrow{r}[swap]{\epsilon_j} & \H_{k,g} \\
    \M_{g,n} \smallsetminus \M_{g,n}^{k-1} \arrow{r}[swap]{p} & \M_{g,j} \smallsetminus \M_{g,j}^{k-1},
    \end{tikzcd}
\end{center}
where $p$ and $q$ forget the markings for indices not in $i_1, \ldots, i_j$.
From this, we see 
\[\beta_{n'*}([R_{i_1}] \cdots [R_{i_j}]\cdot \epsilon_n^*[T^a]) = 
\beta_{n'*}q^* ([R_1] \cdots [R_j] \cdot \epsilon_j^*[T^a]) =
p^*\beta_{j'*}([R_1] \cdots [R_j] \cdot \epsilon_j^*[T^a]).\]
Because $p^*$ preserves tautological classes, it suffices to treat the case $j = n$.

Next, we wish to show that the intersection $R_1 \cap \cdots \cap R_n \cap \epsilon^{-1}(T^a)$ is dimensionally transverse.

\begin{lem} \label{silly}
The map $R_1 \cap \cdots \cap R_n \to \H_{k,g}$ has finite fibers. In particular, every component of $R_1 \cap \cdots \cap R_n$ has codimension $n$ inside $\H_{k,g,n}$ and dominates $\H_{k,g}$.

More generally, $R_1 \cap \cdots \cap R_n \cap \epsilon^{-1}(T^a) \to T^a$ has finite fibers. Every component of $R_1 \cap \cdots \cap R_n \cap \epsilon^{-1}(T^a)$ has codimension $n+a$ inside $\H_{k,g,n}$ and dominates $T^a$.
\end{lem}
\begin{proof}
The fiber of $R_1 \cap \cdots \cap R_n$ over a degree $k$ cover $C \to \pp^1$ corresponds to the finitely many ways to choose and label $n$ distinct ramification points. 
The map $\H_{k,g,n} \to \H_{k,g}$ has relative dimension $n$.
Being an intersection of $n$ divisors, every component of $R_1 \cap \cdots \cap R_n$ has dimension at least $\dim \H_{k,g,n} - n = \dim \H_{k,g}$. Because all fibers of $R_1 \cap \cdots \cap R_n \to \H_{k,g}$ have dimension $0$, it follows that every component of $R_1 \cap \cdots \cap R_n$ has codimension exactly $n$ and dominates $\H_{k,g}$. The same argument applies over any subvariety $T^a \subset \H_{k,g}$.
\end{proof}

\begin{proof}[Proof of Proposition \ref{ourfp}]
By Lemma \ref{silly}, the product of classes is represented by the class of the scheme-theoretic intersection
\[[R_1] \cdots [R_n]\cdot \epsilon^*[T^a] = [R_1 \cap \cdots \cap R_n \cap \epsilon^{-1}(T^a)].\]
If all fibers of $R_1 \cap \cdots \cap R_n \cap \epsilon^{-1}(T^a) \to \M_{g,n}$ are positive-dimensional, then the push forward is zero, hence tautological. Let us therefore assume that $R_1 \cap \cdots \cap R_n \cap \epsilon^{-1}(T^m) \to \M_{g,n}$ is generically finite.
Then the push forward $\beta_{n'*}([R_1 \cap \cdots \cap R_n \cap \epsilon^{-1}(T^a)])$ is some multiple of $[\beta_n'(R_1 \cap \cdots \cap R_n \cap \epsilon^{-1}(T^a))]$, so it suffices to show  $[\beta_n'(R_1 \cap \cdots \cap R_n \cap \epsilon^{-1}(T^a))]$ is tautological on $\M_{g,n} \smallsetminus \M_{g,n}^{k-1}$. We do this by realizing the closure in $\M_{g,n}$ of $\beta_n'(R_1 \cap \cdots \cap R_n \cap \epsilon^{-1}(T^a))$ as an FPH cycle.

Let us first treat the case $a = 0$, as it is simpler. Let $\mu^i = (2, 1, \ldots, 1)$ for $i = 1, \ldots, m := 2g - 2 + 2k$. Let $N = \sum \ell(\mu^i) = m(k-1)$.
Consider the diagram
\begin{center}
\begin{tikzcd}
& & R_1 \cap \cdots \cap R_n \arrow{d} \\
H_g(\mu^1, \ldots, \mu^m) \arrow{dr}[swap]{\rho^\circ} \arrow{r} & \H_{k,g,N} \arrow{d} \arrow{r} & \H_{k,g,n} \arrow{d}{\beta} \arrow{r}{\epsilon} & \H_{k,g} \arrow{d} \\
& \M_{g,N} \arrow{r} & \M_{g,n} \arrow{r} & \M_g
\end{tikzcd}
\end{center}
where the map $\M_{g,N} \to \M_{g,n}$ remembers the
$n$ points whose indices were dedicated to marked ramification points in $H_g(\mu^1, \ldots, \mu^m)$.
Let $\H_{k,g}^s \subset \H_{k,g}$ denote the open locus of simply branched covers. By construction, the image of
 \[H_g(\mu^1, \ldots, \mu^m) \to \H_{k,g,N} \to \H_{k,g,n}\]
is $\epsilon^{-1}(\H_{k,g}^s) \cap R_1 \cap \cdots \cap R_n$.
By Lemma \ref{silly}, we see that $\epsilon^{-1}(\H_{k,g}^s) \cap R_1 \cap \cdots \cap R_n$ is dense in $R_1 \cap \cdots \cap R_n$. It follows that the closure of $\beta_n(R_1 \cap \cdots \cap R_n)$ is equal to the closure of the image of $H_g(\mu^1, \ldots, \mu^m) \to \M_{g,N} \to \M_{g,n}$. Now, $H_g(\mu^1, \ldots, \mu^m) \to \H_{k,g}$ has finite fibers (corresponding to the ways of labeling all points in fibers with ramification) so the map $H_g(\mu^1, \ldots, \mu^m) \to \epsilon^{-1}(\H_{k,g}^s) \cap R_1 \cap \cdots \cap R_n$ also has finite fibers. We assumed the map $R_1 \cap \cdots \cap R_n \to \M_{g,n}$ has generically finite fibers, so $H_g(\mu^1, \ldots, \mu^m) \to \M_{g,n}$ must have generically finite fibers. Therefore, 
the closure of $\beta_n(R_1 \cap \cdots \cap R_n)$ is an FPH cycle, so it is tautological (see Lemma \ref{fpc}).

Next we treat the case $a > 0$. Now the intersection $R_1 \cap \cdots \cap R_n \cap \epsilon^{-1}(T^a)$ has $n+1$ components: For each $1 \leq i\leq n$, there is a component $Y_i$ where $p_i$ is ramified to order $a+2$. There is also a component $Y_0$ in which the general cover has simple ramification at all $p_i$ (and another, unmarked point has ramification order $a+2$).

\begin{center}
    \includegraphics[width=5.9in]{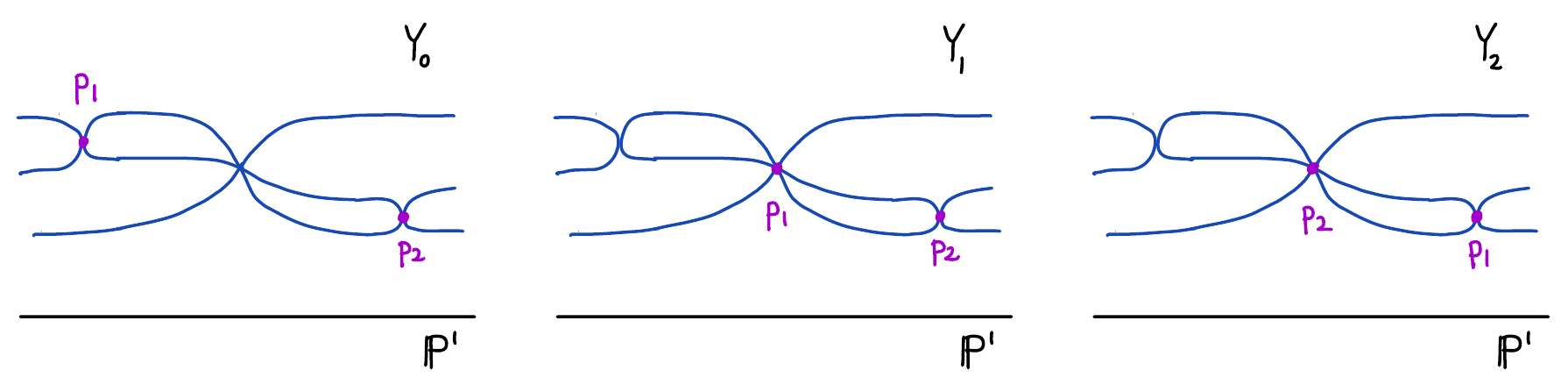}
    \end{center}
    \label{Y0Y1Y2}
\noindent
For each component $Y_i$, 
we want to find a map $H_g(\mu^1, \ldots, \mu^m)$ to $\H_{k,g,n}$ whose image is dense in that component. Then, the rest of the proof will proceed as before.

Let $\mu^0 = (a+2,1, \ldots, 1)$ and $\mu^i = (2, 1, \ldots, 1)$ for $i = 1, \ldots, m := 2g - 2 + 2k - a-1$. In addition, let $N := \sum \ell(\mu^i) = m(k-1) - a$.
We form $H_g(\mu^0, \mu^1, \ldots, \mu^m)$ which has a natural map to $\H_{k,g,N}$. Suppose that our markings on $\H_{k,g,N}$ are such that $p_i$ is the unique ramification point in the $\mu^i$-ramified fiber. Let $\H_{k,g,N} \to \H_{k,g,n+1}$ be the map that forgets all markings besides $p_0, \ldots, p_n$.

\begin{center}    \includegraphics[width=5in]{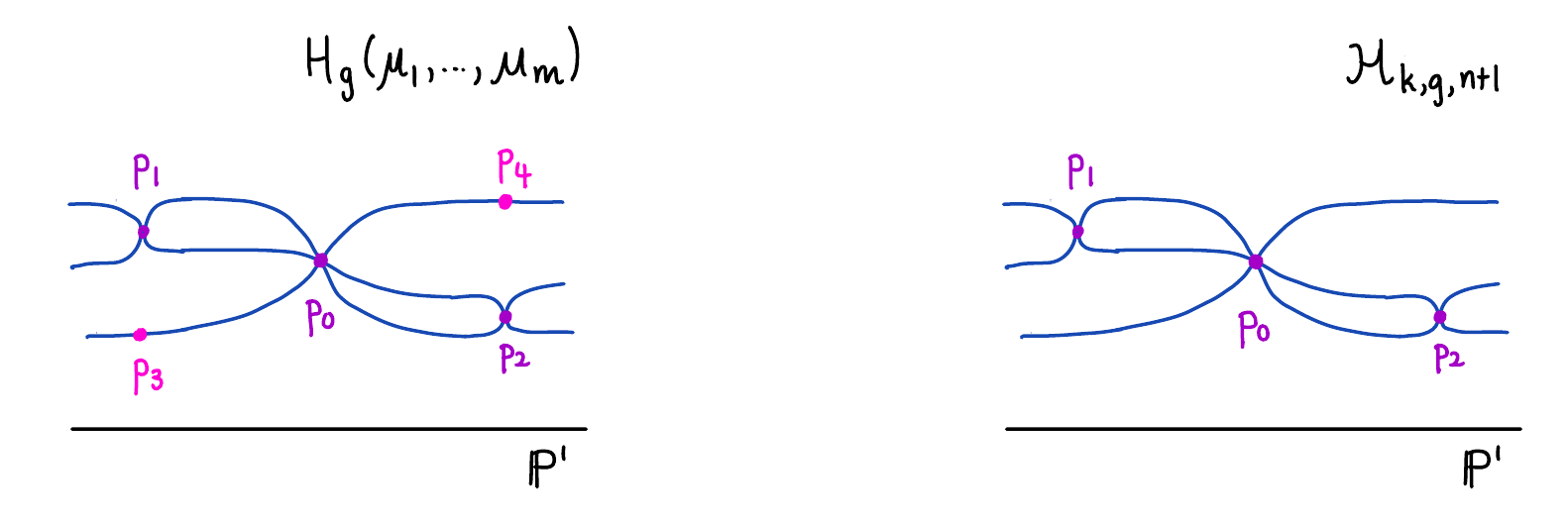}
\end{center}
\noindent
We realize the different components $Y_i$ as the image of 
\begin{equation} \label{compo} 
H_g(\mu^0, \mu^1, \ldots, \mu^m)\to \H_{k,g,N} \to \H_{k,g,n+1} \xrightarrow{f_i} \H_{k,g,n}
\end{equation}
under composition with different forgetful maps $f_i: \H_{k,g,n+1} \to \H_{k,g,n}$.
To obtain $Y_0$, we should let $f_0$ forget the marking $p_0$. To obtain $Y_i$, we use $f_i$ which forgets the marking $p_i$ and relabels $p_0$ to be the new $i^{\mathrm{th}}$ marking.

Let $T^{a,\circ} \subset T^a$ be the open locus of covers with one point of ramification order exactly $a+2$ and no other points of higher order ramification (this will replace the role of $\H_{k,g}^s$ in the $a = 0$ case).
By construction, the image of \eqref{compo} is $\epsilon^{-1}(T^{a,\circ}) \cap Y_i$. By Lemma \ref{silly}, this is dense in $Y_i$. Now the proof finishes as before, realizing the closure of each $\beta_n(Y_i)$ as an FPH cycle. In particular, the closure of $\beta_n(R_1 \cap \cdots \cap R_n \cap \epsilon^{-1}(T^a)) = \beta_n(Y_0) \cup \cdots \cup \beta_n(Y_n)$ is a union of tautological cycles, hence tautological.
\end{proof}

\section{Trigonal curves} \label{trigsec}
Let $\M_{g,n}^3 \subseteq \M_{g,n}$ be the locus of curves of gonality at most $3$.
Our goal in this section is to show that $\M_{g,n}^3$ has the CKgP and
that all classes supported on the locus $\M_{g,n}^3 \subseteq \M_{g,n}$ are tautological when $n$ is sufficiently small compared to $g$.
To do so, we work with the Hurwitz space $\H_{3,g,n}$ parametrizing degree $3$ covers $C\rightarrow \p^1$ with $n$ distinct marked points. 

\subsection{Strategy}
We can of course pullback tautological classes from $\M_{g,n}$ to $\H_{3,g,n}$ along 
\begin{equation} \label{pullback} R^*(\M_{g,n}) \to A^*(\M_{g,n}) \to A^*(\H_{3,g,n}).
\end{equation}
If \eqref{pullback} were surjective, then by the push-pull formula and the fact that the fundamental class of $\M_{g,n}^3$ is tautological (by \cite[Proposition 1]{FaberPandharipande}), we would be done.
When $n = 0$, our work in \cite[Theorem 1.1]{part2} shows \eqref{pullback} is surjective for all $g \geq 4$.
However, we shall see soon that \eqref{pullback} is \emph{not} surjective for any $n \geq 1$!
Nevertheless, the push-pull formula allows us to reduce to understanding generators for $A^*(\H_{3,g,n})$ as a \emph{module} over the image of \eqref{pullback}.

In the next subsection, we give a construction of $\H_{3,g,n}$ when $n$ is sufficiently small. This gives rise to generators for 
$A^*(\H_{3,g,n})$ that come from the algebra of structure theorems for degree $3$ covers. Then we relate these generators to FPH cycles, studied in the previous section.

\subsection{Construction of the stack}
The basic idea is that, using structure theorems for degree $3$ covers, the moduli space of smooth, genus $g$ triple covers of $\pp^1$ is the same as the moduli space of smooth curves of an appropriate class on a Hirzebruch surface. We constructed the latter in our earlier work \cite{part1} as an open substack of a vector bundle over a moduli space of vector bundles on $\pp^1$. To briefly review this construction, let $\B$ be the moduli stack of rank $2$, degree $g+2$ globally generated vector bundles on $\pp^1$-bundles. (This $\B$ corresponds to the choice of Hirzebruch surface). 
The stack $\B$ comes equipped with a universal $\pp^1$ bundle $\pi: \P \to \B$ and a universal rank $2$ bundle $\E$ on $\P$. The projectivization $\gamma: \pp \E^\vee \to \P$ is our universal Hirzebruch surface.
Our trigonal curves live inside $\pp \E^\vee$ with class
\[\L := \gamma^*\det \E^\vee \otimes \O_{\pp \E^\vee}(3).\]
We then define the vector bundle $\U := \gamma_*(\gamma^*\L)= \det \E^\vee \otimes \Sym^3 \E$ on $\P$, and set $\B' \subset \B$ to be the open substack where $\U$ is globally generated on fibers of $\P \to \B$.
By cohomology and base change, $\mathcal{X} := \pi_* \U|_{\B'}$ is a vector bundle over $\B'$ whose fibers are equations of curves in the appropriate class on the corresponding Hirzebruch surface.
We prove in \cite[Lemma 5.1]{part1} that $\H_{3,g}$ is equivalent to the open substack of $\mathcal{X}$ where the vanishing of the equation is a smooth curve in each fiber over $\B'$.

In a similar fashion, the moduli space $\H_{3,g,n}$ of smooth, genus $g$ triple covers of $\pp^1$ with $n$ marked points is the same as  the moduli space of smooth curves of an appropriate class on a Hirzebruch surface with marked points. 
To mark $n$ points on Hirzebruch surfaces, we consider $(\pp \E^\vee)^n = \pp \E^\vee \times_{\B} \cdots \times_{\B} \pp \E^\vee$.
We then need to consider curves in the correct class that pass through the specificed points.
The key point is that when $n$ is sufficiently small, $n$ \emph{distinct} points on a smooth curve of our desired class \emph{impose independent conditions} on the linear system for $\L$. This allows us to construct $\H_{3,g,n}$ in a similar manner, as an open substack of a vector bundle over an open substack of $(\pp \E^\vee)^n$.

Let $\eta_i:(\p \E^{\vee})^n\rightarrow \p\E^{\vee}$ be the projection onto the $i^{\text{th}}$ factor.
Let us continue the same notation as before, 
so we have a digram
\begin{center}
\begin{tikzcd}
\epsilon^* \mathcal{X} \arrow{d} \arrow{r} & \gamma^*\pi^* \mathcal{X} \arrow{d} \arrow{rr} & & \mathcal{X} \arrow{d} \\
(\pp \E^\vee)^n \ar[bend right = 20, rrr, swap, "\epsilon"] \arrow{r}{\eta_i} & \pp \E^\vee \arrow{r}{\gamma} & \P \arrow{r}{\pi} \arrow{r} & \B'.
\end{tikzcd}
\end{center}
From the definition of $\mathcal{X}$, there is a natural evaluation map on $\pp \E^\vee$:
\[
\gamma^*\pi^*\X = \gamma^*\pi^*\pi_*\gamma_* \L \rightarrow \L.\]
 We can then take the direct sum of the pullbacks of these evaluation maps
\begin{equation}\label{trigonalevaluation}
\epsilon^*\X \rightarrow \bigoplus_{i=1}^{n}\eta_i^*\L.
\end{equation}
Define $\Y \subset \epsilon^*\mathcal{X}$ to be the preimage of the zero section under \eqref{trigonalevaluation}.
The stack $\Y$ parametrizes tuples $(E, C, p_1, \ldots, p_n)$ where $E$ is a rank $2$ vector bundle on $\pp^1$; $C \subset \pp E^\vee$ is the vanishing of a section of $\O_{\pp E^\vee}(3) \otimes \det E^\vee$; and $p_1, \ldots, p_n \in C$ is a collection of $n$ (not necessarily distinct) points.
The Hurwitz space $\H_{3,g,n}$ parametrizes the same kinds of tuples but where $C$ is smooth of dimension $1$ and $p_1, \ldots, p_n$ are distinct.
As such, there is a natural open inclusion of $\H_{3,g,n}$ in $\Y$.
Note that
the stack $\Y$ is \emph{not} a vector bundle over $(\pp \E^\vee)^n$ because its fiber dimension jumps when the collection of points does not impose independent conditions.

In order to gain a better understanding of $\Y$, and from it $\H_{3,g,n}$, we want to know when the map \eqref{trigonalevaluation} is surjective. By cohomology and base change, we can reduce to checking surjectivity on the fibers. 
\begin{lem} \label{tlem}
Let $E$ be a rank $2$ degree $g+2$ vector bundle on $\pp^1$ and let $\Gamma\subset 
\p E^{\vee}$ be a collection of $n\leq g+7$ distinct points. Suppose there exists a smooth curve $C$ in class $L := \O_{\p E^{\vee}}(3)\otimes \gamma^*\det E^{\vee}$ such that $\Gamma\subset C$. Then the evaluation map
\[
H^0(\p E^{\vee},L)\rightarrow H^0(\Gamma,L|_{\Gamma})
\]
is surjective. 
\end{lem}
\begin{proof}
The evaluation map factors as
\begin{align*}
H^0(\p E^{\vee},L) &\rightarrow H^0(C,L|_{C})\rightarrow H^0(\Gamma, L|_{\Gamma}).
\end{align*}
The first map is surjective because $H^1(\p E^{\vee},\O)=0$.
In the construction of $C \subset \p E^\vee$, the $\O_{\pp E^\vee}(1)$ restricts to the relative canonical of $C \to \pp^1$, that is
$\O_{\pp E^\vee}(1)|_C \cong \omega_C \otimes \gamma^* \omega_{\pp^1}^\vee$ (see \cite[Theorem 2.1(2)]{CasnatiEkedahl} or \cite[Example 3.12]{part1}).
In particular,
\[\deg L = \deg (\O_{\p E^{\vee}}(3)|_{C}) + \deg(\gamma^*\det E^{\vee})|_{C} = 3(2g-2 +2\cdot 3) - (g+2) \cdot 3 = 3g + 6.\]
Hence, $h^0(C, L)=\chi(C, L) = 2g+7$. Now consider the exact sequence on $C$
\[
0\rightarrow L(-\Gamma)\rightarrow L\rightarrow L|_{\Gamma}\rightarrow 0.
\]
The map $H^0(C, L) \to H^0(\Gamma, L|_{\Gamma})$ will be surjective if $H^1(C, L(-\Gamma)) = 0$.
By Serre duality,
$H^1(C, L(-\Gamma))=H^0(C, \omega_C \otimes L^\vee (\Gamma))$, which vanishes so long as
\[0 > \deg( \omega_C \otimes L^\vee (\Gamma)) = 2g - 2 - (3g+6) + n,\]
or equivalently, so long as $n \leq g+7$.
\end{proof}

Now, define $U \subset (\pp \E^\vee)^n$ be the locus over which the evaluation map \eqref{trigonalevaluation}  is surjective. We know that $U$ is open, but a priori it could be empty, and in fact $U$ will be empty when $n$ is too large.
However, when $n \leq g + 7$,
Lemma \ref{tlem} shows that the image of $\H_{3,g,n}$ inside $(\pp \E^\vee)^n$ is contained in $U$.
Hence, we find that the inclusion $\H_{3,g,n} \subset \Y$ factors through $\H_{3,g,n} \subset \Y|_U$.
Moreover, by definition of $\Y$ and $U$, we have that $\Y|_{U}$ is the kernel of the restriction of \eqref{trigonalevaluation} to $U$.
In particular, $\Y|_U$ is a vector bundle over $U$. This implies the following.

\begin{lem} \label{3ckgp}
$\H_{3,g,n}$ has the CKgP for $n \leq g+7$.
\end{lem}
\begin{proof}
We have maps \[\H_{3,g,n} \subset \Y|_U \to U \subset (\pp \E^\vee)^n \to \P \to \B'\]
where each $\subset$ is an open inclusion and each arrow is a vector bundle or product of projective bundles. Using Lemmas \ref{open}, \ref{affbunCKgP} and \ref{grassmann}, it thus suffices to show that $\B'$ has the CKgP. But $\B'$ is a quotient of an open subset of affine space by $\GL_{g+2} \times \GL_{g+4} \times \BSL_2$ (see \cite[Proposition 4.2]{part1}). Equivalently $\B'$ is an open inside a vector bundle over $\BGL_{g+2} \times \BGL_{g+4} \times \BSL_2$. Now we are done because $\BGL_d$ and $\BSL_2$ have the CKgP (Lemma \ref{bgs}).
\end{proof}

\subsection{Generators for the Chow ring}
Our construction of $\H_{3,g,n}$ gives rise to generators for its Chow ring.
\begin{lem}
\label{trigens}
For $n \leq g+7$, 
there is a surjection $A^*((\pp \E^\vee)^n) \to A^*(\H_{3,g,n})$.
\end{lem}
\begin{proof}
By excision, we have a series of surjections
\[A^*((\pp \E^\vee)^n) \rightarrow A^*(U) \cong A^*(\Y|_U) \rightarrow A^*(\H_{3,g,n}).\]
The middle map is an isomorphism because $\Y|_{U}$ is a vector bundle over $U$.
\end{proof}

Let $z_i := \eta_i^*\gamma^* c_1(\O_{\P}(1))$ and $\zeta_i := \eta_i^*c_1(\O_{\pp \E^\vee}(1))$.
The classes
$z_1, \ldots, z_n, \zeta_1, \ldots, z_n$ generate $A^*((\pp \E^\vee)^n)$ as an \emph{algebra} over $A^*(\B')$. 
The projective bundle theorem gives us relations
\begin{equation} \label{pgrels}
\zeta_i^2 + c_1(\E^\vee) \zeta_i + c_2(\E^\vee) = 0
\qquad \text{and} \qquad 
z_i^2 + c_2(\pi_*\O_{\P}(1)) = 0.
\end{equation}
Taking into account these relations, we see $A^*((\pp \E^\vee)^n)$ is generated 
as a \emph{module} over $A^*(\B')$ by monomials of the form
\begin{equation} \label{monos}
z_1^{a_1} z_2^{a_2} \cdots z_n^{a_n} \zeta_1^{b_1}\zeta_2^{b_2} \cdots \zeta_n^{b_n} \qquad a_i, b_i \leq 1.
\end{equation}
Combining this with Lemma \ref{trigens}, we see that the classes in \eqref{monos} generate $A^*(\H_{3,g,n})$ as a module over the image of $A^*(\B') \to A^*(\H_{3,g,n})$.

Meanwhile,
the map $\H_{3,g,n} \to \B'$ factors through $\H_{3,g,n} \to \H_{3,g} \to \B'$. We proved in \cite[Theorem 1.1(1)]{part2} that $A^*(\B') \to A^*(\H_{3,g})$ is surjective and that for $g \geq 4$, the image is generated by the pullback of $\kappa_1$ along $\H_{3,g} \to \M_g$. 
Thus, the image of $A^*(\B') \to A^*(\H_{3,g,n})$ is just the subring generated by $\tilde{\kappa}_1$,  defined as the pullback of $\kappa_1$ along $\H_{3,g,n} \to \M_{g,n} \to \M_g$. (Note that the tautological rings of $\M_{g,n}$ are closed under pullbacks, so $\tilde{\kappa}_1$ on $\H_{3,g,n}$ is the pullback of a tautological class on $\M_{g,n}$.)
Hence, we obtain:

\begin{lem} \label{mgens}
Suppose $g \geq 4$ and $n \leq g + 7$.
The classes in \eqref{monos} generate
$A^*(\H_{3,g,n})$ as a module over the subring generated by $\tilde{\kappa}_1$.
\end{lem}

Our task now is to relate $z_i$ and $\zeta_i$ to some more geometrically defined classes.

\begin{lem} \label{psilem}
We have $\psi_i = \zeta_i -2z_i$.
\end{lem}
\begin{proof}
It suffices to treat the one pointed case.
Because $\H_{3,g,1}$ is the universal curve $f: \C \to \H_{3,g}$, we see that $\psi = c_1(\omega_f)$. The result now follows from the first equation of \cite[Example 3.12]{part1}.
\end{proof}

Using Lemmas \ref{pgrels} and \ref{psilem} together with the relations in \eqref{pgrels}, we obtain:
\begin{lem} \label{mmgens} Suppose $g \geq 4$ and $n \leq g + 7$. As a module over the subring generated by $\tilde{\kappa}_1, \psi_1, \ldots \psi_n$, we have $A^*(\H_{3,g,n})$ is generated by the monomials $\zeta_1^{b_1} \zeta_2^{b_2} \cdots \zeta_n^{b_n}$ with $b_i \leq 1$.
\end{lem}

Thus, our goal is to show that these monomials $\zeta_1^{b_1} \zeta_2^{b_2} \cdots \zeta_n^{b_n}$ push forward to tautological classes on $\M_{g,n}$. To do so, we use some geometrically defined cycles that represent these monomials.
Recall the divisors
\[R_i = \{(C, p_1, \ldots, p_n): p_i \text{ is a ramification point of $C \to \pp^1$}\}\]
introduced in Section \ref{fphsec}.
The first step is to identify the fundamental class of $R_i$ in terms of our generators.
\begin{lem} \label{rclass}
We have $[R_i] = \zeta_i$.
\end{lem}
\begin{proof}
The map $\eta_i: \H_{3,g,n} \to  \pp \E^\vee$ is the composition of the $i$th section $\sigma_i: \H_{3,g,n} \to \C$ with the universal embedding $\iota: \C \hookrightarrow \pp \E^\vee$. By construction of the embedding, $\iota^*\O_{\pp \E^\vee}(1) \cong \omega_{\alpha}$, where $\alpha: \C \to \P$ is the universal degree $k$ map \cite[Theorem 2.1]{CasnatiEkedahl}.
Thus, $\zeta_i = \eta_i^*c_1(\O_{\pp \E^\vee}(1)) = \sigma_i^*c_1(\omega_\alpha)$, so the result follows from \eqref{ri}.
\end{proof}

Let $\beta_n: \H_{3,g,n} \to \M_{g,n} \smallsetminus \M_{g,n}^2$ be the forgetful map. We now show that our module generators for $A^*(\H_{3,g,n})$ over $R^*(\M_{g,n})$ push forward to tautological classes.

\begin{lem} \label{thepoint}
Let $i_1, \ldots, i_j$ be a subset of distinct indices in $1, \ldots, n$. Then $\beta_{n*}(\zeta_{i_1} \cdots \zeta_{i_j})$ is tautological in $\M_{g,n} \smallsetminus \M_{g,n}^2$.
\end{lem}
\begin{proof}
By Lemma \ref{rclass}, we have 
$\zeta_{i_1} \cdots \zeta_{i_j} = [R_{i_1}] \cdots [R_{i_j}]$, so the result follows from Proposition \ref{ourfp}.
\end{proof}

We now conclude that the trigonal locus has our desired properties, when the number of marked points is sufficiently small.
\begin{lem} \label{trig-thm}
Suppose $g \geq 4$ and $n \leq g + 7$. Then $\M_{g,n}^3$ has the CKgP and all classes supported on $\M_{g,n}^3$ are tautological.
\end{lem}
\begin{proof}
By Theorem \ref{hyperelliptic}, we know that when $n \leq 2g+6$ all classes supported on $\M_{g,n}^2$ are tautological and $\M_{g,n}^2$ has the CKgP, so it remains to show the same for $\M_{g,n}^3 \smallsetminus \M_{g,n}^2$.
The map $\H_{3,g,n} \to \M_{g,n}^3 \smallsetminus \M_{g,n}^2$ is proper and surjective.
Combining Lemma \ref{3ckgp} and \ref{surjCKgP}, we see that $\M_{g,n}^3 \smallsetminus \M_{g,n}^2$ has the CKgP. Applying Lemma \ref{Kstrat}, the union $\M_{g,n}^3 = (\M_{g,n}^3 \smallsetminus \M_{g,n}^2) \cup \M_{g,n}^2$ has the CKgP.

Meanwhile, the map $\H_{3,g,n} \to \M_{g,n}^3 \smallsetminus \M_{g,n}^2$ 
induces a surjection on Chow groups with rational coefficients. In particular, every class supported on  $\M_{g,n}^3 \smallsetminus \M_{g,n}^2 \subset \M_{g,n} \smallsetminus \M_{g,n}^2$ lies in the image of $(\beta_n)_*: A^*(\H_{3,g,n}) \to A^*(\M_{g,n} \smallsetminus \M_{g,n}^2)$. Therefore, it suffices to show that the image of $(\beta_n)_*$ is tautological.
By Lemma \ref{mmgens} and the push-pull formula, we are reduced to showing that $(\beta_n)_*(\zeta_{i_1} \cdots \zeta_{i_j})$ is tautological for any subset of distinct indices $i_1, \ldots, i_j$. This is proved in Lemma \ref{thepoint}.
\end{proof}

\begin{rem}
One can adapt the arguments of this section to work in the case $g=3$.
We would then have two separate proofs showing that classes supported on $\M_{3,n}^3\smallsetminus \M_{3,n}^2$ are tautological, but they seem to work for different values of $n$. The first, Corollary \ref{3open} using plane curves, gives $n\leq 11$, and the second, Lemma \ref{trig-thm} using the Hurwitz space, gives $n\leq 10$. This discrepancy is actually a mirage: every $g^1_3$ on a non-hyperelliptic curve $C$ of genus $3$ comes from projecting from a point $p\in C$ in the canonical embedding. Therefore, the Hurwitz space $\H_{3,3}$ really corresponds to $\M^3_{3,1}\smallsetminus \M^2_{3,1}$, which explains why the plane model appears to give one more marked point than the Hurwitz space model.
\end{rem}

In genus $4$, we have $\M_{4,n} = \M_{4,n}^3$, so Lemma \ref{trig-thm} gives the following.

\begin{lem} \label{4end}
If $n \leq 11$, then $\M_{4,n}$ has the CKgP and $A^* = R^*$.
\end{lem}

\section{Tetragonal curves} \label{tetsec}
In this section, we will attempt to show that classes supported on $\M_{g,n}^4$ are tautological. We do not succeed to show this in full for arbitrary $g$. In fact, it is not true in general that all classes supported on $\M_{g,n}^4$ are tautological \cite{vanZelm}. But the argument will suffice for the low genus cases treated this paper.

\subsection{Strategy}
The strategy is quite similar to the trigonal case, but with a few extra complications. We first define a subring of $A^*(\H_{4,g,n})$ that deserves to be called the tautological ring. By relating the generators to geometrically defined classes we show that these tautological classes on $\H_{4,g,n}$ push forward to tautological classes on $\M_{g,n}$.
Although $A^*(\H_{4,g,n})$ is not generated by tautological classes in general, 
we find a large open substack $\H'_{4,g,n}\subset \H_{4,g,n}$ such that $A^*(\H_{4,g,n}')$ is generated by tautological classes. The next task is to understand the complement of $\H'_{4,g,n}\subset \H_{4,g,n}$. In general, there may be \emph{non-tautological} classes coming from this complement. In particular, the bielliptic locus in genus $12$ lies in the complement. For the genera of interest in this paper, however, we will be able to give a complete description of the complement and show that any Chow classes coming from it are indeed tautological. For example, in genus $6$, the complement consists of pointed hyperelliptic, trigonal, and plane quintic curves, but we know from previous sections that any classes supported on these loci are tautological.

\subsection{Tautological classes} \label{ts}
Using structure theorems for degree $4$ covers \cite{CasnatiEkedahl,part1}, the moduli space of degree $4$ covers of $\p^1$ is the same as the moduli space of codimension $2$ complete intersections in a certain class on a $\p^2$-bundle over $\p^1$. Let $\B$ denote the moduli stack of pairs of globally generated vector bundles of rank $3$ and degree $g+3$ and rank $2$ and degree $g+3$ on $\p^1$ together with an isomorphism $\det \E\cong \det \F$, as constructed in \cite[Definition 5.2]{part1}. There is a universal $\p^1$- bundle $\pi:\P\rightarrow \B$, a universal rank $3$ degree $g+3$ bundle $\E$, and a universal rank $2$ degree $g+3$ bundle $\F$ on $\P$. Let $\gamma:\p \E^{\vee}\rightarrow \P$ be the universal $\p^2$ bundle. 

 There is a natural morphism $\H_{4,g} \to \B$ that sends a degree $4$ cover to its associated vector bundles on $\pp^1$. 
Moreover, the Casnati--Ekedahl structure theorem defines an embedding of the universal curve $\C$ over $\H_{4,g}$ into the projectivization of the pullback of $\E$ \cite{CasnatiEkedahl}. Hence, we obtain a diagram
\begin{equation} \label{CEemb}
\begin{tikzcd}
\C  \arrow{d}[swap]{f} \arrow{r}{a} & \pp \E^\vee \arrow{d} \\
\H_{4,g} \arrow{r} & \B.
\end{tikzcd}
\end{equation}

For $n \geq 1$, we get a diagram
\begin{equation} \label{abd}
\begin{tikzcd}
\C_n \arrow{r}{b} \arrow{d} & \C  \arrow{d} \arrow{r}{a} & \pp \E^\vee \arrow{d} \\
\arrow[bend left = 50, "\sigma_i"]{u} \H_{4,g,n} \arrow{r} & \H_{4,g} \arrow{r} & \B.
\end{tikzcd}
\end{equation} 
For each $i$, the composition $a \circ b \circ \sigma_i$ defines a map $\H_{4,g,n} \to \pp \E^\vee$. Informally, this is defined by sending a pointed curve to the image of the $i$th marking under the Casnati--Ekedahl embedding. Taking the product of these maps for $i = 1, \ldots, n$ we obtain a diagram
\begin{equation} \label{HB}
\begin{tikzcd}
\H_{4,g,n} \arrow{r} \arrow{d} & (\pp \E^\vee)^n \arrow{d} \\
\H_{4,g} \arrow{r} & \B.
\end{tikzcd}
\end{equation}
In \cite{part1, part2}, we studied the map $A^*(\B) \to A^*(\H_{4,g})$; we called the image the tautological subring for $\H_{4,g}$ and related it to the tautological ring on $\M_g$. The pointed version of this is to study the image of $A^*((\pp \E^\vee)^n) \to A^*(\H_{4,g,n})$, which similarly deserves to be called the tautological subring for $\H_{4,g,n}$. The key result we need is Lemma \ref{pushlem} below, which relates tautological classes on $\H_{4,g,n}$ to tautological classes on $\M_{g,n}$.

The first step is to describe the image of $A^*(\B) \to A^*(\H_{4,g,n})$ in terms the classes of ramification loci $T^a$ defined in \eqref{tadef}.
\begin{lem} \label{m1lem}
The image of $A^*(\B) \to A^*(\H_{4,g,n})$ is generated as a module over $\qq[\kappa_1, \kappa_2]$ by $1, [T^1]$ and $[T^2]$. 
\end{lem}
\begin{proof}
The pullback map $A^*(\B) \to A^*(\H_{4,g,n})$ factors through 
\[A^*(\B) \to A^*(\H_{4,g}) \to A^*(\H_{4,g,n}).\] 
The image of $A^*(\B) \to A^*(\H_{4,g})$ is the tautological subring $R^*(\H_{4,g})$ by \cite[Theorem 3.10]{part1}.
Next, we claim that $R^*(\H_{4,g})$ is generated as a module over $\qq[\kappa_1, \kappa_2]$ by the classes of the ramification loci $1,T^1$ and $T^2$.
We proved something similar to this in \cite[Lemma 7.9]{part2}, where $T^1$ and $T^2$ are denoted $T$ and $U$ respectively.
Here, we want to use a slightly different basis. By \cite[Lemma 7.9(1)]{part2}, we know $R^1(\H_{4,g})$ is spanned by $\kappa_1, [T^1]$. We claim $R^2(\H_{4,g})$ is spanned by $\kappa_1[T^1], \kappa_1^2, \kappa_2,$ and $[T^2]$. Using \cite[Equation 7.5
and Lemmas 7.6 and 7.7]{part2}, we can write each of these classes in terms of the usual spanning set $a_1^2, a_2'a_1, a_2'^2, a_3'$ of \cite[Corollary 5.6(2)]{part2}. One then checks that the change of basis matrix is full rank for all $g$. It has determinant
\[
\frac{1658880g - 2985984}{g^2+4g+3},
\]
which never vanishes for $g$ an integer. From \cite[Lemma 7.9(3)]{part2} it is clear $\kappa_1^2[T^1], \kappa_1^3, \kappa_1[T^2]$ also form a spanning set of $R^3(\H_{4,g})$. Finally, \cite[Lemma 7.9(4) and (5)]{part2} complete the proof that $1, [T^1], [T^2]$ generate $R^*(\H_{4,g})$ as a module over $\qq[\kappa_1, \kappa_2]$.
\end{proof}

The next step is to study the generators for $A^*((\pp \E^\vee))$ over $A^*(\B)$.
Let $\eta_i : (\pp \E^\vee)^n \to \pp \E^\vee$ be the $i$th projection map.
Let $z_i := \eta_i^*\gamma^* c_1(\O_{\P}(1))$ and $\zeta_i := \eta_i^*c_1(\O_{\pp \E^\vee}(1))$.
The classes
$z_1, \ldots, z_n, \zeta_1, \ldots, \zeta_n$ generate $A^*((\pp \E^\vee)^n)$ as an \emph{algebra} over $A^*(\B)$. Let $c_2 = c_2(\pi_*\O_{\P}(1)) \in A^*(\B)$.
The projective bundle theorem gives us relations
\begin{equation} \label{pgrels2}
\zeta_i^3 + c_1(\E^\vee) \zeta_i^2 + c_2(\E^{\vee})\zeta_i + c_3(\E^\vee) = 0
\qquad \text{and} \qquad 
z_i^2 + c_2 = 0.
\end{equation}
Taking into account these relations, we see $A^*((\pp \E^\vee)^n)$ is generated 
as a \emph{module} over $A^*(\B)$ by monomials of the form
\begin{equation} \label{monostet}
z_1^{a_1} z_2^{a_2} \cdots z_n^{a_n} \zeta_1^{b_1}\zeta_2^{b_2} \cdots \zeta_n^{b_n} \qquad a_i\leq 1 \qquad b_i\leq 2.
\end{equation}

Now we relate $z_i$ and $\zeta_i$ to the $\psi$ and ramification classes.

\begin{lem} \label{psilem4}
We have 
\begin{enumerate}
    \item $\zeta_i = \psi_i +2z_i$,
    \item $\zeta_i^2 = \psi_i^2 +2\zeta_i\psi_i-4c_2$
\end{enumerate}
\end{lem}
\begin{proof}
The first statement is proved exactly the same way as Lemma \ref{psilem}. 
For the second, we simply square the equality in (1).
\end{proof}

\begin{lem} \label{r4lem}
We have $[R_i]=\zeta_i$.
\end{lem}
\begin{proof}
Letting $a$ be the map in \eqref{abd}, we have $a^*\O_{\pp \E^\vee}(1) = \omega_\alpha$ where $\alpha: \C \to \P$ is the universal degree $k$ map. Hence, $\eta_i^*c_1(\O_{\pp \E^\vee}(1)) = \sigma_i^*c_1(\omega_{\alpha})$ so the result follows from \eqref{ri}.
\end{proof}

Now we are ready to give generators for tautological classes on $\H_{4,g,n}$.
\begin{lem} \label{modgens}
The following classes generate the image of $A^*((\pp \E^\vee)^n) \to A^*(\H_{4,g,n})$ as a module over $R^*(\M_{g,n})$:
\begin{equation} [R_{i_1}] \cdots [R_{i_j}] \cdot [T^c] \qquad i_1, \ldots, i_j \text{ distinct}, \qquad c \leq 2. \end{equation}
\end{lem}
\begin{proof}
Let $S^* \subset A^*(\H_{4,g,n})$ be the graded subring generated by $\kappa_1, \kappa_2, \psi_1, \ldots, \psi_n$ and the image of $A^*(\B) \to A^*(\H_{4,g,n})$.
By Lemma \ref{m1lem}, the ring $S^*$ is generated as a module over $\qq[\kappa_1, \kappa_2, \psi_1, \ldots, \psi_n]$ by $[T^c]$ for $c \leq 2$. By Lemma \ref{r4lem}, we have $[R_i] = \zeta_i$.
It therefore suffices to show that the monomials $\zeta_1^{b_1} \cdots \zeta_n^{b_n}$ with $b_i \leq 1$ generate $A^*(\H_{4,g,n})$ as a module over $S^*$.
To accomplish this, we show that each monomial of the form in \eqref{monostet} lies in the span of elements of $S^*$ times monomials of the form $\zeta_1^{b_1} \cdots \zeta_n^{b_n}$ with $b_i \leq 1$.
First note that by \eqref{pgrels2}, we have
\begin{align}
0 &= \zeta_i^3 - (a_1 + (g+3)a_2'z_i) \zeta_i^2 + (a_2 + a_2'z_i)\zeta - (a_3 + a_3'z_i),\notag \\
\intertext{where $a_i, a_i' \in A^*(\B)$, and so will also be viewed as elements of $S^*$. By Lemma \ref{psilem4}, we can substitute in $z_i = \frac{1}{2}(\zeta_i - \psi_i)$ to get}
0 &= \zeta_i^3 - (a_1 + \tfrac{1}{2}(g+3)(\zeta_i - \psi_i)) \zeta_i^2 + \tfrac{1}{2}(a_2 + a_2'(\zeta_i - \psi_i))\zeta_i - (a_3 + a_3'(\zeta_i - \psi_i)). \notag
\intertext{Because the coefficient $1 + \frac{1}{2}(g+3)$ of $\zeta_i^3$ above is non-zero, we may solve to find}
\zeta_i^3 &\in S^1 \cdot \zeta_i^2 + S^2 \cdot \zeta_i + S^3. \label{zetaisub}
\end{align}

Now,
suppose we are given a monomial of the form in \eqref{monostet}. Using Lemma \ref{psilem4}(1), we can replace each $z_i$ with $\frac{1}{2}(\zeta_i - \psi_i)$. This will give us a sum of terms of the form $S^* \cdot \zeta_1^{b_1'} \cdots \zeta_n^{b_n'}$ where $b_i' \leq b_i + a_i \leq 3$. We can then use \eqref{zetaisub} to eliminate any terms with a $\zeta_i^3$ and obtain a sum of terms of the form $S^* \cdot \zeta_1^{b_1''} \cdots \zeta_n^{b_n''}$ where $b_i'' \leq 2$. Finally, we can use Lemma \ref{psilem4}(2) to eliminate each $\zeta_i^2$
that appears, thus leaving a sum of terms of the form $ S^* \cdot \zeta_1^{b_1'''} \cdots \zeta_n^{b_n'''}$ where $b_i''' \leq 1$.
\end{proof}

We can now prove our desired result, which roughly says tautological classes on $\H_{4,g,n}$ push forward to tautological classes on $\M_{g,n}$.

\begin{lem} \label{pushlem}
Let 
$\beta_n': \H_{4,g,n} \smallsetminus \beta_n^{-1}(\M_{g,n}^3) \to \M_{g,n} \smallsetminus \M_{g,n}^{3}$ be as in \eqref{bnp}. Suppose that $x \in 
A^*( \H_{4,g,n}\smallsetminus \beta_n^{-1}(\M_{g,n}^{3}))$ lies in the image of 
\[A^*(\pp \E^\vee)^n) \to A^*(\H_{4,g,n}) \to A^*( \H_{4,g,n}\smallsetminus \beta_n^{-1}(\M_{g,n}^{3})).\]
Then $\beta_{n*}'x$ is tautological on $\M_{g,n} \smallsetminus \M_{g,n}^3$.
\end{lem}

\begin{proof}
Using the push-pull formula, it suffices to show that each of the generators in \eqref{modgens} pushes forward to a tautological class on $\M_{g,n} \smallsetminus \M_{g,n}^3$. This is Proposition \ref{ourfp}.
\end{proof}

Our next task is to find a large open substack $\H'_{4,g,n} \subset \H_{4,g,n}$ whose Chow ring is generated by tautological classes.

\subsection{Construction of a certain substack}
The tetragonal curves live inside $\p \E^{\vee}$ as the vanishing loci of sections of the rank $2$ bundle $\W:=\O_{\p \E^{\vee}}(2)\otimes \gamma^*\F^{\vee}$.
Set $\U=\gamma_*(\W)=\Sym^2 \E\otimes \F^{\vee}$. Let $\B'\subset \B$ be the complement of the support of $R^1\pi_*\U$. By cohomology and base change, $\X:=\pi_*\U|_{\B'}$ is a vector bundle over $\B'$ whose fibers correspond to equations of tetragonal curves. Define $\H'_{4,g} \subset \mathcal{X}$ to be the open substack representing smooth tetragonal curves. 

We construct the moduli space $\H'_{4,g,n}$ similarly. Throughout this section, we shall work over $\B'$, so when we write $\P$ and $\pp \E^\vee$, we mean $\P|_{\B'}$ and $\pp \E^\vee|_{\pi^{-1}\B'}$. Let $\eta_i:(\p \E^{\vee})^n\rightarrow \p \E^{\vee}$ be the projection to the $i^{\text{th}}$ factor, so we have a diagram:
\begin{center}
\begin{tikzcd}
\epsilon^* \mathcal{X} \arrow{r} \arrow{d} & \gamma^*\pi^* \mathcal{X} \arrow{d} \arrow{rr} & & \mathcal{X} \arrow{d} \\
(\pp \E^\vee)^n \arrow{r}{\eta_i} \ar[bend right = 20, "\epsilon", swap, rrr] & \pp \E^\vee \arrow{r}{\gamma} & \P \arrow{r}{\pi} \arrow{r} & \B'.
\end{tikzcd}
\end{center}
We have evaluation maps
in a rank $2$ bundle \begin{equation} \label{regulareval}
\gamma^*\pi^*\X\rightarrow \W
\end{equation}
on $\pp \E^\vee$.
Pulling back to $(\pp \E^\vee)^n$, and taking a sum over the factors, we obtain an evaluation map on $(\pp \E^\vee)^n$
\begin{equation}\label{evaltet}
\epsilon^*\X\rightarrow \bigoplus_{i=1}^n \eta_i^*\W.
\end{equation}
We define $\Y\subset \epsilon^*\X$ to be the preimage of the zero section under \eqref{evaltet}. The stack $\Y$ parametrizes tuples $(E,F,C,p_1,\dots,p_n)$ such that $C \subset \pp E^\vee$ is the vanishing locus of a section of $F^\vee \otimes \O_{\pp E^\vee}(2)$ which contains $p_1, \ldots, p_n$. There is an open inclusion $\H_{4,g,n} \subset \Y$ corresponding to the open conditions that $C$ is smooth and $p_1, \ldots, p_n$ are distinct.

We want to know when \eqref{evaltet} is surjective. By cohomology and base change, we reduce to the case of a single curve $C$, which is a complete intersection on $\p E^{\vee}$ in class $\O_{\p E^{\vee}}(2)\otimes \gamma^*F^{\vee}$. 

\begin{lem}
Let $E=\O(e_1)\oplus \O(e_2)\oplus \O(e_3)$ and $F=\O(f_1)\oplus \O(f_2)$ with $f_1 \leq f_2$ be vector bundles of degree $g+3$ on $\pp^1$.
Let $\Gamma\subset \p E^{\vee}$ be a collection of $n\leq 4f_1-2g+1$ distinct points. Suppose that there exists a smooth, irreducible curve $C$ that is the zero locus of a section of $W:=\O_{\p E^{\vee}}(2)\otimes \gamma^*F^{\vee}$ such that $\Gamma\subset C$. Then the evaluation map
\[
H^0(\p E^{\vee}, W)\rightarrow H^0(\Gamma, W|_{\Gamma})
\]
is surjective.
\end{lem}
\begin{proof}
We factor the evaluation map through
\[
H^0(\p E^{\vee},W)\rightarrow H^0(C,W|_{C})\rightarrow H^0(\Gamma,W|_{\Gamma}).
\]
The first map is surjective because $H^1(\p E^{\vee},\O)=0$.
By \cite[Theorem 2.1(2)]{CasnatiEkedahl} or \cite[Example 3.12]{part2}, $\O_{\p E^{\vee}}(1)$ restricts to $\omega_C\otimes \gamma^*\omega_{\p^1}^{\vee}$. 
Consider the exact sequence on $C$
\[
0\rightarrow W|_{C}(-\Gamma)\rightarrow W\rightarrow W|_{\Gamma}\rightarrow 0.
\]
The map $H^0(C,W)\rightarrow H^0(\Gamma,W|_{\Gamma})$ is surjective if $H^1(C,W|_{C}(-\Gamma))=0$. By Serre duality, 
\[
H^1(C,W|_{C}(-\Gamma))=H^0(C,\omega_C\otimes W|_{C}^{\vee}(\Gamma))^{\vee}.
\]
For $i=1,2$, set $W_{f_i}=\O_{\p E^{\vee}}(2)\otimes \gamma^* \O(f_i)$, so that $W=W_{f_1}\oplus W_{f_2}$.
We have 
\[
\deg W_{f_i}|_C=\deg (\O_{\p E^{\vee}}(2)|_C\otimes \gamma^*\O(-f_i)|_{C})=2(2g-2+4\cdot 2)-4f_i=4g+12-4f_i.
\]
Then 
\[
H^0(C,\omega_C\otimes W|_{C}^{\vee}(\Gamma))=H^0(C,\omega_C\otimes W_{f_1}|_{C}^{\vee}(\Gamma))\oplus H^0(C,\omega_C\otimes W_{f_2}|_{C}^{\vee}(\Gamma)).
\]
Thus, the required vanishing will occur if
\[
0>\deg \omega_C\otimes W_{f_i}|_{C}^{\vee}(\Gamma)=2g-2-(4g+12-4f_i)+n=4f_i+n-2g-14,
\]
equivalently if, $n\leq 2g+13-4f_i$.
Since $f_1+f_2=g+3$ and $f_1 \leq f_2$, we can rewrite this as
$n\leq 2g+13-4(g+3-f_1)=4f_1-2g+1$. 
\end{proof}

\begin{rem}
In the degree $3$ case, our curve was defined by the vanishing of a section of a line bundle. In degree $4$, we have the zero locus of a section of a rank $2$ vector bundle $W = F^\vee \otimes \O_{\pp E^\vee}(2)$, and the splitting type of the rank $2$ bundle $F$ enters into the calculation of how many points we can mark and still know we are imposing independent conditions.
Observe also that $f_1 \leq \frac{g+3}{2}$, so the number of marked points we can hope to get with this technique is bounded (independent of $g$) via $n \leq 4f_1 - 2g + 1 \leq 7$.
\end{rem}

It turns out that the locus of covers for which $f_1$ is small often corresponds curves with special geometry properties (see Lemmas \ref{h4} and \ref{h44} below). One might hope to access these special curves by other means and then focus on their complement, the locus of covers with $f_1 \geq f$ for some $f$. We therefore make the following definition.
\begin{definition}
Let $\H_{4,g,n}^f \subset \H_{4,g,n}'$ be the union of Casnati--Ekedahl strata where $f_1 \geq f$. Equivalently, $\H_{4,g,n}^f = \H_{4,g,n}' \smallsetminus \Supp R^1\pi_*\F(f-1)$.
\end{definition}

Recall that we write $\beta: \H_{4,g,n} \to \M_{g,n}$ for the map to the moduli space of curves.
\begin{lem} \label{h4}
We have $\H_{4,5,n}^4 = \H_{4,5,n} \smallsetminus \beta^{-1}(\M_{5,n}^3)$.
\end{lem}
\begin{proof}
By \cite[Section 4.3]{789}, the complement of $\H_{4,5}'$ is $\beta^{-1}(\M_{5}^2)$, which is also the locus where $f_1 = 2$. Meanwhile, if $f_1 = 3$, then we claim $C$ is trigonal. First, note that $\Psi_2$ in  \cite[Section 4.3]{789} is actually empty because $f_2 > 2e_2$, which forces such curves to be singular by \cite[Equation 4.7]{789}. Therefore, if $F$ has 
splitting type $(3, 5)$ then $E$ has splitting type $(2, 3, 3)$. In this case,
the image of $\p E^{\vee}$ in $\p^4$ is a cone over $\pp^1 \times \pp^1$ and $C$ passes through the cone point because $p_{11} = 0$ and $\deg q_{11} = 1$. Projecting from the cone point sends $C$ to a curve of bidegree $(4, 3)$ on $\pp^1 \times \pp^1$, so we see that $C$ is trigonal.
\end{proof}

There is also a nice geometric explanation of the locus where $f_1$ is small in genus $6$. Let $PQ_n \subset \M_{6,n}$ be the preimage of the locus $PQ \subset \M_6$ of smooth plane quintics.
\begin{lem} \label{h44}
We have $\H_{4,6,n}^4 = \H_{4,6,n} \smallsetminus \beta^{-1}(\M_{6,n}^3 \cup PQ_n)$.
\end{lem}
\begin{proof}
The allowed pairs of splitting types were determined in \cite[Section 4.3]{789}.
The claim follows immediately from Lemma 4.5 and Remark 4.7 of \cite{789}.
\end{proof}

Thus, for our purposes in genus $5$ and $6$ at least, it will be enough to know information about $\H_{4,g,n}^f$ for $f=4$.
To gain a better understanding of $\H_{4,g,n}^f$, define $U \subset (\pp \E^\vee)^n$ to be the locus over which the evaluation map \eqref{evaltet}  is surjective. We know that $U$ is open, but a priori it could be empty, and in fact $U$ will be empty when $n$ is too large.
However, when $n \leq 4f-2g+1$,
Lemma \ref{tlem} shows that the image of $\H^f_{4,g,n}$ inside $(\pp \E^\vee)^n$ is contained in $U$.
Hence, we find that the inclusion $\H^f_{4,g,n} \subset \Y$ factors through $\H^f_{4,g,n} \subset \Y|_U$.
Moreover, by definition of $\Y$ and $U$, we have that $\Y|_{U}$ is the kernel of the restriction of \eqref{evaltet} to $U$.
In particular, $\Y|_U$ is a vector bundle over $U$. In summary, there are maps
\begin{equation}  \label{opse} \H_{4,g,n}^f \subset \Y|_U \rightarrow U \subset (\pp \E^\vee)^n \rightarrow \P \rightarrow \B' \subset \B
\end{equation}
where the $\subset$'s are open inclusions and the arrows are vector bundles or projective bundles. This leads to the following.

\begin{lem} \label{h4ckgp}
If $n \leq 4f - 2g + 1$, then $\H_{4,g,n}^f$ has the CKgP
\end{lem}
\begin{proof}
Considering \eqref{opse} and our standard lemmas about the CKgP, it suffices to show that $\B$ has the CKgP. By definition, $\B$ is an open inside a line bundle over $\V_{3,g+3} \times_{\BSL_2} \V_{2,g+2}$
(defined in \cite[Equation 4.1]{part1}), so by Lemmas \ref{open} and \ref{affbunCKgP} it suffices to show 
$\V_{3,g+3} \times_{\BSL_2} \V_{2,g+2}$ has the CKgP. Then, \cite[Equation 4.1]{part1} realizes this stack as a quotient of an open subset of affine space by a product of $\GL_d$'s and $\SL_2$. Applying Lemmas \ref{open} and \ref{affbunCKgP} once more, together with Lemma \ref{bgs}
completes the proof.
\end{proof}

\subsection{Generators for the Chow ring}
Just like in the degree $3$ case, our demonstration that $\H_{4,g,n}^f$ has the CKgP also gives rise to natural generators for its Chow ring. In particular, they are all tautological in the sense discussed in Section \ref{ts}

\begin{lem} \label{tetgens}
For $n \leq 4f-2g+1$, the map $A^*((\pp \E^\vee)^n) \to A^*(\H^f_{4,g,n})$ is surjective.
\end{lem}
\begin{proof}
By excision, we have a series of surjections
\[A^*((\pp \E^\vee)^n \rightarrow A^*(U) \cong A^*(\Y|_U) \rightarrow A^*(\H^f_{4,g,n}).\]
The middle map is an isomorphism because $\Y|_{U}$ is a vector bundle over $U$.
\end{proof}


\subsection{Conclusion when $g = 5$} \label{g5}
In genus $5$, we 
use push forward along the proper map 
\[\beta'_n: \H_{4,5,n}^4 = \H_{4,5,n} \smallsetminus \beta^{-1}(\M_{5,n}^3) \to \M_{5,n} \smallsetminus \M_{5,n}^3.\]
(The equality of domains above is Lemma \ref{h4}.)
When $f = 4$, we have $4f - 2g +1 = 7$.

\begin{lem} \label{5end}
If $n \leq 7$, then $\M_{5,n}$ has the CKgP and $A^* = R^*$.
\end{lem}
\begin{proof}
We already know from Lemma \ref{trig-thm}, that $\M_{5,n}^3$ has the CKgP and that all classes supported on it are tautological. It thus remains to show that $\M_{5,n} \smallsetminus \M_{5,n}^3$ has the CKgP and $A^* = R^*$.

We have a proper surjective map $\beta_n': \H_{4,5,n}^4 \to \M_{5,n} \smallsetminus \M_{5,n}^3$. We know $\H_{4,5,n}^4$ has the CKgP (Lemma \ref{h4ckgp}), so $\M_{5,n} \smallsetminus \M_{5,n}^3$ has the CKgP by Lemma \ref{surjCKgP}.
We also know that $\beta_{n*}'$ induces a surjection on Chow groups, so it suffices to show that the image of $\beta_{n*}'$ is tautological.
The result now follows from Lemmas \ref{tetgens} and \ref{pushlem}.
\end{proof}

\subsection{Conclusion when $g = 6$} \label{g6}
When $g = 6$, we use push forwards along the proper map \[\beta': \H_{4,6,n}^4 = \H_{4,6,n} \smallsetminus \beta^{-1}(\M_{6,n}^3 \cup PQ_n) \to \M_{6,n} \smallsetminus (\M_{6,n}^3 \cup PQ_n)\]
where $PQ_n$ is the plane quintic locus, defined as the image of $\G_{5,n} \to \M_{6,n}$ from Section \ref{planesec}.
(The equality of domains above is Lemma \ref{h4}.)
When $f = 4$, we have $4f - 2g + 1 = 5$

\begin{lem} \label{6end} 
If $n \leq 5$, then $\M_{6,n}$ has the CKgP and $A^* = R^*$
\end{lem}
\begin{proof}
We already know from Lemma \ref{trig-thm}, that $\M_{5,n}^3$ has the CKgP and that all classes supported on it are tautological.
The fundamental class of $PQ_n$ is the pullback of the class of the plane quintic locus on $\M_6$, which is tautological by an argument of Faber \cite{Faberconjecture} (as applied in \cite{PenevVakil}).
By Lemma \ref{gdn}, we know that $A^*(PQ_n)$ is generated by restrictions of tautological classes. By the push-pull formula, all classes supported on $PQ_n$ are tautological. Furthermore, $PQ_n$ has the CKgP by Lemma \ref{gck}.

It thus remains to show that the stratum $\M_{6,n} \smallsetminus (\M_{6,n}^3 \cup PQ_n)$ has the CKgP and $A^* = R^*$. The proof is now nearly identical to Lemma \ref{5end}.
The only difference is to note that, by excision, Proposition \ref{ourfp} remains valid upon restricting to the complement of $\beta^{-1}(PQ_n)$ in the domain and $PQ_n$ in the target.
\end{proof}

\subsection{Final step: filling in}
Lemmas \ref{4end}, \ref{5end} and \ref{6end} add three more columns of open circles (pictured in blue) to our chart, giving the chart on the left below.  We also have an open circle for $\M_7$ (pictured in black) by our previous work \cite[Theorem 1.1]{789}. (It is clear that each of the strata involved in the proof given in \cite{789} has the CKgP, being a quotient of an open subset of affine space by a suitable group.)
Using the filling criteria (Lemma \ref{fc1} for $\Mb_{g,n}$ and Lemma \ref{tc} for $\M_{g,n}^{\ct}$) we obtain the chart on the right.

\begin{center}
\begin{tikzpicture}[scale = .4]

\draw[->] (8.5, 7) -- (12.5,7);
\node[scale=.8] at (10.5, 8) {Filling criteria};
\node[scale=.8] at (10.5, 6) {version 1 / ct};

\node[scale=.5] at (1, -.7) {$1$};
\node[scale=.5] at (2, -.7) {$2$};
\node[scale=.5] at (3, -.7) {$3$};
\node[scale=.5] at (4, -.7) {$4$};
\node[scale=.5] at (5, -.7) {$5$};
\node[scale=.5] at (6, -.7) {$6$};
\node[scale=.5] at (7, -.7) {$7$};

\node[scale=.5] at (-.7, 1) {$1$};
\node[scale=.5] at (-.7, 2) {$2$};
\node[scale=.5] at (-.7, 3) {$3$};
\node[scale=.5] at (-.7, 4) {$4$};
\node[scale=.5] at (-.7, 5) {$5$};
\node[scale=.5] at (-.7, 6) {$6$};
\node[scale=.5] at (-.7, 7) {$7$};
\node[scale=.5] at (-.7, 8) {$8$};
\node[scale=.5] at (-.7, 9) {$9$};
\node[scale=.5] at (-.7, 10) {$10$};
\node[scale=.5] at (-.7, 11) {$11$};
\node[scale=.5] at (-.7, 12) {$12$};
\node[scale=.5] at (-.7, 13) {$13$};

\draw[->] (0, 0) -- (8, 0);
\draw[->] (0, 0) -- (0, 14);
\node[scale=.9] at (8.4,0) {$g$};
\node[scale=.9] at (0, 14.4) {$n$};
\draw (-.1, 1) -- (.1, 1);
\draw (-.1, 2) -- (.1, 2);
\draw (1, -.1) -- (1, .1);
\draw (2, -.1) -- (2, .1);
\draw (3, -.1) -- (3, .1);
\draw (4, -.1) -- (4, .1);
\filldraw (0, 3) circle (4pt);
\filldraw (0, 4) circle (4pt);
\filldraw (0, 5) circle (4pt);
\filldraw (0, 6) circle (4pt);
\filldraw (0, 7) circle (4pt);
\filldraw (0, 8) circle (4pt);
\filldraw (0, 9) circle (4pt);
\filldraw (0, 10) circle (4pt);
\filldraw (0, 11) circle (4pt);
\filldraw (0, 12) circle (4pt);
\filldraw (0, 13) circle (4pt);
\filldraw (1, 1) circle (4pt);
\filldraw (1, 2) circle (4pt);
\filldraw (1, 3) circle (4pt);
\filldraw (1, 4) circle (4pt);
\filldraw (1, 5) circle (4pt);
\filldraw (1, 6) circle (4pt);
\filldraw (1, 7) circle (4pt);
\filldraw (1, 8) circle (4pt);
\filldraw (1, 9) circle (4pt);
\filldraw (1, 10) circle (4pt);
\node[scale = .6, color = red] at (1, 11) {$\times$};
\node[scale = .6, color = red] at (1, 12) {$\times$};
\node[scale = .6, color = red] at (1, 13) {$\times$};
\filldraw (2, 0) circle (4pt);
\filldraw (2, 1) circle (4pt);
\filldraw (2, 2) circle (4pt);
\filldraw (2, 3) circle (4pt);
\filldraw (2, 4) circle (4pt);
\filldraw (2, 5) circle (4pt);
\filldraw (2, 6) circle (4pt);
\filldraw (2, 7) circle (4pt);
\filldraw (2, 8) circle (4pt);
\filldraw (2, 9) circle (4pt);
\draw (2, 10) circle (4pt);
\filldraw (3, 0) circle (4pt);
\filldraw (3, 1) circle (4pt);
\filldraw (3, 2) circle (4pt);
\filldraw (3, 3) circle (4pt);
\filldraw (3, 4) circle (4pt);
\filldraw (3, 5) circle (4pt);
\filldraw (3, 6) circle (4pt);
\filldraw (3, 7) circle (4pt);
\filldraw (3, 8) circle (4pt);
\draw (3, 9) circle (4pt);
\draw (3, 10) circle (4pt);
\draw (3, 11) circle (4pt);
\filldraw[color = white] (4,0) circle (4pt);
\draw[color=blue] (4, 0) circle (4pt);
\draw[color=blue] (4, 1) circle (4pt);
\draw[color=blue] (4, 2) circle (4pt);
\draw[color=blue] (4, 3) circle (4pt);
\draw[color=blue] (4, 4) circle (4pt);
\draw[color=blue] (4, 5) circle (4pt);
\draw[color=blue] (4, 6) circle (4pt);
\draw[color=blue] (4, 7) circle (4pt);
\draw[color=blue] (4, 8) circle (4pt);
\draw[color=blue] (4, 9) circle (4pt);
\draw[color=blue] (4, 10) circle (4pt);
\draw[color=blue] (4, 11) circle (4pt);
\filldraw[color=white] (6, 0) circle (4pt);
\draw[color=blue] (6, 0) circle (4pt);
\draw[color=blue] (6, 1) circle (4pt);
\draw[color=blue] (6, 2) circle (4pt);
\draw[color=blue] (6, 3) circle (4pt);
\draw[color=blue] (6, 4) circle (4pt);
\draw[color=blue] (6, 5) circle (4pt);

\filldraw[color=white] (5, 0) circle (4pt);
\draw[color=blue] (5, 0) circle (4pt);
\draw[color=blue] (5, 1) circle (4pt);
\draw[color=blue] (5, 2) circle (4pt);
\draw[color=blue] (5, 3) circle (4pt);
\draw[color=blue] (5, 4) circle (4pt);
\draw[color=blue] (5, 5) circle (4pt);
\draw[color=blue] (5, 6) circle (4pt);
\draw[color=blue] (5, 7) circle (4pt);
\filldraw[color=white] (7, 0) circle (4pt);
\draw (7, 0) circle (4pt);

\node[color=blue, scale = .8] at (4, -2) {Lemmas 9.11, 10.12 and 10.13};
\end{tikzpicture}
\hspace{.25in}
\begin{tikzpicture}[scale = .4]
\node[scale=.5] at (1, -.7) {$1$};
\node[scale=.5] at (2, -.7) {$2$};
\node[scale=.5] at (3, -.7) {$3$};
\node[scale=.5] at (4, -.7) {$4$};
\node[scale=.5] at (5, -.7) {$5$};
\node[scale=.5] at (6, -.7) {$6$};
\node[scale=.5] at (7, -.7) {$7$};

\node[scale=.5] at (-.7, 1) {$1$};
\node[scale=.5] at (-.7, 2) {$2$};
\node[scale=.5] at (-.7, 3) {$3$};
\node[scale=.5] at (-.7, 4) {$4$};
\node[scale=.5] at (-.7, 5) {$5$};
\node[scale=.5] at (-.7, 6) {$6$};
\node[scale=.5] at (-.7, 7) {$7$};
\node[scale=.5] at (-.7, 8) {$8$};
\node[scale=.5] at (-.7, 9) {$9$};
\node[scale=.5] at (-.7, 10) {$10$};
\node[scale=.5] at (-.7, 11) {$11$};
\node[scale=.5] at (-.7, 12) {$12$};
\node[scale=.5] at (-.7, 13) {$13$};
\draw[->] (0, 0) -- (8, 0);
\draw[->] (0, 0) -- (0, 14);
\node[scale=.9] at (8.4,0) {$g$};
\node[scale=.9] at (0, 14.4) {$n$};
\draw (-.1, 1) -- (.1, 1);
\draw (-.1, 2) -- (.1, 2);
\draw (1, -.1) -- (1, .1);
\draw (2, -.1) -- (2, .1);
\draw (3, -.1) -- (3, .1);
\draw (4, -.1) -- (4, .1);
\filldraw (0, 3) circle (4pt);
\filldraw (0, 4) circle (4pt);
\filldraw (0, 5) circle (4pt);
\filldraw (0, 6) circle (4pt);
\filldraw (0, 7) circle (4pt);
\filldraw (0, 8) circle (4pt);
\filldraw (0, 9) circle (4pt);
\filldraw (0, 10) circle (4pt);
\filldraw (0, 11) circle (4pt);
\filldraw (0, 12) circle (4pt);
\filldraw (0, 13) circle (4pt);
\filldraw (1, 1) circle (4pt);
\filldraw (1, 2) circle (4pt);
\filldraw (1, 3) circle (4pt);
\filldraw (1, 4) circle (4pt);
\filldraw (1, 5) circle (4pt);
\filldraw (1, 6) circle (4pt);
\filldraw (1, 7) circle (4pt);
\filldraw (1, 8) circle (4pt);
\filldraw (1, 9) circle (4pt);
\filldraw (1, 10) circle (4pt);
\node[scale = .6, color = red] at (1, 11) {$\times$};
\node[scale = .6, color = red] at (1, 12) {$\times$};
\node[scale = .6, color = red] at (1, 13) {$\times$};
\filldraw (2, 0) circle (4pt);
\filldraw (2, 1) circle (4pt);
\filldraw (2, 2) circle (4pt);
\filldraw (2, 3) circle (4pt);
\filldraw (2, 4) circle (4pt);
\filldraw (2, 5) circle (4pt);
\filldraw (2, 6) circle (4pt);
\filldraw (2, 7) circle (4pt);
\filldraw (2, 8) circle (4pt);
\filldraw (2, 9) circle (4pt);
\draw (2, 10) circle (4pt);
\filldraw (3, 0) circle (4pt);
\filldraw (3, 1) circle (4pt);
\filldraw (3, 2) circle (4pt);
\filldraw (3, 3) circle (4pt);
\filldraw (3, 4) circle (4pt);
\filldraw (3, 5) circle (4pt);
\filldraw (3, 6) circle (4pt);
\filldraw (3, 7) circle (4pt);
\filldraw (3, 8) circle (4pt);
\draw (3, 9) circle (4pt);
\draw (3, 10) circle (4pt);
\draw (3, 11) circle (4pt);
\filldraw[color = white] (4,0) circle (4pt);
\filldraw[color=blue] (4, 0) circle (4pt);
\filldraw[color=blue] (4, 1) circle (4pt);
\filldraw[color=blue] (4, 2) circle (4pt);
\filldraw[color=blue] (4, 3) circle (4pt);
\filldraw[color=blue] (4, 4) circle (4pt);
\filldraw[color=blue] (4, 5) circle (4pt);
\filldraw[color=blue] (4, 6) circle (4pt);
\filldraw[color=blue] (4, 7) circle (4pt);
\filldraw[color=white] (4, 7) circle (1.6pt);
\draw (4, 8) circle (4pt);
\draw (4, 9) circle (4pt);
\draw (4, 10) circle (4pt);
\draw (4, 11) circle (4pt);
\filldraw[color=white] (6, 0) circle (4pt);
\filldraw[color=blue] (6, 0) circle (4pt);
\filldraw[color=blue] (6, 1) circle (4pt);
\filldraw[color=blue] (6, 2) circle (4pt);
\filldraw[color=blue] (6, 3) circle (4pt);
\filldraw[color=white] (6, 3) circle (1.6pt);
\filldraw[color=blue] (6, 4) circle (4pt);
\filldraw[color=white] (6, 4) circle (1.6pt);
\filldraw[color=blue] (6, 5) circle (4pt);
\filldraw[color=white] (6, 5) circle (1.6pt);

\filldraw[color=white] (5, 0) circle (4pt);
\filldraw[color=blue] (5, 0) circle (4pt);
\filldraw[color=blue] (5, 1) circle (4pt);
\filldraw[color=blue] (5, 2) circle (4pt);
\filldraw[color=blue] (5, 3) circle (4pt);
\filldraw[color=blue] (5, 4) circle (4pt);
\filldraw[color=blue] (5, 5) circle (4pt);
\filldraw[color=blue] (5, 6) circle (4pt);
\draw (5, 7) circle (4pt);
\filldraw[color=white] (5, 5) circle (1.6pt);
\filldraw[color=white] (5, 6) circle (1.6pt);

\filldraw[color=white] (7, 0) circle (4pt);
\filldraw[color=blue] (7, 0) circle (4pt);

\node[color=white, scale = .8] at (4, -2) {Lemmas 9.12 and 9.13};
\end{tikzpicture}
\end{center}

\bibliographystyle{amsplain}
\bibliography{refs}

\providecommand{\bysame}{\leavevmode\hbox to3em{\hrulefill}\thinspace}
\providecommand{\MR}{\relax\ifhmode\unskip\space\fi MR }
\providecommand{\MRhref}[2]{%
  \href{http://www.ams.org/mathscinet-getitem?mr=#1}{#2}
}
\providecommand{\href}[2]{#2}
\begin{thebibliography}{10}

\bibitem{ArbarelloCornalbaKappa}
Enrico Arbarello and Maurizio Cornalba, \emph{Combinatorial and
  algebro-geometric cohomology classes on the moduli spaces of curves}, J.
  Algebraic Geom. \textbf{5} (1996), no.~4, 705--749. \MR{1486986}

\bibitem{ArbarelloCornalba}
\bysame, \emph{Calculating cohomology groups of moduli spaces of curves via
  algebraic geometry}, Inst. Hautes \'{E}tudes Sci. Publ. Math. (1998), no.~88,
  97--127 (1999). \MR{1733327}

\bibitem{BaeSchmitt2}
Younghan Bae and Johannes Schmitt, \emph{Chow rings of stacks of prestable
  curves {I}{I}}, arXiv preprint arXiv:2107.09192 (2021).

\bibitem{Behrend}
Kai~A. Behrend, \emph{The {L}efschetz trace formula for algebraic stacks},
  Invent. Math. \textbf{112} (1993), no.~1, 127--149. \MR{1207479}

\bibitem{Belorousski}
Pavel Belorousski, \emph{Chow rings of moduli spaces of pointed elliptic
  curves}, ProQuest LLC, Ann Arbor, MI, 1998, Thesis (Ph.D.)--The University of
  Chicago. \MR{2716762}

\bibitem{Bergstrom3}
Jonas Bergstr\"{o}m, \emph{Cohomology of moduli spaces of curves of genus three
  via point counts}, J. Reine Angew. Math. \textbf{622} (2008), 155--187.
  \MR{2433615}

\bibitem{Bergstrom2}
\bysame, \emph{Equivariant counts of points of the moduli spaces of pointed
  hyperelliptic curves}, Doc. Math. \textbf{14} (2009), 259--296. \MR{2538614}

\bibitem{BergstromTommasi}
Jonas Bergstr\"{o}m and Orsola Tommasi, \emph{The rational cohomology of
  {$\overline{\mathscr{ M}}_4$}}, Math. Ann. \textbf{338} (2007), no.~1,
  207--239. \MR{2295510}

\bibitem{BergstromFaber}
Jonas Bergström and Carel Faber, \emph{Cohomology of moduli spaces via a
  result of chenevier and lannes}, arXiv preprint arXiv:2207.05130 (2022).

\bibitem{BergstromFaberPayne}
Jonas Bergström, Carel Faber, and Sam Payne, \emph{Polynomial point counts and
  odd cohomology vanishing on moduli spaces of stable curves}, arXiv preprint
  arXiv:2206.07759 (2022).

\bibitem{part2}
Samir Canning and Hannah Larson, \emph{Chow rings of low-degree {H}urwitz
  spaces}, arXiv preprint arXiv:2110.01059 (2021).

\bibitem{789}
\bysame, \emph{The {C}how rings of the moduli spaces of curves of genus 7, 8,
  and 9}, arXiv preprint arXiv:2104.05820 (2021).

\bibitem{part1}
\bysame, \emph{Tautological classes on low-degree {H}urwitz spaces}, arXiv
  preprint arXiv:2103.09902 (2021).

\bibitem{Hyperelliptic}
\bysame, \emph{The rational {C}how rings of moduli spaces of hyperelliptic
  curves with marked points},  (2022).

\bibitem{CasnatiEkedahl}
G.~Casnati and T.~Ekedahl, \emph{Covers of algebraic varieties. {I}. {A}
  general structure theorem, covers of degree {$3,4$} and {E}nriques surfaces},
  J. Algebraic Geom. \textbf{5} (1996), no.~3, 439--460. \MR{1382731}

\bibitem{CvdGF}
Fabien Cl\'{e}ry, Carel Faber, and Gerard van~der Geer, \emph{Concomitants of
  ternary quartics and vector-valued {S}iegel and {T}eichm\"{u}ller modular
  forms of genus three}, Selecta Math. (N.S.) \textbf{26} (2020), no.~4, Paper
  No. 55, 39. \MR{4125987}

\bibitem{admcycles}
Vincent Delecroix, Johannes Schmitt, and Jason van Zelm, \emph{admcycles--a
  sage package for calculations in the tautological ring of the moduli space of
  stable curves}, arXiv preprint arXiv:2002.01709 (2020).

\bibitem{DiLorenzoPerniceVistoli}
Andrea Di~Lorenzo, Michele Pernice, and Angelo Vistoli, \emph{Stable cuspidal
  curves and the integral chow ring of $\overline{\mathscr{\M}}_{2,1}$}, arXiv
  preprint arXiv:2108.03680 (2021).

\bibitem{EdidinGraham}
Dan Edidin and William Graham, \emph{Equivariant intersection theory}, Invent.
  Math. \textbf{131} (1998), no.~3, 595--634. \MR{1614555}

\bibitem{EdidinGrahamLocalization}
\bysame, \emph{Localization in equivariant intersection theory and the {B}ott
  residue formula}, Amer. J. Math. \textbf{120} (1998), no.~3, 619--636.
  \MR{1623412}

\bibitem{FPGorenstein}
C.~Faber and R.~Pandharipande, \emph{Logarithmic series and {H}odge integrals
  in the tautological ring}, vol.~48, 2000, With an appendix by Don Zagier,
  Dedicated to William Fulton on the occasion of his 60th birthday,
  pp.~215--252. \MR{1786488}

\bibitem{FaberPandharipande}
\bysame, \emph{Relative maps and tautological classes}, J. Eur. Math. Soc.
  (JEMS) \textbf{7} (2005), no.~1, 13--49. \MR{2120989}

\bibitem{FPhandbook}
\bysame, \emph{Tautological and non-tautological cohomology of the moduli space
  of curves}, Handbook of moduli. {V}ol. {I}, Adv. Lect. Math. (ALM), vol.~24,
  Int. Press, Somerville, MA, 2013, pp.~293--330. \MR{3184167}

\bibitem{FaberPhD}
Carel Faber, \emph{{C}how rings of moduli spaces of curves}, PhD Thesis,
  Universiteit van Amsterdam (1988).

\bibitem{FaberI}
\bysame, \emph{Chow rings of moduli spaces of curves. {I}. {T}he {C}how ring of
  {$\overline{\mathscr{M}}_3$}}, Ann. of Math. (2) \textbf{132} (1990), no.~2,
  331--419. \MR{1070600}

\bibitem{FaberII}
\bysame, \emph{Chow rings of moduli spaces of curves. {II}. {S}ome results on
  the {C}how ring of {$\overline{\mathscr{M}}_4$}}, Ann. of Math. (2)
  \textbf{132} (1990), no.~3, 421--449. \MR{1078265}

\bibitem{Faberconjecture}
\bysame, \emph{A conjectural description of the tautological ring of the moduli
  space of curves}, Moduli of curves and abelian varieties, Aspects Math., E33,
  Friedr. Vieweg, Braunschweig, 1999, pp.~109--129. \MR{1722541}

\bibitem{fabernonvanishing}
Carel~F Faber, \emph{A non-vanishing result for the tautological ring of
  $\mathcal{M}_g$}, arXiv preprint math/9711219 (1997).

\bibitem{GetzlerLooijenga}
Ezra Getzler and Eduard Looijenga, \emph{The hodge polynomial
  of$\backslash$mbar\_ $\{$3, 1$\}$}, arXiv preprint math/9910174 (1999).

\bibitem{GraberPandharipande}
T.~Graber and R.~Pandharipande, \emph{Constructions of nontautological classes
  on moduli spaces of curves}, Michigan Math. J. \textbf{51} (2003), no.~1,
  93--109. \MR{1960923}

\bibitem{GraberVakil}
Tom Graber and Ravi Vakil, \emph{On the tautological ring of
  {$\overline{\mathscr {M}}_{g,n}$}}, Turkish J. Math. \textbf{25} (2001),
  no.~1, 237--243. \MR{1829089}

\bibitem{Izadi}
E.~Izadi, \emph{The {C}how ring of the moduli space of curves of genus {$5$}},
  The moduli space of curves ({T}exel {I}sland, 1994), Progr. Math., vol. 129,
  Birkh\"{a}user Boston, Boston, MA, 1995, pp.~267--304. \MR{1363060}

\bibitem{Janda}
Felix Janda, \emph{Relations on {$\overline M_{g,n}$} via equivariant
  {G}romov-{W}itten theory of {$\mathbb {P}^1$}}, Algebr. Geom. \textbf{4}
  (2017), no.~3, 311--336. \MR{3652083}

\bibitem{Jannsen}
Uwe Jannsen, \emph{Motivic sheaves and filtrations on {C}how groups}, Motives
  ({S}eattle, {WA}, 1991), Proc. Sympos. Pure Math., vol.~55, Amer. Math. Soc.,
  Providence, RI, 1994, pp.~245--302. \MR{1265533}

\bibitem{Keel}
Sean Keel, \emph{Intersection theory of moduli space of stable {$n$}-pointed
  curves of genus zero}, Trans. Amer. Math. Soc. \textbf{330} (1992), no.~2,
  545--574. \MR{1034665}

\bibitem{Kresch}
Andrew Kresch, \emph{Cycle groups for {A}rtin stacks}, Invent. Math.
  \textbf{138} (1999), no.~3, 495--536. \MR{1719823}

\bibitem{Aaron}
Aaron Landesman, \emph{The locus of plane curves in the moduli stack of
  curves}, preprint.

\bibitem{EricLarson}
Eric Larson, \emph{The integral {C}how ring of {$\overline M_2$}}, Algebr.
  Geom. \textbf{8} (2021), no.~3, 286--318. \MR{4206438}

\bibitem{Looijenga}
Eduard Looijenga, \emph{On the tautological ring of {${\mathscr M}_g$}},
  Invent. Math. \textbf{121} (1995), no.~2, 411--419. \MR{1346214}

\bibitem{Mumford}
David Mumford, \emph{Towards an enumerative geometry of the moduli space of
  curves}, Arithmetic and geometry, {V}ol. {II}, Progr. Math., vol.~36,
  Birkh\"{a}user Boston, Boston, MA, 1983, pp.~271--328. \MR{717614}

\bibitem{calculus}
Rahul Pandharipande, \emph{A calculus for the moduli space of curves},
  Algebraic geometry: {S}alt {L}ake {C}ity 2015, Proc. Sympos. Pure Math.,
  vol.~97, Amer. Math. Soc., Providence, RI, 2018, pp.~459--487. \MR{3821159}

\bibitem{PandharipandePixtonZvonkine}
Rahul Pandharipande, Aaron Pixton, and Dimitri Zvonkine, \emph{Relations on
  {$\overline{\mathscr{M}}_{g,n}$} via {$3$}-spin structures}, J. Amer. Math.
  Soc. \textbf{28} (2015), no.~1, 279--309. \MR{3264769}

\bibitem{zerocycles}
Rahul Pandharipande and Johannes Schmitt, \emph{Zero cycles on the moduli space
  of curves}, \'{E}pijournal G\'{e}om. Alg\'{e}brique \textbf{4} (2020), Art.
  12, 26. \MR{4149970}

\bibitem{PenevVakil}
Nikola Penev and Ravi Vakil, \emph{The {C}how ring of the moduli space of
  curves of genus six}, Algebr. Geom. \textbf{2} (2015), no.~1, 123--136.
  \MR{3322200}

\bibitem{Petersengenus1}
Dan Petersen, \emph{The structure of the tautological ring in genus one}, Duke
  Math. J. \textbf{163} (2014), no.~4, 777--793. \MR{3178432}

\bibitem{Petersen}
\bysame, \emph{Tautological rings of spaces of pointed genus two curves of
  compact type}, Compos. Math. \textbf{152} (2016), no.~7, 1398--1420.
  \MR{3530445}

\bibitem{PetersenTommasi}
Dan Petersen and Orsola Tommasi, \emph{The {G}orenstein conjecture fails for
  the tautological ring of {$\overline{\mathcal{M}}_{2,n}$}}, Invent. Math.
  \textbf{196} (2014), no.~1, 139--161. \MR{3179574}

\bibitem{Pikaart}
Martin Pikaart, \emph{An orbifold partition of {$\overline M{}^n_g$}}, The
  moduli space of curves ({T}exel {I}sland, 1994), Progr. Math., vol. 129,
  Birkh\"{a}user Boston, Boston, MA, 1995, pp.~467--482. \MR{1363067}

\bibitem{Pixton}
Aaron Pixton, \emph{The tautological ring of the moduli space of curves},
  ProQuest LLC, Ann Arbor, MI, 2013, Thesis (Ph.D.)--Princeton University.
  \MR{3153424}

\bibitem{Roitman}
A.~A. Ro\u{\i}tman, \emph{Rational equivalence of zero-dimensional cycles},
  Mat. Sb. (N.S.) \textbf{89(131)} (1972), 569--585, 671. \MR{0327767}

\bibitem{SchmittvanZelm}
Johannes Schmitt and Jason van Zelm, \emph{Intersections of loci of admissible
  covers with tautological classes}, Selecta Math. (N.S.) \textbf{26} (2020),
  no.~5, Paper No. 79, 69. \MR{4177576}

\bibitem{TotaroCKgP}
Burt Totaro, \emph{The motive of a classifying space}, Geom. Topol. \textbf{20}
  (2016), no.~4, 2079--2133. \MR{3548464}

\bibitem{BogaartEdixhoven}
Theo van~den Bogaart and Bas Edixhoven, \emph{Algebraic stacks whose number of
  points over finite fields is a polynomial}, arXiv preprint math/0505178.pdf
  (2008).

\bibitem{vanZelm}
Jason van Zelm, \emph{Nontautological bielliptic cycles}, Pacific J. Math.
  \textbf{294} (2018), no.~2, 495--504. \MR{3770123}

\bibitem{Vistoli}
Angelo Vistoli, \emph{Intersection theory on algebraic stacks and on their
  moduli spaces.}, Invent. Math. \textbf{96} (1989), 613--670.

\end{thebibliography}
\end{document}